\documentclass[a4paper,12pt,reqno]{amsart}
\usepackage[english]{babel}
\usepackage{hyperref}\hypersetup{colorlinks=true, citecolor=blue}
\usepackage{xfrac, mathrsfs, graphicx, esint}

\numberwithin{equation}{section}

\newtheorem{theorem}{Theorem}[section]
\newtheorem{corollary}[theorem]{Corollary}
\newtheorem{lemma}[theorem]{Lemma}
\newtheorem{proposition}[theorem]{Proposition}
\newtheorem{remark}[theorem]{Remark}

\newtheorem{conge}[theorem]{Conjecture}
\newtheorem{definition}[theorem]{Definition}

\def \R {{\mathbb {R}}}
\def \N {{\mathbb N}}

\def \Z {{\mathbb Z}}

\newcommand\MS{{\rm MS}}
\newcommand\MM{\mathcal{M}}
\newcommand\HH{\mathcal{H}}
\newcommand\LL{\mathcal{L}}

\newcommand\Om{\Omega}

\newcommand{\erho}{e(\rho)}
\newcommand{\lrho}{\ell(\rho)}

\newcommand{\e}{\varepsilon}

\newcommand{\divv}{{\rm div}}
\newcommand{\diam}{{\rm diam}}
\def\Per{\mathrm{Per}}

\newcommand\res{\mathop{\hbox{\vrule height 7pt width .3pt depth 0pt
\vrule height .3pt width 5pt depth 0pt}}\nolimits}

\newcommand{\dimh}{\mathrm{dim}_{\HH}}
\newcommand{\FF}{\mathcal{F}}
\newcommand{\GG}{\mathcal{G}}

\newcommand{\CC}{\mathcal{C}}
\newcommand{\SSu}{\overline{S_u}}
\newcommand{\EE}{\mathcal{E}}

\newcommand\eps{\varepsilon}
\newcommand\Bb{\overline{B}}

\title[Fine regularity results for Mumford-Shah minimizers]{Fine regularity results 
for Mumford-Shah minimizers: porosity, higher integrability and the Mumford-Shah conjecture}
\author{Matteo Focardi}
\address{DiMaI ``U. Dini''\\
V.le Morgagni 67/a -- I-50134 -- Firenze}
\email{focardi@math.unifi.it}

\begin{document}

\begin{abstract}
{We review some classical results and more recent insights about the regularity theory for local minimizers 
of the Mumford and Shah energy and their connections with the Mumford and Shah conjecture. 
We discuss in details the links among the latter, the porosity of the jump set and the higher integrability 
of the approximate gradient. In particular, higher integrability turns out to be related with an explicit estimate 
on the Hausdorff dimension of the singular set and an energetic characterization of the conjecture itself.}


\end{abstract}
\maketitle
\section{Introduction}

The Mumford and Shah model is a prominent example of variational problem in image segmentation 
(see \cite{MS89}). It is an algorithm able to detect the boundaries of the contours of the 
objects in a black and white digitized image. Representing the latter by a greyscale function 
$g\in L^\infty(\Om,[0,1])$, a smoothed version of the original image is then obtained by minimizing 
the functional 
\begin{equation}\label{e:enrgcomplete}
(v,K)\to\mathscr{F}(v,K,\Om)+\gamma\int_{\Om\setminus K}|v-g|^2dx,
\end{equation}
with
\begin{equation}\label{e:F}
\mathscr{F}(v,K,\Om):=\int_{\Om\setminus K}|\nabla v|^2\,dx+\beta\,\HH^{1}(K),
\end{equation}
where $\Om\subseteq\R^2$ is an open set, $K$ is a relatively closed subset of $\Om$
with finite $\HH^{1}$ measure, $v\in C^1(\Om\setminus K)$, $\beta$ and $\gamma$ 
are nonnegative parameters to be tuned suitably according to the applications.
In our discussion we can set $\beta=1$ without loss of generality.

The role of the squared $L^2$ distance in \eqref{e:enrgcomplete} is that 
of a fidelity term in order that the output of the process is close in 
an average sense to the original input image $g$. The set $K$ 
represents the set of contours of the objects in the image, the 
length of which is kept controlled by the penalization of its $\HH^{1}$
measure to avoid over segmentation, while the Dirichlet energy of $v$ favors 
sharp contours rather than zones where a thin layer of gray is used to pass 
smoothly from white to black or vice versa. 

We stress the attention upon the fact that the set $K$ is not assigned 
a priori and it is not a boundary in general. Therefore, this problem 
is not a free boundary problem, and new ideas and techniques had to be 
developed to solve it.
Since its appearance in the late 80's to today the research on the 
Mumford and Shah problem, and on related fields, has been very active and  
different approaches have been developed. In this notes we shall focus mainly 
on that proposed by De Giorgi and Ambrosio. 
This is only due to a matter of taste of the Author and it is also dictated by understandable reasons of space.
Even more, it is not possible to be exhaustive in our (short) presentation, therefore 
we refer to the books by Ambrosio, Fusco and Pallara \cite{AFP00} and David \cite{Dav05} for the proofs of many results we shall only quote, 
for a more detailed account of the several contributions in literature, for the many connections with other 
fields and for complete lists of references (see also the recent survey \cite{Lem15} that covers several parts 
of the regularity theory that are not presented here).

Going back to the Mumford and Shah minimization problem and trying to follow the path of the Direct Method of the 
Calculus of Variations, it is clear that a weak formulation calls for a function space allowing 
for discontinuities of co-dimension $1$ in which an existence theory can be established. 
Therefore, by taking into account the structure of the energy,  De Giorgi and Ambrosio were led to 
consider the space $SBV$ of \emph{Special functions of Bounded Variation}, i.e.~the subspace of $BV$ functions with 
singular part of the distributional derivative concentrated on a $1$-dimensional 
set called in what follows the \emph{jump set} (throughout the paper we will use standard notations and results 
concerning the spaces $BV$ and $SBV$, following the book \cite{AFP00}). 

The purpose of the present set of notes is basically to resume and collect several of the regularity properties 
known at present for Mumford and Shah minimizers. More precisely, Section~\ref{s:basic} is devoted to recalling
basic facts about the functional setting of the problem and its weak formulation. The celebrated De Giorgi, Carriero 
and Leaci \cite{DGCL89} regularity result implying the equivalence between the strong and 
weak formulations, is discussed in details. In subsection~\ref{ss:dlb} we provide a recent proof by De Lellis and Focardi 
valid in the $2$d case that gives an explicit constant in the density lower bound, and in subsection~\ref{ss:monotonicity}
we discuss the almost monotonicity formula by Bucur and Luckhaus.
Next, we state the Mumford and Shah conjecture. The understanding of 
such a claim is the goal at which researchers involved in this problem are striving for.
In this perspective well-established and more recent fine regularity results on the jump set of minimizers are discussed 
in Section~\ref{s:regSu}. Furthermore, we highlight two different paths that might lead to the solution in positive of 
the Mumford and Shah conjecture: the complete characterization of blow~ups in subsection~\ref{ss:bup} and a sharp 
higher integrability of the (approximate) gradient in Theorem~\ref{t:MSconj} together with the uniqueness of blow~up limits.
In particular, we discuss in details the latter by following the ideas introduced by Ambrosio, Fusco and Hutchinson 
\cite{AFH} linking higher integrability of the gradient of a minimizer with the size of the \emph{singular set} of
the minimizer itself, i.e.~the 
subset of points of the jump set having no neighborhood in which the jump set itself is a regular curve.
 An explicit estimate shows that the bigger the integrability exponent of the gradient is, the lower the Hausdorff 
dimension of the singular set is (cf. Theorem~\ref{t:AFH1}). Pushing forward this approach, an energetic characterization 
of a slightly weaker form of the Mumford and Shah conjecture can be found beyond the scale of $L^p$ spaces 
(cf. Theorem~\ref{t:MSconj}).
In particular, the quoted estimate on the Hausdorff dimension of the full singular set reduces to the higher 
integrability property of the gradient and a corresponding estimate on a special subset of singular points: those for 
which the scaled Dirichlet energy is infinitesimal.
The latter topic is dealt with in full details in Section~\ref{s:cpt} in the setting of Caccioppoli partitions
as done by De Lellis and Focardi in \cite{DF1}. The analysis of Section~\ref{s:cpt} allowed the same Authors to 
prove the higher integrability property in $2$-dimensions as explained in Section~\ref{s:hi2}.
A different path leading to higher integrability in any dimension is to exploit the porosity of the 
jump set. This approach, due to De Philippis and Figalli \cite{DePFig14}, is the object of Section~\ref{s:hi}. 
Some preliminaries on porous sets are discussed in Section~\ref{s:porosity}. 

To conclude this introduction it is worth mentioning that the Mumford and Shah energy and the theory developed 
in order to study it, have been employed in many other fields. The applications to Fracture Mechanics, both 
in a static setting and for quasi-static irreversible crack-growth for brittle materials according to Griffith 
are important instances of that (see in particular \cite{BouFraMar}, \cite[Section 4.6.6]{AFP00} and \cite{Br06},
\cite{DMFT05}, \cite{Mie05}). It is also valuable to recall that several contributions in literature are devoted 
to the asymptotic analysis or the variational approximation of free discontinuity energies by means of 
De Giorgi's \emph{$\Gamma$-convergence} theory. We refer to the books by Braides \cite{Br98,Br02,Br06} 
for the analysis of several interesting problems arising from models in different fields (for a quick introduction 
to $\Gamma$-convergence see \cite{Foc}, for a more detailed account consult the treatise \cite{DM}).

The occasion to write this set of notes stems from the course ``\emph{Fine regularity results for Mumford-Shah 
minimizers: higher integrability of the gradient and estimates on the Hausdorff dimension of the singular set}'' 
taught by the Author in July 2014 at Centro De Giorgi in Pisa 
within the activities of the ``School on Free Discontinuity problems'', ERC Research Period on Calculus 
of Variations and Analysis in Metric Spaces. 
The material collected here covers entirely the six lectures of the course, additional topics and some more 
recent insights are also included for the sake of completeness and clarity.
It is a pleasure to acknowledge the hospitality of Centro De Giorgi and to gratefully thank N.~Fusco and A.~Pratelli, 
the organizers of the school, for their kind invitation. Let me also thank all the people in the audience for 
their attention, patience, comments and questions. In particular, the kind help of R.~Cristoferi and 
E.~Radici who read a preliminary version of these notes is acknowledged. Nevertheless, the Author 
is the solely responsible for all the inaccuracies contained in them.

\section{Existence Theory and first Regularity results}\label{s:basic}

In this section we shall overview the first basic issues of the problem. More 
generally we discuss the $n$-dimensional case, though we shall often make
specific comments related to the $2$-dimensional setting of the original 
problem (and sometimes to the $3$d case as well). We shall 
freely use the notation for $BV$ functions and Caccioppoli sets adopted 
in the book by Ambrosio, Fusco and Pallara \cite{AFP00}. We shall always 
refer to it also for the many results that we shall apply or even only 
quote without giving a precise citation.

\subsection{Functional setting of the problem}\label{ss:fsetting}

A function $v\in L^1(\Om)$ belongs to $BV(\Om)$ if and only if $Dv$ is a 
(vector-valued) Radon measure on the non empty open subset $\Om$ of $\R^n$. 
The distributional derivative of $v$ can be decomposed according to  
\[
Dv=\nabla v\,\LL^n\res\Omega+(v^+-v^-)\nu_v\,\,\HH^{n-1}\res S_v+D^cv,
\] 
where
\begin{itemize}

\item[(i)] $\nabla v$ is the density of the absolutely continuous part 
of $Dv$ with respect to $\LL^n\res\Om$ (and the \emph{approximate 
gradient} of $v$ in the sense of Geometric Measure Theory as well); 

\item[(ii)] $S_v$ is the set of \emph{approximate discontinuities} of 
$v$, an $\HH^{n-1}$-rectifiable set (so that $\LL^n(S_v)=0$) endowed 
with approximate normal $\nu_v$ for $\HH^{n-1}$ a.e. on $S_v$; 

\item[(iii)] $v^\pm$ are the \emph{approximate one-sided traces} left by $v$ $\HH^{n-1}$ a.e. on $S_v$;

\item[(iv)] $D^cv$ is the rest in the Radon-Nikodym decomposition of the singular part of $Dv$ 
after the absolutely continuous part with respect to $\HH^{n-1}\res S_v$ has been identified.
Thus, it is a singular measure both with respect to $\LL^n\res\Om$ and to $\HH^{n-1}\res S_v$
(for more details see \cite[Proposition~3.92]{AFP00}).

\end{itemize}

By taking into account the structure of the energy in \eqref{e:enrgcomplete}, only volume 
and surface contributions are penalized, so that it is natural to introduce the following 
subspace of $BV$. 
\begin{definition}[\cite{DGA}, Section 4.1~\cite{AFP00}]
$v\in BV(\Om)$ is a \emph{Special function of Bounded Variation}, 
in short \emph{$v\in SBV(\Om)$}, if $D^cv=0$, i.e.~$Dv=\nabla v\,\LL^n\res\Omega+(v^+-v^-)\nu_v\,\,\HH^{n-1}\res S_v$.
\end{definition}
No Cantor staircase type behavior is allowed for these functions. 
Simple examples are collected in the ensuing list:
\begin{itemize}
\item[(i)] if $n=1$ and $\Om=(\alpha,\beta)$, $SBV\big((\alpha,\beta)\big)$ is easily described in view of the well 
known decomposition of $BV$ functions of one variable. Indeed, any function in $SBV\big((\alpha,\beta)\big)$ is the 
sum of a $W^{1,1}\big((\alpha,\beta)\big)$ function with one of pure jump, i.e. $\sum_{i\in\N}a_i\chi_{(\alpha_i,\alpha_{i+1})}$, 
with $\alpha=\alpha_0$, $\alpha_i<\alpha_{i+1}<\beta$, $(a_i)_{i\in I}\in\ell^\infty$;

\item[(ii)]  $W^{1,1}(\Om)\subset SBV(\Om)$. Clearly, $Dv=\nabla v\,\LL^n\res\Omega$. 
In this case $\nabla v$ coincides with the usual distributional gradient;

\item[(iii)] let $(E_i)_{i\in I}$, $I\subseteq\N$ , be a \emph{Caccioppoli partition} of $\Om$, 
i.e.~$\LL^n\big(\Om\setminus\cup_iE_i\big)=0$ and $\LL^n(E_i\cap E_j)=0$ if $i\neq j$,
with the $E_i$'s sets of finite perimeter such that
\[
\sum_{i\in I}\Per(E_i)<\infty.
\]
Then, $v=\sum_{i\in I} a_i\,\chi_{E_i}\in SBV(\Om)$ if 
$(a_i)_{i\in I}\in\ell^\infty$. In this case, if $J_{\mathscr{E}}:=\cup_i\partial^\ast E_i$
denotes the \emph{set of interfaces} of $\mathscr{E}$, with $\partial^\ast E_i$ the \emph{essential boundary} 
of $E_i$, then $\HH^{n-1}(S_v\setminus J_{\mathscr{E}})=0$ and
\[
Dv=(v^+-v^-)\nu_v\,\,\HH^{n-1}\res J_{\mathscr{E}}.
\]
Functions of this type have zero approximate gradient, they are called \emph{piecewise constant} 
and form a subspace denoted by $SBV_0(\Omega)$ (cf. \cite[Theorem~4.23]{AFP00});

\item[(iv)] the function $v(\rho,\theta):=\sqrt{\rho}\cdot\sin(\sfrac\theta2)$ for 
$\theta\in(-\pi,\pi)$ and $\rho>0$ is in $SBV(B_r)$ for all $r>0$. 
In particular, $v\in SBV(B_r)\setminus\big(W^{1,1}(B_r)\oplus SBV_0(B_r)\big)$. 

\end{itemize}
A general receipt to construct interesting examples of $SBV$ functions can be obtained as follows 
(see \cite[Proposition 4.4]{AFP00}). 
\begin{proposition}\label{p:weak}
If $K\subset\Omega$ is a closed set such that $\HH^{n-1}(K)<+\infty$ and 
$v\in W^{1,1}\cap L^\infty(\Om\setminus K)$, then $v\in SBV(\Om)$ and 
\begin{equation}\label{e:ex}
\HH^{n-1}(S_v\setminus K)=0.
\end{equation}
\end{proposition}
Clearly, property \eqref{e:ex} above is not valid for a generic
member of $SBV$, but it does for a significant class of functions:
local minimizers of the energy under consideration (see below for the definition), 
actually satisfying even a stronger property (cf. Proposition~\ref{p:strongweak}).

\subsection{Tonelli's Direct Method and Weak formulation}

The difficulty in applying the Direct Method is related to the surface term for which 
it is hard to find a topology ensuring at the same time lower semicontinuity and 
pre-compactness for minimizing sequences. Using the Hausdorff local topology requires
a very delicate study of the latter ones to rule out typical counterexamples 
as shown by Maddalena and Solimini in \cite{MadSol01b}. Here, we shall follow instead 
the original approach by De Giorgi and Ambrosio \cite{DGA}.

Keeping in mind the example in Proposition~\ref{p:weak}, the weak formulation of 
the problem under study is obtained naively by taking $K=S_v$. 
Loosely speaking in this approach the set of contours $K$ is identified 
by the (Borel) set $S_v$ of (approximate) discontinuities of the function 
$v$ that is not fixed a priori. This is the reason for the terminology 
\emph{free discontinuity} problem coined by De Giorgi.
The (weak counterpart of the) Mumford and Shah energy $\mathscr{F}$ in \eqref{e:F} 
of a function $v$ in $SBV(\Om)$ on an open subset $A\subseteq\Om$ then reads as 
\begin{equation}\label{e:E}
\mathscr{F}(v,A)=\MS(v,A)+\gamma\int_A|v-g|^2dx,
\end{equation}
where
\begin{equation}\label{e:ms}
\MS(v,A):=\int_A|\nabla v|^2dx+\HH^{n-1}(S_v\cap A).
\end{equation}
For the sake of simplicity in case $A=\Om$ we drop the dependence on the set of integration.

In passing, we note that, the class $\{v\in BV(\Om):\, Dv=D^cv\}$ of Cantor type functions 
is dense in $BV$ w.r.to the $L^1$ topology, thus it is easy to infer that 
\[
 \inf_{BV(\Om)}\mathscr{F}=0,
\]
so that the restriction to $SBV$ is needed in order not to trivialize the problem.

Ambrosio's $SBV$ closure and compactness theorem (see 
\cite[Theorems 4.7 and 4.8]{AFP00}) ensures the existence of a
minimizer of $\mathscr{F}$ on $SBV$. 
\begin{theorem}[Ambrosio~\cite{Amb89}]\label{t:luigi}
 Let $(v_j)_j\subset SBV(\Om)$ be such that
 \[
  \sup_j\big(\MS(v_j)+\|v_j\|_{L^\infty(\Om)}\big)<\infty,
 \]
then there exists a subsequence $(v_{j_k})_k$ and a function $v\in SBV(\Om)$
such that $v_{j_k}\to v$ $L^p(\Om)$, for all $p\in[1,\infty)$.

Moreover, we have the separated lower semicontinuity estimates
\begin{equation}\label{e:scivol}
\int_\Om|\nabla v|^2dx\leq \liminf_k\int_\Om|\nabla v_{j_k}|^2dx
\end{equation}
and
\begin{equation}\label{e:scisurf}
\HH^{n-1}(S_v) \leq\liminf_k\HH^{n-1}(S_{v_{j_k}}).
\end{equation}
\end{theorem}
Ambrosio's theorem is the natural counterpart of Rellich-Kondrakov theorem in 
Sobolev spaces. Indeed, for Sobolev functions, it reduces essentially to that statement 
provided that an $L^p$ rather than an $L^\infty$ bound is assumed. More generally, 
Ambrosio's theorem holds true in the bigger space $GSBV$. 
In particular, \eqref{e:scivol} and \eqref{e:scisurf} display a separate lower
semicontinuity property for the two terms of the energy in a way that the two 
terms cannot combine to create neither a contribution for the other nor a Cantor 
type one.

By means of the chain rule formula for $BV$ functions one can prove that 
the functional under consideration is decreasing under truncation, i.e.~for all $k\in\N$
\[
 \mathscr{F}(\tau_k(v))\leq \mathscr{F}(v)\quad \forall v\in SBV(\Om),
\]
if $\tau_k(v):=(v\wedge k)\vee(-k)$. 

Therefore, being $g\in L^\infty(\Omega)$, we can always restrict ourselves to 
minimize it over the ball in $L^\infty(\Omega)$ of radius $\|g\|_{L^\infty(\Omega)}$.
In conclusion, Theorem~\ref{t:luigi} always provides the existence of a (global) minimizer
for the weak formulation of the problem.

Once the existence has been checked, necessary conditions satisfied by minimizers
are deduced. Supposing $g\in C^1(\Om)$, by means of internal variations, i.e.~constructing competitors to 
test the minimality of $u$ by composition with diffeomorphisms of $\Omega$ arbitrarily 
close to the identity of the type $\mathrm{Id}+\eps\,\phi$, the Euler-Lagrange equation 
takes the form 
\begin{multline}\label{e:EL}
\int_{\Om\setminus S_u}\Big(\big(|\nabla u|^2+\gamma(u-g)^2\big)\mathrm{div}\phi
-2\langle\nabla u,\nabla u\cdot\nabla\phi\rangle-2\gamma\,(u-g)\langle\nabla g,\phi\rangle\Big)dx\\
+\int_{S_u}\mathrm{div}^{S_u}\phi\,d\HH^{n-1}=0
\end{multline}
for all $\phi\in C^1_c(\Om,\R^n)$, $\mathrm{div}^{S_u}\phi$ denoting the tangential divergence 
of the field $\phi$ on $S_u$ (cf. \cite[Theorem~7.35]{AFP00}).

Instead, by using outer variations, i.e.~range perturbations of the type $u+\eps (v-u)$ for $v\in SBV(\Om)$ 
such that $\mathrm{spt}(u-v)\subset\hskip-0.125cm\subset\Om$ and $S_v\subseteq S_u$,
we find 
\begin{equation}\label{e:ELo}
 \int_{\Omega}\big(\langle\nabla u,\nabla(v-u)\rangle+\gamma\,(u-g)\,(v-u)\big)\,dx=0.
 \end{equation}

\subsection{Back to the strong formulation: the density lower bound}\label{ss:dlb}

Existence of minimizers for the strong formulation of the problem is 
obtained via a regularity property enjoyed by (the jump set of) the 
minimizers of the weak counterpart. The results obtained in this framework 
will be instrumental also to establish way much finer regularity properties 
in the ensuing sections.

We start off analyzing the scaling of the energy in order to understand the 
local behavior of minimizers. This operation has to be done with some care 
since the volume and length terms in $\MS$ scale differently under affine 
change of variables of the domain. Let $v\in SBV(B_\rho(x))$, set
\begin{equation}\label{e:blowup}
v_{x,\rho}(y):=\rho^{-1/2}v(x+\rho\,y),
\end{equation}
then $v_{x,\rho}\in SBV(B_1)$, with  
\[
\MS(v_{x,\rho},B_1)=\rho^{1-n}\MS(v,B_\rho(x))
\]
and
\[
\int_{B_1}|v_{x,\rho}-g_{x,\rho}|^2dz=\rho^{-1-n}\int_{B_\rho(x)}|v-g|^2dy.
\]
Thus, 
\[
{\rho^{1-n}}\Big(\MS(v,B_\rho(x))+\int_{B_\rho(x)}|v-g|^2dz\Big)=
\MS(v_{x,\rho},B_1)+\rho^2\int_{B_1}|v_{x,\rho}-g_{x,\rho}|^2dy.
\]
By taking into account that $g\in L^\infty$ and that along the minimization 
process we are actually interested only in functions satisfying the bound
$\|v\|_{L^\infty(\Om)}\leq \|g\|_{L^\infty(\Om)}$, we get 
\[
 \rho^2\int_{B_1}|v_{x,\rho}-g_{x,\rho}|^2dy\leq 2\rho\,\|g\|^2_{L^\infty(\Om)}=O(\rho)\qquad\rho\downarrow 0.
\]
This calculation shows that, at the first order, the leading term in the energy $\mathscr{F}$ computed 
on $B_\rho(x)$ is that related to the $\MS$ functional, the other being a contribution of 
higher order that can be neglected in a preliminary analysis. 
 
Motivated by this, we introduce a notion of minimality involving only the
leading part of the energy. This corresponds to setting $\gamma=0$ in the definition 
of $\mathscr{F}$ (cf. \eqref{e:E}).
\begin{definition}
A function $u\in SBV(\Om)$ with $\MS(u)<\infty$\footnote{The finite energy condition is actually not needed
due to the local character of the notion introduced, it is assumed only for the sake of simplicity.} 
is a local minimizer of $\MS$ if
\[
\MS(u)\leq \MS(v)\quad \text{ whenever } \{v\neq u\}\subset\hskip-0.125cm\subset\Om.
\]
\end{definition}
In what follows, $u$ will always denote a local minimizer of $\MS$ unless otherwise stated, 
and the class of all local minimizers shall be denoted by $\MM(\Om)$.
Actually, we shall often refer to local minimizers simply as 
minimizers if no confusion can arise.
In particular, regularity properties for minimizers of the whole energy can be obtained by 
perturbing the theory developed for local minimizers
(see for instance Corollary~\ref{c:qmin} and Theorem~\ref{t:monform} below).

Harmonic functions with small oscillation are minimizers as a simple consequence of \eqref{e:ELo}. 
\begin{proposition}[Chambolle, see Proposition~6.8 \cite{AFP00}]
 If $u$ is harmonic in $\Omega^\prime$, then $u\in\MM(\Omega)$, for all 
 $\Omega\subset\hskip-0.125cm\subset\Omega^\prime$, provided
 \begin{equation}\label{e:harmMM}
 \big(\sup_\Om u-\inf_\Om u\big)\|\nabla u\|_{L^\infty(\Omega)}\leq 1. 
 \end{equation}
\end{proposition}
\begin{proof}
Let $A\subset\hskip-0.125cm\subset\Om$. By Theorem~\ref{t:luigi} it is easy to show 
the existence of a minimizer $w\in SBV(\Om)$ of the Dirichlet problem
$\min\left\{\MS(v):\, v\in SBV(\Om),\, v=u\text{ on $\Om\setminus A$}\right\}$.
Moreover, by truncation $\inf_\Om u\leq w\leq\sup_\Om u$ $\LL^n$ a.e. on $\Om$.

By the arbitrariness of $A$, the local minimality of $u$ follows provided we show that 
$\MS(u,\Om)\leq\MS(w,\Om)$.
To this aim, we use the Euler-Lagrange condition \eqref{e:ELo} with $\gamma=0$, namely                                                                    
\[
 \int_{\Omega}\langle\nabla w,\nabla(u-w)\rangle\,dx=0\Longleftrightarrow
 \int_{\Omega}|\nabla w|^2\,dx=\int_{\Omega}\langle\nabla w,\nabla u\rangle\,dx,
 \]
to get
\begin{multline*}
 \MS(u,\Om)\leq\MS(w,\Om)\Longleftrightarrow
 \int_\Om\langle\nabla u,\nabla(u-w)\rangle\,dx\leq\HH^{n-1}(S_w)\\
 \Longleftrightarrow\int_\Om\nabla u\cdot\,dD(u-w)-\int_{S_w}\langle\nabla u,\nu_w\rangle(w^+-w^-)d\HH^{n-1}
 \leq\HH^{n-1}(S_w).
 \end{multline*}
An integration by parts, the harmonicity of $u$ and the equality $w=u$ on $\Om\setminus A$ give
\[
 \int_\Om\nabla u\cdot\,dD(u-w)=-\int_\Om (u-w)\triangle u\,dx=0,
\]
and therefore 
\[
 \MS(u,\Om)\leq\MS(w,\Om)\Longleftrightarrow
 -\int_{S_w}\langle\nabla u,\nu_w\rangle(w^+-w^-)d\HH^{n-1}
 \leq\HH^{n-1}(S_w).
\]
The conclusion follows from condition \eqref{e:harmMM} as $\inf_\Om u\leq w\leq\sup_\Om u$ $\LL^n$ 
a.e. on $\Om$.
\end{proof}
By means of the slicing theory in $SBV$, i.e.~the characterization of $SBV$ via restrictions to lines, one can also prove that
\emph{pure jumps}, i.e.~functions as
\begin{equation}\label{e:purejumps}
a\chi_{\{\langle x-x_o,\nu\rangle>0\}}+b\chi_{\{\langle x-x_o,\nu\rangle<0\}}
\end{equation}
for $a$ and $b\in\R$ and $\nu\in\mathbb{S}^{n-1}$, are local minimizers as well 
(cf. \cite[Proposition~6.8]{AFP00}]). Further examples shall be discussed in what follows
(cf. subsection~\ref{ss:MSconj}).

As established in \cite{DGCL89} in all dimensions (and proved alternatively 
in \cite{DMMS92} and \cite{Dav96} in dimension two), if $u\in\MM(\Om)$ then 
the pair $(u,\Om\cap\SSu)$ is a minimizer of $\mathscr{F}$ for $\gamma=0$.
The main point is the identity $\HH^{n-1}(\Om\cap\SSu\setminus S_u)=0$, 
which holds for every $u\in \MM (\Omega)$. 
The groundbreaking paper \cite{DGCL89} proves this identity via the following 
\emph{density lower bound} estimate (see \cite[Theorem 7.21]{AFP00}).
\begin{theorem}[De Giorgi, Carriero and Leaci \cite{DGCL89}]\label{t:DGCL}
There exist dimensional constants $\theta,\,\varrho>0$ such that for every $u\in\MM(\Om)$
\begin{equation}\label{e:dgcl0}
\MS(u, B_r(z))\geq\theta\,r^{n-1} 
\end{equation}
for all $z\in\Om\cap\SSu$, and all $r\in(0,\varrho\wedge\mathrm{dist}(z,\partial\Om))$. 
\end{theorem}
Building upon the same ideas, in \cite{CL90} it is proved a slightly more precise result 
(see again \cite[Theorem 7.21]{AFP00}).
\begin{theorem}[Carriero and Leaci \cite{CL90}]\label{t:CL}
There exists a dimensional constant $\theta_0,\,\varrho_0>0$ such that for every $u\in\MM(\Om)$
\begin{equation}\label{e:dgcl}
\HH^{n-1}(S_u\cap B_r(z))\geq\theta_0\,r^{n-1} 
\end{equation}
for all $z\in \Om\cap\SSu$, and all $r\in(0,\varrho_0\wedge\mathrm{dist}(z,\partial\Omega))$.
\end{theorem}
In particular, from the latter we infer the so called \emph{elimination property} for $\Om\cap\SSu$, 
i.e.~if $\HH^{n-1}(S_u\cap B_r(z))<\frac{\theta_0}{2^{n-1}}\,r^{n-1}$
then actually $\overline{S_u}\cap B_{\sfrac r2}(z)=\emptyset$.

Given Theorem~\ref{t:DGCL} or \ref{t:CL} for granted we can easily prove
the equivalence of the strong and weak formulation of the problem by means of the ensuing 
\emph{density estimates}.
\begin{lemma}\label{l:de}
Let $\mu$ be a Radon measure on $\R^n$, $B$ be a Borel set and $s\in[0,n]$ be such that 
\[
 \limsup_{r\downarrow 0}\frac{\mu(B_r(x))}{\omega_sr^s}\geq t\qquad\text{for all $x\in B$}.
\]
Then, $\mu(B)\geq t\,\HH^{s}(B)$.
\end{lemma}
 
\begin{proposition}\label{p:strongweak}
Let $u\in\MM(\Om)$, then  $\HH^{n-1}(\Om\cap\SSu\setminus S_u)=0$.
In particular, $\big(u,\Om\cap\SSu)$ is a local minimizer for $\mathscr{F}$ (with $\gamma=0$).
\end{proposition}
\begin{proof}[Proof of Proposition~\ref{p:strongweak}]
In view of Theorem~\ref{t:CL} we may apply the density estimates of Lemma~\ref{l:de} 
to $\mu=\HH^{n-1}\res S_u$ and to the Borel set $\SSu\setminus S_u$ with $t=\theta_0$.
Therefore, we deduce that
\[
 \theta_0\HH^{n-1}(\Om\cap\SSu\setminus S_u)\leq \mu(\Om\cap\SSu\setminus S_u)=0.
\]
Clearly, $\MS(u)=\mathscr{F}(u,\Om\cap\SSu)$, and the conclusion follows at once.
\end{proof}
The argument for \eqref{e:dgcl0} used by De Giorgi, Carriero and Leaci in 
\cite{DGCL89}, and similarly in \cite{CL90} for \eqref{e:dgcl}, is indirect:
it relies on Ambrosio's $SBV$ compactness theorem and Poincar\'e-Wirtinger type 
inequality in $SBV$ established in \cite{DGCL89} (see also \cite[Theorem 4.14]{AFP00} and 
\cite[Proposition~2]{BarFoc} for a version in which boundary values are preserved) 
to analyze blow~up limits of minimizers (see subsection~\ref{ss:bup} for the definition of blow~ups) 
with vanishing jump energy and prove that they are harmonic functions (cf. \cite[Theorem 7.21]{AFP00}). 
A contradiction argument shows that on small balls the energy of local minimizers inherits the decay properties 
as that of harmonic functions. Actually, the proof holds true for much more general energies
(see \cite{FF97}, \cite[Chapter~7]{AFP00}). 

In the paper \cite{DF} an elementary proof valid only in $2$-dimensions and tailored on 
the $\MS$ energy is given. 
No Poincar\'e-Wirtinger inequality, nor any compactness argument are required. 
Moreover, it has the merit to exhibit an explicit constant.
Indeed, the proof in \cite{DF} is based on an observation of geometric nature 
and on a direct variational comparison argument. It also differs from those exploited 
in \cite{DMMS92} and \cite{Dav96} to derive \eqref{e:dgcl} in the $2$-dimensional case. 

\begin{theorem}[De Lellis and Focardi \cite{DF}]\label{t:maindlb}
Let $u\in\MM(\Om)$. Then 
\begin{equation}\label{e:nostra}
\MS(u,B_r(z))\geq r
\end{equation}
for all $z\in\Om\cap\SSu$ and all $r\in(0,\mathrm{dist}(z,\partial\Om))$.

More precisely, the set $\Om_u:=\{z\in\Om:\,\eqref{e:nostra} \text{ fails}\}$
is open and $\Omega_u=\Omega\setminus \SSu$\footnote{Actually, the very same proof shows also 
that $\Omega_u=\Omega\setminus\overline{J_u}$, where $J_u$ is the subset of points of 
$S_u$ for which one sided traces exist. Recall that $\HH^{n-1}(S_v\setminus J_v)=0$ for 
all $v\in BV(\Omega)$.}.
\end{theorem}
To the aim of establishing Theorem~\ref{t:maindlb} we prove a consequence of \eqref{e:EL}, 
a monotonicity formula discovered independently by David and L\'eger in \cite[Proposition~3.5]{DavLeg02} 
and by Maddalena and Solimini in \cite{MadSol01c}. 
The proof we present here is that given in \cite[Lemma~2.1]{DF} (an analogous result holds true 
in any dimension with essentially the same proof).
\begin{lemma}\label{l:rdlvartns}
Let $u\in\MM(\Om)$, $\Om\subset\R^2$, then for every $z\in\Om$ and for $\LL^1$ a.e. 
$r\in(0,\mathrm{dist}(z,\partial\Om))$
\begin{equation}\label{e:rdlvartns0}
 r\int_{\partial B_r(z)}\left(\left(\frac{\partial u}{\partial \nu}\right)^2-
\left(\frac{\partial u}{\partial\tau}\right)^2\right)d\HH^1+\HH^1(S_u\cap B_z(r))
=
\int_{S_u\cap\partial B_r(z)}|\langle \nu_u^\perp(x),x\rangle|d\HH^0(x),
\end{equation}
$\frac{\partial u}{\partial \nu}$ and $\frac{\partial u}{\partial\tau}$ 
being the projections of $\nabla u$ in the normal and tangential directions
to $\partial B_r(z)$, respectively.\footnote{For $\xi\in\R^2$, $\xi^\perp$ is the vector obtained by an 
anticlockwise rotation.}
\end{lemma}
\begin{proof}[Proof of Lemma~\ref{l:rdlvartns}]
With fixed a point $z\in\Om$, $r>0$ with $B_r(z)\subseteq\Om$, we consider 
special radial vector fields $\eta_{r,s}\in \mathrm{Lip}\cap C_c(B_r(z),\R^2)$,
$s\in(0,r)$, in the first variation formula \eqref{e:EL} (with $\gamma=0$). 
Moreover, for the sake of simplicity we assume $z=0$, 
and drop the subscript $z$ in what follows. Let 
\[
\eta_{r,s}(x):=x\,\chi_{[0,s]}(|x|)+\frac{|x|-r}{s-r}\,x\,\chi_{(s,r]}(|x|),
\]
then a routine calculation leads to
\[
\nabla\eta_{r,s}(x):={\rm Id}\,\chi_{[0,s]}(|x|)+\left(
\frac{|x|-r}{s-r}{\rm Id}+\frac 1{s-r}\frac x{|x|}\otimes x\right) 
\chi_{(s,r]}(|x|)
\]
$\LL^2$ a.e. in $\Om$. In turn, from the latter formula we infer
for $\LL^2$ a.e. in $\Om$ 
\[
\divv\eta_{r,s}(x)=2\,\chi_{[0,s]}(|x|)+\left(
2\frac{|x|-r}{s-r}+\frac {|x|}{s-r}
\right)\chi_{(s,r]}(|x|),
\]
and, if $\nu_u(x)$ is a unit vector normal field in $x\in S_u$,
for $\HH^1$ a.e. $x\in S_u$
\[
\divv^{S_u}\eta_{r,s}(x)=\chi_{[0,s]}(|x|)+\left(\frac{|x|-r}{s-r}+
\frac 1{|x|(s-r)}|\langle x,\nu_u^\perp\rangle|^2\right)\chi_{(s,r]}(|x|).
\]
Consider the set $I:=\{\rho\in(0,\mathrm{dist}(0,\partial\Om)):\,
\HH^1(S_u\cap\partial B_\rho)=0\}$, then
$(0,\mathrm{dist}(0,\partial\Om))\setminus I$ is at most countable
being $\HH^1(S_u)<+\infty$. If $\rho$ and $s\in I$, by inserting 
$\eta_s$ in \eqref{e:EL} 
we find
\begin{multline*}
\frac 1{s-r}\int_{B_r\setminus B_s}|x||\nabla u|^2dx
-\frac{2}{s-r}\int_{B_r\setminus B_s}|x|
\langle\nabla u,\left(\mathrm{Id}-\frac{x}{|x|}\otimes\frac{x}{|x|}
\right)\nabla u\rangle dx\\
=\HH^1(S_u\cap B_s)+\int_{S_u\cap(B_r\setminus B_s)}\frac{|x|-r}{s-r}d\HH^1+
\frac 1{s-r}\int_{S_u\cap(B_r\setminus B_s)}|x|
|\langle\frac x{|x|},\nu_u^\perp\rangle|^2d\HH^1.
\end{multline*}
Next we employ Co-Area formula and rewrite equality above as 
\begin{multline*}
\frac 1{s-r}\int_s^r\rho\,d\rho\int_{\partial B_\rho}|\nabla u|^2d\HH^1
-\frac{2}{s-r}\int_s^r\rho\,d\rho\int_{\partial B_\rho}
\left|\frac{\partial u}{\partial\tau}\right|^2d\HH^1\\
=\HH^1(S_u\cap B_s)+\int_{S_u\cap(B_r\setminus B_s)}\frac{|x|-r}{s-r}d\HH^1+
\frac 1{s-r}\int_s^rd\rho\int_{S_u\cap \partial B_\rho} 
|\langle x,\nu_u^\perp\rangle|d\HH^0
\end{multline*}
where $\nu:=x/|x|$ denotes the radial unit vector and 
$\tau:=\nu^\perp$ the tangential one. 
Lebesgue differentiation theorem then provides a subset $I^\prime$ of 
full measure in $I$ such that if $r\in I^\prime$ and we let $s\uparrow t^-$
it follows
\begin{equation*}
-r\int_{\partial B_r}|\nabla u|^2d\HH^1
+2r\int_{\partial B_r}\left|\frac{\partial u}{\partial\tau}\right|^2d\HH^1
=\HH^1(S_u\cap B_r)-\int_{S_u\cap \partial B_r}|\langle x,\nu_u^\perp\rangle|d\HH^0.
\end{equation*}
Formula \eqref{e:rdlvartns0} then follows straightforwardly.
\end{proof}
We are now ready to prove Theorem~\ref{t:maindlb}.
\begin{proof}[Proof of Theorem~\ref{t:maindlb}]

Given $u\in\MM(\Om)$, $z\in\Om$ and $r\in(0,\mathrm{dist}(z,\partial\Om))$ let
\[
e_z(r):=\int_{B_r(z)}|\nabla u|^2dx,\quad 
\ell_z(r):=\HH^1(S_u\cap B_r(z)),
\]
and
\[
m_z(r):=\MS(u,B_r(z)),\quad h_z(r):=e_z(r)+\frac 12\ell_z(r).
\]
Clearly, $m_z(r)=e_z(r)+\ell_z(r)\leq 2h_z(r)$, with 
 equality if and only if $e_z(r)=0$. 

%
%


Introduce the set $S^\star_u$ of points $x\in S_u$ for which
\begin{equation}\label{e:justar2}
\lim_{r\downarrow 0} \frac{\mathcal{H}^1 (S_u\cap B_r (x))}{2r} =1\, .
\end{equation}
Since $S_u$ is rectifiable, $\mathcal{H}^1 (S_u\setminus S^\star_u)=0$.
Next let $z\in \Omega$ be such that
\begin{equation}\label{e:assmptn}
m_z (R)<R\qquad\mbox{for some $R\in(0,\mathrm{dist}(z,\partial\Om))$.}
\end{equation}
We claim that $z\not\in S^\star_u$.  

W.l.o.g. we take $z=0$ and drop the subscript $z$ in $e, \ell, m$ and $h$.

In addition we can assume $e(R)>0$. Otherwise, by the Co-Area formula and 
the trace theory of BV functions, we would find a radius $r<R$ such that 
$u|_{\partial B_r}$ is a constant (cf. the argument below). In turn, $u$ 
would necessarily be constant in $B_r$ because the energy decreases under 
truncations, thus implying $z\not\in S^\star_u$.
We can also assume $\ell(R)>0$, since otherwise $u$ would be harmonic in 
$B_R$ and thus we would conclude $z\not\in S^\star_u$.

We start next to compare the energy of $u$ with that of an harmonic 
competitor on a suitable disk. The inequality $\ell (R)\leq m (R) <R$
is crucial to select good radii.
\medskip

\noindent \emph{Step 1:} For any fixed $r\in(0,R-\ell(R))$, there exists 
a set $I_r$ of positive length in $(r,R)$ such that 
\begin{equation}\label{e:step1}
\frac{h(\rho)}\rho\leq\frac 12\cdot
\frac{e(R)-e(r)}{R-r-(\ell(R)-\ell(r))} \qquad \mbox{for all $\rho\in I_r$.}
\end{equation}
Define $J_r:=\{t\in(r,R):\,\HH^0(S_u\cap\partial B_t)=0\}$.
We claim the existence of $J^\prime_r\subseteq J_r$ with 
$\LL^1(J^\prime_r)>0$ and such that 
\begin{equation}\label{e:mean}
\int_{\partial B_\rho}|\nabla u|^2d\HH^1\leq
\frac{e(R)-e(r)}{R-r-(\ell(R)-\ell(r))} \qquad\mbox{for all $\rho\in J^\prime_r$.}
\end{equation}
Indeed, we use the Co-Area formula for rectifiable sets 
(see \cite[Theorem 2.93]{AFP00}) to find
\[
 \LL^1((r,R)\setminus J_r)\leq
\int_{(r,R)\setminus J_r}\HH^0(S_u\cap\partial B_t)dt
=\int_{S_u\cap (B_R\setminus\overline{B_r})}
\left|\langle \nu_u^\perp(x),\frac x{|x|}\rangle\right|d\HH^1(x)
\leq\ell(R)-\ell(r).
\]
In turn, this inequality implies $\LL^1(J_r)\geq R-r-(\ell(R)-\ell(r))>0$, 
thanks to the choice of $r$.  Then, define $J^\prime_r$ to be the subset 
of radii $\rho\in J_r$ for which 
\[
\int_{\partial B_\rho}|\nabla u|^2d\HH^1\leq
\fint_{J_r}\left(\int_{\partial B_t}|\nabla u|^2d\HH^1\right)dt\, .
\]
Formula \eqref{e:mean} follows by the Co-Area formula 
and the estimate 
$\LL^1(J_r)\geq R-r-(\ell(R)-\ell(r))$.

We define $I_r$ as the subset of radii $\rho \in J^\prime_r$ 
satisfying both \eqref{e:rdlvartns0} and \eqref{e:mean}. 
Therefore,
\begin{equation}\label{e:stima1}
\int_{\partial B_\rho}\left(\frac{\partial u}{\partial \tau}\right)^2d\HH^1
=\frac 12\int_{\partial B_\rho}|\nabla u|^2d\HH^1+\frac{\ell(\rho)}{2\rho}
\qquad \forall \rho\in I_r.
\end{equation}
Clearly, $I_r$ has full measure in $J^\prime_r$, so that $\LL^1(I_r)>0$.  

For any $\rho\in I_r$, we let $w$ be the harmonic 
function in $B_\rho$ with trace $u$ on $\partial B_\rho$. 
Then, as 
$\frac{\partial w}{\partial\tau}=\frac{\partial u}{\partial\tau}$
$\HH^1$ a.e. on $\partial B_\rho$, the local minimality of $u$ entails
\[
m(\rho)\leq\int_{B_\rho}|\nabla w|^2dx\leq 
\rho\int_{\partial B_\rho}\left(\frac{\partial u}{\partial\tau}\right)^2d\HH^1
\stackrel{\eqref{e:stima1}}{=}
\frac \rho2\int_{\partial B_\rho}|\nabla u|^2d\HH^1+\frac{\ell(\rho)}{2}.
\]
The inequality \eqref{e:step1} follows from the latter inequality and from \eqref{e:mean}:
\begin{equation*}
h(\rho)=\erho+\frac{\lrho}2\leq
\frac\rho2\int_{\partial B_\rho}|\nabla u|^2d\HH^1\leq
\frac\rho 2\cdot\frac{e(R)-e(r)}{R-r-(\ell(R)-\ell(r))}.
\end{equation*}
\medskip

\noindent \emph{Step 2:} We now show that $0\not\in S^\star_u$.

Let $\e\in(0,1)$ be fixed such that $m(R)\leq(1-\e)R$, 
and fix any radius $r\in(0,R-\ell(R)-\frac 1{1-\e}e(R))$.
Step~1 and the choice of $r$ then imply
\[
\frac{h(\rho)}{\rho}\leq\frac 12\frac{e(R)-e(r)}{R-r-(\ell(R)-\ell(r))}
\leq\frac{e(R)}{2(R-\ell(R)-r)}<\frac{1-\e}2,
\]
in turn giving $m(\rho)\leq 2h(\rho)<(1-\e)\rho$.
Let $\rho_\infty:=\inf\{t>0:\,m(t)\leq(1-\e)t\}$, then 
$\rho_\infty\in[0,\rho]$. Note that if
$\rho_\infty$ were strictly positive then actually $\rho_\infty$ would be 
a minimum. In such a case, we could apply the argument above and find 
$\widetilde{\rho}\in(r_\infty,\rho_\infty)$, with
$r_\infty\in(0,\rho_\infty-\ell(\rho_\infty)-\frac 1{1-\e}e(\rho_\infty))$,
such that $m(\widetilde{\rho})<(1-\e)\widetilde{\rho}$ 
contradicting the minimality of $\rho_\infty$.
Hence, there is a sequence $\rho_k\downarrow 0^+$ with
$m(\rho_k) \leq (1-\e) \rho_k$. Then, clearly condition \eqref{e:justar2} is 
violated, so that $0\not\in S^\star_u$.

\medskip
\noindent \emph{Conclusion:}
We first prove that $\Om_u$ is open. Let 
$z\in\Om_u$ and let $R>0$ and $\e>0$ be such that $m_z (R) \le (1-\e) R$
and $B_{\e R}(z)\subset\Om$. Let now  $x\in B_{\e R}(z)$, then 
\[
m_x(R-|x-z|)\leq m_z(R)\le(1-\e)R<R-|x-z|,
\]
therefore $x\in \Omega_u$. 

As $\Om_u$ is open and $S^\star_u\cap \Omega_u =\emptyset$ by Step 2,  
we infer $\Om\cap\overline{S_u^\star}\subseteq\Om\setminus\Om_u$. 
Moreover, let $z\notin\overline{S_u^\star}$ and $r>0$ be such that 
$B_r(z)\subseteq\Om\setminus\overline{S_u^\star}$. Since 
$\HH^1(S_u\setminus\overline{S_u^\star})=0$,
$u\in W^{1,2}(B_r(z))$ and thus $u$ is an harmonic function in $B_r(z)$ 
by minimality. Therefore $z\in\Om_u$, and in conclusion 
$\Om\setminus\Om_u=\Om\cap\overline{S_u^\star}=\Om\cap\overline{S_u}$.
\end{proof}

A simple iteration of Theorem~\ref{t:maindlb} gives a density lower bound 
as in \eqref{e:dgcl} with an explicit constant $\theta_0$ (see \cite[Corollary 1.2]{DF}).  
\begin{corollary}[De Lellis and Focardi \cite{DF}]\label{c:dlbrough1}
If $u\in\MM(\Om)$, then $\mathcal{H}^1 (\Om\cap\SS_u\setminus J_u)=0$ and
\begin{equation}\label{e:dlb}
\HH^1(S_u\cap B_r(z))\geq \frac\pi{2^{23}}r 
\end{equation}
for all $z\in\Om\cap\SS_u$ and all $r\in(0,\mathrm{dist}(z,\partial\Omega))$.
\end{corollary}

A similar result can be established for quasi-minimizers of the Mumford and 
Shah energy, the most prominent examples being minimizers of the functional 
$\mathscr{F}$ in equation \eqref{e:enrgcomplete}.
More precisely, a \emph{quasi-minimizer} is any function $u$ in $SBV(\Om)$ with 
$\MS(u)<\infty$ satisfying for some $\alpha>0$ and $c_\alpha\geq 0$ for all balls 
$B_\rho(z)\subset\Om$
\[
\MS(u,B_\rho(z))\leq \MS(w,B_\rho(z))+c_\alpha\,\rho^{1+\alpha}
\]
whenever $\{w\neq u\}\subset\hskip-0.125cm\subset B_\rho(z)$.

One can then prove the ensuing infinitesimal version of \eqref{e:nostra}
(cf. with \cite[Corollary 1.3]{DF}). 
\begin{corollary}[De Lellis and Focardi \cite{DF}]\label{c:qmin}
Let $u$ be a quasi-minimizers of the Mumford and Shah energy, then 
\begin{equation}\label{e:dlbqm}
\Om\cap\overline{J_u}=\Om\cap\SSu=\left\{z\in\Om:\,\liminf_{r\downarrow 0^+}
\frac{m_z(r)}{r}\geq \frac 23\right\}.
\end{equation}
\end{corollary}
The proof of this corollary, though, needs a blow~up analysis 
and a new $SBV$ Poincar\'e-Wirtinger type inequality of independent 
interest, obtained by improving upon some ideas contained in \cite{FGP07} 
(cf. with \cite[Theorem B.6]{DF}); it is, therefore, much more technical.

Let us remark that it is possible to improve slightly 
Theorem~\ref{t:maindlb} by combining the ideas of its proof hinted to
above with the $SBV$ Poincar\'e-Wirtinger type inequality 
in \cite[Theorem B.6]{DF}, and show that actually
\[
\Om\setminus\SSu=\{x\in\Om:\,m_x(r)\leq r\quad\mbox{ for some 
$r\in\mathrm{dist}(x,\partial\Om)$}\}
\]
(see \cite[Remark 2.3]{DF}).

A natural question is the sharpness of the estimates \eqref{e:nostra} 
and \eqref{e:dlb}. The analysis performed by Bonnet~\cite{B96} suggests 
that $\sfrac{\pi}{2^{23}}$ in \eqref{e:dlb} should be replaced by $1$,
and $1$ in \eqref{e:nostra} by $2$. Note that the square root function 
$u(\rho,\theta)=\sqrt{\frac 2\pi \rho}\cdot\sin(\sfrac\theta2)$ satisfies 
$\MS(u,B_\rho(0)) = \HH^1(S_u\cap B_\rho(0)) = \rho$ for all $\rho>0$
(its minimality will be discussed in subsection~\ref{ss:MSconj}).
Thus both the constants conjectured above would be sharp by 
\cite[Section 62]{Dav05}. Unfortunately, none of them have been proven so far.

We point out that Bucur and Luckhaus \cite{BL}, independently from \cite{DF1}, 
have been able to improve the ideas in Theorem~\ref{t:maindlb} carrying on the 
proof without the $2$-dimensional limitation via a delicate induction argument.
Their approach leads to a remarkable monotonicity formula for (a truncated version 
of) the energy valid for a broad class of approximate minimizers that shall be 
the topic of the next subsection~\ref{ss:monotonicity}.

To conclude this paragraph we notice that the derivation of an energy upper bound for minimizers on 
balls is much easier as a result of a simple comparison argument.
\begin{proposition}\label{p:dub}
 For every $u\in\MM(\Om)$ and $B_r(z)\subseteq\Om$ 
 \begin{equation}\label{e:dub}
  \MS(u,B_r(z))\leq n\,\omega_n\, r^{n-1}.
 \end{equation}
 \end{proposition}
\begin{proof}
 Let $\e>0$, consider $u_\e:=u\chi_{\Om\setminus B_{r-\e}(z)}$ and test the minimality of $u$ with $u_\e$.
 Conclude by letting $\e\downarrow 0$.
\end{proof}

\subsection{Bucur and Luckhaus' almost monotonicity formula}\label{ss:monotonicity}

Let us start off by introducing the notion of \emph{almost-quasi minimizers} of the $\MS$ energy.
\begin{definition}
 Let $\Lambda\geq 1$, $\alpha>0$ and $c_\alpha\geq 0$, a function $u\in SBV(\Omega)$ with $\MS(u)<\infty$ 
is a $(\Lambda,\alpha,c_\alpha)$-almost-quasi minimizer in $\Omega$, we write 
$u\in\MM_{\{\Lambda,\alpha,c_\alpha\}}(\Omega)$, if for all balls $B_\rho(z)\subset\Omega$
 \[
  \MS(u,B_\rho(z))\leq\int_{B_\rho(z)}|\nabla v|^2\,dx+\Lambda\,\HH^{n-1}\big(S_v\cap B_\rho(z)\big)
+c_\alpha\,\rho^{n-1+\alpha},
 \]
whenever $\{u\neq v\}\subset\hskip-0.125cm\subset\Omega$.
\end{definition}
Clearly, minimizers are $(1,\alpha,c_\alpha)$-almost minimizers for all $\alpha$ and $c_\alpha$ as in the 
definition above. Almost-quasi minimizers has turned out to be useful in studying the regularity properties 
of solutions to free boundary - free discontinuity problems with Robin conditions (cf. \cite[Section~4]{BL}).
We are now ready to state the Bucur and Luckhaus' almost monotonicity formula.
\begin{theorem}[Bucur and Luckhaus \cite{BL}]\label{t:monform}
There is a (small) dimensional constant $C(n)>0$ such that for all $u\in\MM_{\{\Lambda,\alpha,c_\alpha\}}(\Om)$ 
and $z\in\Om$ the function
\begin{equation}\label{e:monform}
\big(0,\mathrm{dist}(z,\partial\Om)\big)\ni r\longmapsto E_z(r):=\Big(
\frac{\MS(u,B_r(z))}{r^{n-1}}\wedge\frac{C(n)}{n-1}\Lambda^{2-n}\Big)
+(n-1)\,\frac{c_\alpha}{\alpha}\,r^\alpha
\end{equation}
is non decreasing.
\end{theorem}
To explain the strategy of proof of Theorem~\ref{t:monform} we need to state two additional 
results on $SBV$ functions defined on boundaries of balls. The first is related to the 
approximation with $H^1$ functions, the second to the problem of lifting.
The tangential gradient on boundaries of balls shall be denoted by $\nabla_\tau$ in what follows.
\begin{lemma}\label{l:SBVbdry}
Let $n\geq 2$, then there is a (small) dimensional constant $C(n)>0$ such that for all 
$v\in SBV(\partial B_r)$ with
\begin{equation}\label{e:bdryenrgy}
\EE(v,\partial B_r):=\int_{\partial B_r}|\nabla_\tau v|^2d\HH^{n-1}+\HH^{n-2}(S_v\cap \partial B_r)
\leq C(n)\,\Lambda^{2-n}\,r^{n-2}, 
\end{equation}
there exists  $w\in H^1(\partial B_r)$ such that 
\begin{equation}\label{e:tracelifted}
\EE(w,\partial B_r)+(n-1)\,\frac\Lambda r\,\HH^{n-1}\big(\{x\in \partial B_r:\,v\neq w\}\big)\leq \EE(v,\partial B_r).
\end{equation}
\end{lemma}
\begin{lemma}\label{l:Harmbdry}
Let $n\geq 2$, suppose $v\in SBV(\partial B_r)$ satisfies \eqref{e:bdryenrgy} in Lemma~\ref{l:SBVbdry}.
Then there exists $\widehat{w}\in H^1(B_r)$ harmonic in $B_r$ such that
\begin{equation}\label{e:tracelifted2}
\MS(\widehat{w},B_r)+\Lambda\,\HH^{n-1}\big(\{x\in\partial B_r:\,v\neq \widehat{w}\}\big)
\leq\frac{r}{n-1}\,\EE(v,\partial B_r).
\end{equation}
\end{lemma}
The proof of Theorem~\ref{t:monform} is based on a cyclic induction argument that starts with
the validation of Lemma~\ref{l:SBVbdry} in $\R^2$ and then runs as follows:
\[
\textrm{Lemma~\ref{l:SBVbdry} in $\R^n$}{\Rightarrow}
\textrm{Lemma~\ref{l:Harmbdry} in $\R^n$}{\Rightarrow}
\textrm{Theorem~\ref{t:monform} in $\R^n$}{\Rightarrow}
\textrm{Lemma~\ref{l:SBVbdry} in $\R^{n+1}$}.
\]
The first inductive step can be established by using the geometric argument in the proof of 
Theorem~\ref{t:maindlb} with the constant $C(2)$ being any number in $(0,1)$.

Before proving all the implications above we establish the essential closure of $S_u$ 
for almost-quasi minimizers as a consequence of Theorem~\ref{t:monform}.
\begin{theorem}\label{t:essentialclosure}
Let $u\in\MM_{\{\Lambda,\alpha,c_\alpha\}}(\Omega)$, for some $\Lambda\geq 1$, $\alpha>0$ and $c_\alpha\geq 0$.
Then, $\HH^{n-1}\big(\Om\cap\SSu\setminus S_u\big)=0$.
\end{theorem}
\begin{proof}
 Consider the subset $S_u^\star$ of points in $S_u$ with density one, i.e.~$x\in S_u$ such that 
 \[
 \lim_{\rho\downarrow 0}\frac{\HH^{n-1}\big(S_u\cap B_\rho(x)\big)}{\omega_{n-1}\rho^{n-1}}=1
  \]
 (cf. \eqref{e:justar2} in Theorem~\ref{t:maindlb}), then clearly $\overline{S_u^\star}=\SSu$.
Theorem~\ref{t:monform} yields for $x\in S_u^\star$
\[
 E_x(\rho)\geq\lim_{\rho\downarrow 0}E_x(\rho)\geq\omega_{n-1}\wedge \frac{C(n)}{n-1}\Lambda^{2-n}\qquad
 \forall\rho\in\big(0,\mathrm{dist}(x,\partial\Omega)\big).
\]
By approximation it is then easy to deduce the same estimate for all points $x\in \Om\cap\SSu$. In turn, 
this implies that $\Om\cap\SSu\setminus S_u\subset A_u:=\{y\in\Omega\setminus S_u:\,\liminf_\rho E_x(\rho)>0\}$.

In particular, the density estimates in Lemma~\ref{l:de} applied to the measure induced
on Borel sets by $\MS(u,\cdot)$ gives that $\HH^{n-1}(A_u)=0$, and the conclusion 
$\HH^{n-1}\big(\Om\cap\SSu\setminus S_u\big)=0$ follows at once (cf. the proof of Proposition~\ref{p:strongweak}).
\end{proof}
Density lower bounds for $\MS(u,\cdot)$ as in Theorem~\ref{t:DGCL} or for $\HH^{n-1}\res S_u$ as in Theorem~\ref{t:CL} 
can also be recovered. Either by means of Theorem~\ref{t:monform} and of a decay lemma due to De~Giorgi, Carriero 
and Leaci or in an alternative direct way that provides explicit constants (cf. \cite[Lemma~7.14]{AFP00} and 
\cite[Section~3.3]{BL}, respectively).

Let us now turn to the proofs of the implications among Theorem~\ref{t:monform}, Lemma~\ref{l:SBVbdry} and 
Lemma~\ref{l:Harmbdry}. Recall that we have already commented on the validity of Lemma~\ref{l:SBVbdry} in case $n=2$.
\begin{proof}[Proof of Lemma~\ref{l:SBVbdry} in $\R^n{\Rightarrow}$ Lemma~\ref{l:Harmbdry} in $\R^n$]
 Denote by $w$ the function provided by Lemma~\ref{l:SBVbdry} and by $\widehat{w}$ its harmonic 
 extension to $B_\rho$. Note that\footnote{The first inequality follows, for instance, from 
 Almgren's monotonicity formula for the frequency function (cf. \cite[Exercise~20 pg. 525]{Evans})
 and a direct comparison argument with the one-homogeneous extension of the boundary trace. 
 Alternatively, one can expand $\widehat{w}$ in spherical harmonics.}
 \[
  \frac{n-1}{\rho}\int_{B_\rho}|\nabla \widehat{w}|^2dx\leq 
  \int_{\partial B_\rho}|\nabla_\tau \widehat{w}|^2d\HH^{n-1}\leq\EE(w,\partial B_\rho). 
 \]
Hence, by \eqref{e:tracelifted} we conclude that
\begin{multline*}
 \MS(\widehat{w},B_\rho)+\Lambda\,\HH^{n-1}\big(\{x\in\partial B_\rho:\,v\neq \widehat{w}\}\big)\\
\leq\frac{\rho}{n-1}\,\Big(\EE(w,\partial B_\rho)
+(n-1)\frac{\Lambda}\rho\,\HH^{n-1}\big(\{x\in\partial B_\rho:\,v\neq \widehat{w}\}\big)\Big)
\leq \frac{\rho}{n-1}\,\EE(v,\partial B_\rho).
\end{multline*}
\end{proof}
\begin{proof}[Proof of Lemma~\ref{l:Harmbdry} in $\R^n{\Rightarrow}$ Theorem~\ref{t:monform} in $\R^n$]
Set $I:=\big(0,\mathrm{dist}(0,\partial\Om)\big)$, then the energy function 
$m_z(\rho):=\MS(u,B_\rho(z)):I\to[0,+\infty)$ is non-decreasing and thus it belongs to $BV_{loc}(I)$. 
Note that 
\begin{equation}\label{e:Ezmz}
 E_z(\rho)=\Big(\frac{m_z(\rho)}{\rho^{n-1}}\wedge\frac{C(n)}{n-1}\,\Lambda^{2-n}\Big)
+(n-1)\,\frac{c_\alpha}{\alpha}\rho^{\alpha},
\end{equation}
therefore, $E_z\in BV_{loc}(I)$. Denoting by $m_z^\prime$ and $D^sm_z$,  respectively, the density 
of the absolutely continuous part and the singular part of $Dm_z$ with respect to $\LL^1\res I$ in 
the Radon-Nikodym decomposition, the Leibnitz rule for $BV$ functions (cf. \cite[Example~3.97]{AFP00}) 
yields
\[
 D\Big(\frac{m_z(\rho)}{\rho^{n-1}}\Big)=
 \Big(\frac{m_z^\prime(\rho)}{\rho^{n-1}}-(n-1)\frac{m_z(\rho)}{\rho^{n}}\Big)+\frac{1}{\rho^{n-1}}D^sm_z(\rho).
\]
Therefore, the latter equality, \eqref{e:Ezmz} and the Chain Rule formula for Lipschitz functions 
\cite[Theorem~3.99]{AFP00} imply that 
$E_z$ is non-decreasing if and only if the density $E_z^\prime$ of the absolutely continuous part of the 
distributional derivative $DE_z$ is non-negative. 

By the locality of the distributional derivative (see \cite[Remark~3.93]{AFP00}) it holds that
$E^\prime_z(\rho)=(n-1)\,c_\alpha\rho^{\alpha-1}>0$ at $\LL^1$ a.e. $\rho\in I$ for which 
\[
{m_z(\rho)}\geq \frac{C(n)}{n-1}\,\Lambda^{2-n}\rho^{n-1}.
\]
Instead, at $\LL^1$ a.e. $\rho\in I$ at which
\begin{equation}\label{e:caso2}
{m_z(\rho)}<\frac{C(n)}{n-1}\,\Lambda^{2-n}{\rho^{n-1}}
\end{equation}
we have 
\[
E^\prime_z(\rho)=\frac{m_z^\prime(\rho)}{\rho^{n-1}}-(n-1)\frac{m_z(\rho)}{\rho^{n}}+(n-1)\,c_\alpha\rho^{\alpha-1}.
\]
Suppose by contradiction that $E^\prime_z<0$ on some subset $J$ of $I$ of positive measure, i.e.
\begin{equation}\label{e:Eminzero}
m_z^\prime(\rho)<(n-1)\frac{m_z(\rho)}{\rho}-(n-1)\,c_\alpha\rho^{n-2+\alpha}.
\end{equation}
Since for $\LL^1$ a.e. in $I$ by the slicing theory $u|_{\partial B_\rho}\in SBV(\partial B_\rho)$ 
(cf. \cite[Section~3.11]{AFP00}) and by the Co-Area formula $\EE(u,\partial B_\rho)\leq m_z^\prime(\rho)$ (cf. 
\cite[Theorem~2.93]{AFP00}), we conclude that for $\LL^1$ a.e. $\rho\in J$
\[
\EE(u,\partial B_\rho)\leq m_z^\prime(\rho)\stackrel{\eqref{e:caso2},\eqref{e:Eminzero}}{<}\, C(n)\,\Lambda^{2-n}\,\rho^{n-2}-
(n-1)\,c_\alpha\,\rho^{n-2+\alpha}.
\]
Hence, with fixed $\rho\in J$ as above, Lemma~\ref{l:Harmbdry} provides an harmonic function 
$\widehat{w}\in H^1(B_\rho)$  satisfying \eqref{e:tracelifted2}. In turn, the latter inequality and the assumption $E^\prime_z(\rho)<0$ give
\begin{multline*}
\MS(\widehat{w},B_\rho)+\Lambda\,\HH^{n-1}(\{x\in\partial B_\rho:\,u\neq \widehat{w}\})\\
 \stackrel{\eqref{e:tracelifted2}}{\leq}\frac{\rho}{n-1}\EE(u,\partial B_\rho)\leq \frac{\rho}{n-1}m_z^\prime(\rho)
\stackrel{\eqref{e:Eminzero}}{<}m_z(\rho)-\,c_\alpha\rho^{n-1+\alpha},
\end{multline*}
leading to a contradiction to the almost-quasi minimality of $u$ in $\Omega$ 
by taking the trial function 
$\widehat{w}\,\chi_{B_\rho(z)}+u\,\chi_{\Omega\setminus B_\rho(z)}$.
\end{proof}
To conclude the cyclic induction argument we set some notation: in the following proof we 
denote by $B_\rho$ the $(n+1)$-dimensional ball of radius $\rho$ and by $B_\rho^n$ its intersection 
with the hyperplane $\R^n\times\{0\}$. Moreover, we still denote by $\EE$ the $n$ dimensional version 
of the boundary energy in \eqref{e:bdryenrgy}.
\begin{proof}[Proof of Theorem~\ref{t:monform} in $\R^n{\Rightarrow}$ Lemma~\ref{l:SBVbdry} in $\R^{n+1}$]
Up to a scaling argument we may assume the radius $r$ in the statement of Lemma~\ref{l:SBVbdry} to be $1$. 

Then, consider $v\in SBV(\partial B_1)$ such that
\begin{equation}\label{e:env}
 \EE(v,\partial B_1)\leq C\,\Lambda^{1-n}
\end{equation}
for some constant $C>0$.

\noindent{\bf Claim:} There exists $C(n+1)>0$ such that if $v$ satisfies \eqref{e:env} with $C\in (0,C(n+1)]$, 
every minimizer $w\in SBV(\partial B_1)$ of the problem 
\[
 \inf_{\zeta\in SBV(\partial B_1)}F(\zeta,\partial B_1) 
\]
actually belongs to $H^1(\partial B_1)$, where for all open sets $A\subseteq\partial B_1$ and $\zeta\in SBV(\partial B_1)$
if $\EE(\zeta,A)$ is defined as in \eqref{e:bdryenrgy} by integrating $\zeta$ on $A$, then
\begin{equation}\label{e:pbaux}
F(\zeta,A):=\EE(\zeta,A) 
 +n\,\Lambda\,\HH^n\big(\{x\in A:\,\zeta\neq v\}\big).
\end{equation}
Given the claim above for granted we conclude the thesis of Lemma~\ref{l:SBVbdry} straightforwardly
by comparing the values of the energy $F(\cdot,\partial B_1)$ in \eqref{e:pbaux} on $w$ and $v$ respectively, 
namely
\[
 \int_{\partial B_1}|\nabla_\tau w|^2d\HH^n+n\,\Lambda\,\HH^n\big(\{x\in\partial B_1:\,w\neq v\}\big)
 \leq\int_{\partial B_1}|\nabla_\tau v|^2d\HH^n+\HH^{n-1}(S_v).
\]

We are then left with proving the claim above. Suppose by contradiction that for some constant $C>0$
some minimizer $w$ of \eqref{e:pbaux} satisfies $\HH^{n-1}(S_w)>0$. 
Note that $C$ can be taken arbitrarily small, it shall be chosen suitably in what follows.
Even more, assume that the north pole 
$\mathscr{N}:=(0,\ldots,0,1)\in\partial B_1$ is a point of density one for $S_w$. Set $\lambda:=2-\sqrt{3}$ and 
denote by $\pi:\partial B_1\cap B_{\lambda}(\mathscr{N})\to B_{\sfrac 12}^n$ the orthogonal 
projection, then $\pi\in\mathrm{Lip}_{1}(\partial B_1\cap B_{\lambda}(\mathscr{N}),B^n_{\sfrac12})$ and 
$\pi^{-1}\in\mathrm{Lip}_{\ell}(B^n_{\sfrac12},\partial B_1\cap B_{\lambda}(\mathscr{N}))$, for some $\ell>0$.
Actually, as $\pi\in C^{1}(\partial B_1\cap B_{\lambda}(\mathscr{N}),B^n_{\sfrac12})$ and 
$\pi^{-1}\in C^1(B^n_{\sfrac12},\partial B_1\cap B_{\lambda}(\mathscr{N}))$ 
\begin{equation}\label{e:lipconst}
\mathrm{Lip}\big(\pi,\pi^{-1}(B^n_\rho)\big)\to 1,\quad
\mathrm{Lip}\big(\pi^{-1},B^n_\rho\big)\to 1\qquad\text{as }\rho\downarrow 0.
\end{equation}
For $\zeta\in SBV(\partial B_1\cap B_{\lambda}(\mathscr{N}))$ let $\overline{\zeta}:=\zeta\circ\pi^{-1}$, 
then $\overline{\zeta}\in SBV(B_{\sfrac 12}^n)$ and $S_{\overline{\zeta}}=\pi(S_\zeta)$. Moreover, for all
$\rho\in(0,\sfrac 12)$, 
we have
\begin{equation}\label{e:overlinewsur}
\HH^{n-1}\big(S_{\overline{\zeta}}\cap B^n_\rho\big)\leq \HH^{n-1}\big(S_\zeta\cap\pi^{-1}(B^n_\rho)\big)
\leq \ell^{n-1}
\,\HH^{n-1}(S_{\overline{\zeta}}\cap B^n_\rho),
\end{equation}
and by the generalized Area Formula (see \cite[Theorem~2.91]{AFP00})
\begin{equation}\label{e:overlinewvol}
\int_{B^n_\rho}|\nabla \overline{\zeta}|^2dy\leq k_1(\rho)
 \int_{\pi^{-1}(B^n_\rho)}|\nabla_\tau \zeta|^2d\HH^n,
 \end{equation}
and
\begin{equation}\label{e:overlinewvol2}
 \int_{\pi^{-1}(B^n_\rho)}|\nabla_\tau \zeta|^2d\HH^n\leq k_2(\rho)
\int_{B^n_\rho}|\nabla \overline{\zeta}|^2dy,
 \end{equation}
for some $k_1(\rho)$ and $k_2(\rho)>0$, with
$k_1(\rho)$, $k_2(\rho)\downarrow 1$ as $\rho\downarrow 0$ by \eqref{e:lipconst}. In particular, 
for some $\tau_n>0$, $0\leq k_1(\rho)\vee k_2(\rho)\leq 1+\tau_n\rho$ as $\rho\downarrow 0$.

We next prove that $\overline{w}$ is a $(\Lambda_1,1,c_1)$-almost-quasi minimizer on $B^n_{\sfrac12}$ 
of the $n$-dimensional Mumford and Shah energy for $\Lambda_1:=\ell^{n-1}$ and a suitable $c_1=c_1(n,\ell,\Lambda)>0$.
Indeed, recalling that $\Lambda\geq 1$, if $\overline{\zeta}$ is a test function for $\overline{w}$,
i.e.~$\{y\in B^n_{\sfrac 12}:\,\overline{\zeta}\neq\overline{w}\}\subset\hskip-0.125cm\subset B^n_{\rho}$, 
$\rho\in(0,\sfrac12)$, we deduce from \eqref{e:overlinewsur}, \eqref{e:overlinewvol}, \eqref{e:overlinewvol2} 
and the minimality of $w$ for the energy in \eqref{e:pbaux} that
\begin{align}
\MS&(\overline{w},B^n_\rho)\leq\int_{B^n_\rho}|\nabla\overline{w}|^2dx
+\HH^{n-1}(S_{\overline{w}}\cap B^n_\rho)
+\HH^n(\{y\in B^n_\rho:\,\overline{w}\neq\overline{v}\})\notag\\
 &\leq k_1(\rho)\int_{\pi^{-1}(B^n_\rho)}|\nabla_\tau w|^2d\HH^n
 +\HH^{n-1}\big(S_w\cap \pi^{-1}(B^n_\rho)\big)
 +\HH^n\big(\{x\in \pi^{-1}(B^n_\rho):\,w\neq v\}\big)\notag\\
 &\leq k_1(\rho)F(w,\pi^{-1}(B^n_\rho))\leq k_1(\rho)F(\zeta,\pi^{-1}(B^n_\rho))\label{e:testw0}\\
 &\leq k_1(\rho)\Big( k_2(\rho)\int_{B^n_\rho}|\nabla\overline{\zeta}|^2dx
 + \Lambda_1\HH^{n-1}(S_{\overline{\zeta}}\cap B^n_\rho)
 + n\,\Lambda\,\Lambda_1^{\frac{n}{n-1}}\HH^n(\{y\in B^n_\rho:\,\overline{\zeta}\neq\overline{v}\})
 \Big)\notag\\
 &\leq (1+\tau_n\,\rho)^2\Big( \int_{B^n_\rho}|\nabla\overline{\zeta}|^2dx
 + \Lambda_1\HH^{n-1}(S_{\overline{\zeta}}\cap B^n_\rho)\Big)
 + (1+\tau_n)n\omega_n\,\Lambda\,\Lambda_1^{\frac{n}{n-1}}\rho^n,\label{e:testw1}
 \end{align}
 being $\HH^n(\{y\in B^n_\rho:\,\overline{\zeta}\neq\overline{v}\})\leq\omega_n\rho^n$.

Next we note that $w$ satisfies the energy upper bound $F(w,\pi^{-1}(B^n_\rho))\leq k_n\rho^{n-1}$,
for all $\rho\in(0,\sfrac 12)$ and for some $k_n>0$,  by a direct comparison argument similar to that in Proposition~\ref{p:dub}. Therefore, \eqref{e:testw0} yields
\[
\MS(\overline{w},B^n_\rho)\leq(1+\tau_n)k_n\rho^{n-1}.
\]
Hence, we may assume $F(\zeta,\pi^{-1}(B^n_\rho))\leq 2(1+\tau_n) k_n\rho^{n-1}$ 
being otherwise the conclusion obvious. The latter condition and \eqref{e:testw1} then imply 
$\overline{w}\in\MM_{\{\Lambda_1,1,c_1\}}\big(B^n_{\sfrac 12}\big)$ with
$c_1:=(1+\tau_n)\big(n\,\omega_n\,\Lambda\,\Lambda_1^{\frac{n}{n-1}}+2(1+\tau_n)k_n\big)$.

By inductive assumption, Theorem~\ref{t:monform} implies in particular that
\begin{multline}\label{e:bastaaaa}
 \Big(\frac{\MS(\overline{w},B^n_\rho)}{\rho^{n-1}}\wedge\frac{C(n)}{n-1}\Lambda_1^{2-n}\Big)+(n-1)c_1\rho\\
 \geq\lim_{\rho\downarrow 0}\Big( \Big(
 \frac{\MS(\overline{w},B^n_\rho)}{\rho^{n-1}}\wedge\frac{C(n)}{n-1}\Lambda_1^{2-n}\Big)+(n-1)c_1\rho\Big)
 \geq\omega_{n-1}\wedge \frac{C(n)}{n-1}\Lambda_1^{2-n}=:\beta_n.
 \end{multline}
In the last inequality we have used that $\mathscr{N}$ is a point of density one for $S_w$ and \eqref{e:lipconst}.

Given any $\rho\in(0,\frac{\beta_n}{2(n-1)c_1}\wedge\sfrac 12)$, inequality \eqref{e:bastaaaa} yields
\[
 \frac{\MS(\overline{w},B^n_\rho)}{\rho^{n-1}}\geq\frac{\beta_n}{2}.
\]
On the other hand, repeating the argument leading to the first and the second inequality in \eqref{e:testw0} 
imply for any $\rho\in(0,\frac{\beta_n}{2(n-1)c_1}\wedge \sfrac 12)$
\[
\frac{\MS(\overline{w},B^n_\rho)}{\rho^{n-1}}\leq(1+\tau_n\rho)\frac{\EE(v,\partial B_1)}{\rho^{n-1}}
\stackrel{\eqref{e:env}}{\leq}C\, (1+\tau_n\rho)\,(\rho\,\Lambda)^{1-n}.
\]
Thus, with fixed $\overline{\rho}\in(0,\frac{\beta_n}{2(n-1)c_1}\wedge \sfrac 12)$, we infer a contradiction 
from the last two estimates by choosing the constant $C=C(\overline{\rho})>0$ in \eqref{e:env} so that 
$C\, (1+\tau_n)\,(\overline{\rho}\,\Lambda)^{1-n}<\sfrac{\beta_n}2$.
\end{proof}

\subsection{The Mumford-Shah Conjecture}\label{ss:MSconj}

Having established the existence of strong (local) minimizers, it is elementary 
to infer that $u$ is harmonic on $\Om\setminus\SSu$. Hence, we will focus in the rest 
of the note on the regularity properties of the set $\Om\cap\SSu$ that will be instrumental also
to gain further regularity on $u$ itself (cf. Theorem~\ref{t:uSSu}).

The interest of the researchers in this problem is motivated by the Mumford and Shah conjecture 
(in $2$-dimensions) that we recall below for the readers' convenience.
\begin{conge}[Mumford and Shah \cite{MS89}]\label{c:MS1}
If $u\in\MM(\Om)$, $\Om\subseteq\R^2$, then $\Om\cap\SSu$ is the union of (at most) countably 
many injective $C^1$ arcs $\gamma_i: [a_i,b_i]\to \Omega$ with the following 
properties: 
\begin{itemize}
\item[(c1)] Any compact $K\subset \Omega$ intersects at most finitely 
many arcs;
\item[(c2)] Two arcs can have at most an endpoint $p$ in common, and if 
this is the case, then $p$ is in fact the endpoint of three arcs, forming 
equal angles of $\sfrac{2\pi}{3}$.
\end{itemize}
\end{conge}
So according to this conjecture only two possible singular configurations
occur: either three arcs meet in an end forming angles equal to $\sfrac{2\pi}{3}$, or an arc 
has a free end in $\Omega$. In what follows, we shall call \emph{triple junction} 
the first configuration and \emph{crack-tip} the second.

A suitable theory of calibrations for free discontinuity problems
established by Alberti, Bouchitt\'e and Dal Maso \cite{ABDM} shows that the model case of
triple junction functions, 
\begin{equation}\label{e:triplejunctions}
a\,\chi_{\{\vartheta\in(-\sfrac\pi6,\sfrac\pi2]\}}
+b\,\chi_{\{\vartheta\in(\sfrac\pi2,\sfrac{7\pi}6]\}}+c\,\chi_{\{\vartheta\in(\sfrac{7\pi}6,\sfrac{11\pi}6]\}}
\end{equation}
with $|a-b|\cdot|a-c|\cdot|b-c|>0$, is indeed a local minimizer (for more results on calibrations in the setting of  
free discontinuity problems see \cite{DMMM00,mora02,mora02b,moramorini01,morini02}).

Instead, Bonnet and David \cite{BD01} have shown that the model crack-tip functions, 
i.e.~functions that up to rigid motions can be written as  
\begin{equation}\label{e:crack-tipmodel}
C\pm\sqrt{\frac 2\pi \rho}\cdot\sin(\theta/2)
\end{equation}
for $\theta\in(-\pi,\pi)$, $\rho>0$ and some constant $C\in\R$,
are {\em global} minimizers of the Mumford and Shah energy\footnote{
The prefactor $\sqrt{\frac 2\pi}$ results from a simple calculation 
to ensure stationarity for a crack-tip function by plugging it in
the Euler-Lagrange equation \eqref{e:EL}. 
}.

More recently, second order sufficient conditions for minimality have been investigated. 
More precisely, Bonacini and Morini in \cite{BonMor} for suitable critical points, strictly 
stable and regular in their terminology, have proved that strict local minimality is implied by 
strict positivity of the second variation. 
This approach in the case of triple junctions is currently under investigation \cite{Cristo15}. 

Conjecture~\ref{c:MS1} has been first proven in some particular cases in the ensuing
weaker form. 
\begin{conge}\label{c:MS2}
If $u\in\MM(\Om)$, $\Om\subseteq\R^2$, then $\Om\cap\SSu$ is the union of (at most) countably 
many injective $C^0$ arcs $\gamma_i: [a_i,b_i]\to \Omega$ which are $C^1$ 
on $]a_i, b_i[$ and satisfy the two conditions of conjecture \ref{c:MS1}.
\end{conge}
The subtle difference between conjecture~\ref{c:MS1} and conjecture~\ref{c:MS2} above
is in the following point: assuming conjecture~\ref{c:MS2} holds, if 
$y_0=\gamma_i (a_i)$ is a ``loose end'' of the arc $\gamma_i$,  i.e.~it does not
belong to any other arc, then the techniques in \cite{B96} show that any 
\emph{blow~up limit},  i.e.~any limit of subsequences of the family $(u_{y_0,\rho})_\rho$ 
as in \eqref{e:blowup}, is a crack-tip but do not ensure the uniqueness of the limit 
itself (for more details on the notion of blow up see the proof of Proposition~\ref{p:bupchar}).

\subsection{Blow up analysis and the Mumford and Shah conjecture}\label{ss:bup}

The Mumford and Shah conjecture, in the form stated in conjecture~\ref{c:MS2}, has been 
attacked first in the contribution by Bonnet \cite{B96}. Bonnet's approach is based on a weaker notion 
of minimality for the strong formulation of the problem, that includes a topological condition to be 
satisfied by the competitors, though still sufficient to develop a regularity theory. 
\begin{definition}\label{d:Bmin}
Let $\Om\subseteq\R^2$ be an open set, a pair $(u,K)$, $K\subset\Om$ closed and 
$u\in W^{1,2}_{\mathrm{loc}}\big(\Om\setminus K\big)$, is a Bonnet minimizer of 
$\mathscr{F}$ in $\Om$ if $\mathscr{F}(u,K,\Om)\leq \mathscr{F}(v,L,\Om)$ among all couples 
$(v,L)$, $L\subset\Om$ closed and $v\in W^{1,2}_{\mathrm{loc}}\big(\Om\setminus L\big)$,  
such that there is a ball $B_\rho(x)\subset\Om$  for which
\begin{itemize}
\item[(i)] $u=v$ and $K=L$ on $\Om\setminus\overline{B_\rho(x)}$,
\item[(ii)] any pair of points in $\Om\setminus\big(K\cup\overline{B_\rho(x)}\big)$ that lie in different 
connected components of $\Om\setminus K$ are also in different connected components of $\Om\setminus L$.
\end{itemize}
Moreover, in case $\Om=\R^2$, we say that $(u,K)$ is a global minimizer of $\mathscr{F}$.
\end{definition}
In particular, under the assumption that $\Om\cap\SSu$ has a finite number of connected 
components Bonnet has classified all the blow~up limits of minimizers as in Definition~\ref{d:Bmin}, 
and then also of local minimizers, as
\begin{itemize}
 \item[(i)] constant functions,
 \item[(ii)] pure jumps (cf. \eqref{e:purejumps}), 
 \item[(iii)] triple junction functions (cf. \eqref{e:triplejunctions}),
 \item[(iv)] crack-tip functions (cf. \eqref{e:crack-tipmodel}),
\end{itemize}
establishing conjecture~\ref{c:MS2} in such a restricted framework. 
In particular, as already mentioned, Bonnet's result does not 
deal with the stronger conjecture~\ref{c:MS1} as it cannot exclude the possibility that $\gamma_i$ 
``spirals'' around $y_0$ infinitely many times (compare with the discussion at the end 
of \cite[Section 1]{B96}). 
The method of Bonnet relies on a key monotonicity formula for the rescaled Dirichlet energy
that so far has no counterpart in general. Recently, Lemenant \cite{Lem14}
has exhibited another monotone quantity that plays the role of the rescaled Dirichlet energy in 
higher dimensions in case $\Om\cap\SSu$ is contained in a sufficiently smooth cone. In view of this,  
a rigidity result for $\Om\cap\SSu$ in the $3$-dimensional case can be deduced.
Let us also point out that the almost monotonicity formula established by Bucur and Luckhaus in 
Theorem~\ref{t:monform} is of little use in the blow~up analysis due to the truncation with the constant 
$C(n)$. 

The contributions of L\'eger \cite{Leg99} and of David and L\'eger \cite{DavLeg02} improve 
upon Bonnet's results: the former addressing the case of $\Om\cap\SSu$ satisfying a suitable 
flatness assumption, the latter identifying pure jumps and triple junction functions as the only minimizers 
in the sense of Bonnet for which $\R^2\setminus\SSu$ is not connected.
All these efforts are directed to push forward Bonnet's ideas. Indeed, \cite[Proposition~71.2]{Dav05} shows 
in general, i.e.~with no extra connectedness assumptions on $\Om\cap\SSu$, that the complete classification of 
the blow~up limits of local minimizers as in the list above turns out to be a viable strategy to establish 
conjecture~\ref{c:MS1}. 
More precisely, coupling the latter piece of information with a detailed local description of the geometry 
of $\Om\cap\SSu$, that is the topic of the ensuing subsection~\ref{ss:epsregSu}, would yield the conclusion. 
A further interesting consequence of the analysis in the paper \cite{DavLeg02} is that the Mumford and Shah 
conjecture turns out to be equivalent to the uniqueness (up to rotations and translations) of crack-tips 
as global minimizers of the $\MS$ energy as conjectured by De Giorgi in \cite{DeG89}.

However, we shall not enter here into this streamline of results but rather refer for more details on them 
to the monograph \cite{Dav05} and to the recent review paper \cite{Lem15}.

Instead, a different (but related) perspective to the goal of understanding conjecture~\ref{c:MS1} is taken 
in what follows. We shall link the latter to a sharp higher integrability property of the approximate 
gradient following Ambrosio, Fusco and Hutchinson \cite{AFH}. 
In order to do this, in the next section we review the state of the art about the regularity properties of 
$\Om\cap\SSu$.

\section{Regularity of the jump set}\label{s:regSu}

The aim of this section is to survey on the regularity of $\Om\cap\SSu$. In subsection~\ref{ss:epsregSu} we shall 
first recall classical and more recent $\eps$-regularity results, and then state an estimate on the size of 
the subset of singular points in $\Om\cap\SSu$ that will be dealt with in details in Section~\ref{s:cpt}.
In particular, in subsection~\ref{ss:HI} the links of the higher integrability of the gradient with the 
Mumford and Shah conjecture~\ref{c:MS1} will be highlighted following the approach of Ambrosio, Fusco and Hutchinson 
\cite{AFH}. A slight improvement of the latter ideas leads to an energetic characterization of the 
conjecture~\ref{c:MS2} as exposed in subsection~\ref{ss:enconj}.

\subsection{$\eps$-regularity theorems}\label{ss:epsregSu}

The starting point to address the regularity of $\Om\cap\SSu$ for local minimizers is the 
ensuing $\eps$-regularity result. 
\begin{theorem}[Ambrosio, Fusco and Pallara \cite{AFP97}]\label{t:AFP97}
Let $u\in\MM(\Om)$, then there exists $\Sigma_u\subset\Om\cap\SSu$ 
relatively closed in $\Om$ with $\HH^{n-1}(\Sigma_u)=0$, and such that
$\Om\cap\SSu\setminus\Sigma_u$ is locally a  $C^{1,\gamma}$ hypersurface for 
all $\gamma\in(0,1)$ and $C^{1,1}$ if $n=2$. 

More precisely, there exist $\eps_0=\eps_0(n),\,\rho_0=\rho_0(n)>0$ such that 
\begin{equation}\label{e:sigma}
\Sigma_u=\{z\in\Om\cap\SSu:\,\mathscr{D}(z,\rho)+\mathscr{A}(z,\rho)
\ge\eps_0\quad\forall\rho\in(0,\rho_0\wedge \mathrm{dist}(z,\partial\Omega))\}
\end{equation}
where 
\[
\mathscr{D}_u(z,\rho):=\rho^{1-n}\int_{B_\rho(z)}|\nabla u|^2dy,
\quad\text{(scaled Dirichlet energy)} 
\]
\[
\mathscr{A}_u(z,\rho):=\rho^{-n-1}\min_{T\in\Pi}\int_{S_u\cap B_\rho(z)}
\mathrm{dist}^2(y,T)d\HH^{n-1}(y),
\quad\text{(scaled mean flatness)}, 
\]
with $\Pi$ the class of $(n-1)$-affine planes.
\end{theorem}
For more details on Theorem~\ref{t:AFP97} we refer to \cite[Chapter~8]{AFP00}, that is 
entirely devoted to the proof of it, and to \cite{Fuscosur} for a hint of the strategy 
of proof. Here we shall only comment on the quantities involved in \eqref{e:sigma}.

First note that the affine change of variables mapping $B_\rho(x)$ into $B_1$ shows that 
$\mathscr{D}_u(x,\cdot)$ and $\mathscr{A}_u(x,\cdot)$ are equal to the Dirichlet energy 
and the mean flatness on $B_1$ of the rescaled maps $u_{x,\rho}$ in \eqref{e:blowup}, respectively. 

Further, the scaled mean flatness measures in an average sense the deviation of 
$\Om\cap\SSu$ from being flat in $z$. For instance, if $\Om\cap\SSu$ is a $C^1$ hypersurface in a 
neighborhood of $z$ it is easy to check that $\mathscr{A}_u(z,\rho)=o(1)$ as 
$\rho\downarrow 0$. Actually , the density lower bound in Theorem~\ref{t:CL} 
and the density upper bound in Proposition~\ref{p:dub} allow us to show that 
the rescaled mean flatness $\mathscr{A}_u$ is equivalent to its $L^\infty$ version, 
namely
\[
\mathscr{A}_{u,\infty}(z,\rho):= \rho^{-1}\min_{T\in\Pi}\sup_{y\in S_u\cap B_\rho(z)}
\mathrm{dist}(y,T).
\]
\begin{proposition}
 Let $u\in\MM(\Om)$, then for all $z\in\SSu$ and $B_\rho(z)\subset\Om$
 we have
 \[
\frac{\theta_0}{2^{n+1}}\,\mathscr{A}_{u,\infty}^{n+1}(z,\sfrac\rho 2)\leq 
\mathscr{A}_u(z,\rho)\leq n\,\omega_n\,\mathscr{A}_{u,\infty}^2(z,\rho).
 \]
 where $\theta_0$ is the constant in Theorem~\ref{t:CL}.
\end{proposition}
\begin{proof}
 The estimate from above easily follows by the very definitions of $\mathscr{A}_u$ and 
 $\mathscr{A}_{u,\infty}$ by taking into account \eqref{e:dub}.

 Let then $\bar{T}$ be the affine hyperplane through $z$ giving the minimum in the definition 
 of $\mathscr{A}_{u}(z,\rho)$. If $\bar{y}\in S_u\cap B_{\sfrac\rho2}(z)$ is a point of almost 
 maximum distance from $\bar{T}$, i.e. for a fixed $\delta\in(0,1)$ 
 \[
  d:=\mathrm{dist}(\bar{y},\bar{T})\geq(1-\delta)
  \sup_{y\in S_u\cap B_{\sfrac\rho 2}(z)}\mathrm{dist}(y,\bar{T}),
 \]
then we can estimate as follows thanks to \eqref{e:dgcl}
\begin{multline*}
 \mathscr{A}_{u}(z,\rho)\geq \rho^{-n-1}\int_{S_u\cap B_{\sfrac d2}(\bar{y})}
\mathrm{dist}^2(y,\bar{T})d\HH^{n-1}(y)\geq \rho^{-n-1}\,\frac{d^2}4\,\HH^{n-1}\big(S_u\cap B_{\sfrac d2}(\bar{y})\big)\\
\geq \rho^{-n-1}\,{\theta_0}\,\Big(\frac{d}{2}\Big)^{n+1}
\geq (1-\delta)^{n+1}\frac{\theta_0}{2^{n+1}}\mathscr{A}_{u,\infty}^{n+1}(z,\sfrac\rho2).
\end{multline*}
The conclusion follows at once as $\delta\downarrow 0$.
\end{proof}

The local graph property of $\Om\cap\SSu$ established in Theorem~\ref{t:AFP97}, 
the Euler-Lagrange condition and the regularity theory for elliptic PDEs 
with Neumann boundary conditions determine the regularity of $u$ close to $\Om\cap\SSu$ 
(see \cite[Theorem 7.49]{AFP00} or \cite[Proposition 17.15]{Dav05} if $n=2$).
\begin{theorem}[Ambrosio, Fusco and Pallara \cite{AFP99}]\label{t:uSSu}
 Let $u\in\MM(\Om)$ and $A\cap\SSu$, $A\subset\Om$ open, be the graph 
 of a $C^{1,\gamma}$ function $\phi$, $\gamma\in(0,1)$. Then, $\phi\in C^{\infty}$ 
and $u$ has $C^\infty$ extension on each side of $A\cap\SSu$. 
\end{theorem}
Actually, Koch, Leoni and Morini \cite{KoLeoMor} proved that if $A\cap\SSu$ is 
$C^{1,\gamma}$ then it is actually analytic, as conjectured by De Giorgi (cf. 
\cite{DeG90, DeG91}).

Going back to the regularity issue for $\SSu$ we resume below the outcomes of a different approach developed 
by David in the $2$-dimensional case. In this setting it is also possible to address the situation in which 
$\Om\cap\SSu$ is close in the Hausdorff distance to a triple-junction 
(cf. Section~\ref{s:porosity} for more comments in this respect). 
\begin{theorem}[David~\cite{Dav96}, Corollary~51.17 and Theorem~53.4 \cite{Dav05}]\label{t:epsreg}
There exists $\varepsilon>0$ and an absolute constant $c\in(0,1)$ 
with the following properties. If $u\in\MM(\Om)$, $z\in \SSu$, 
$B_r (z)\subset \Omega$ and $\mathscr{C}$ is either a line or a triple junction such that 
\begin{equation}\label{e:eps}
\int_{B_{r}(z)}|\nabla u|^2\,dx+
\mathrm{dist}_{\HH}(\SSu\cap B_{r}(z),\mathscr{C}\cap 
B_{r}(z))\leq\varepsilon\,r,
\end{equation}
then there exists a $C^{1}$-diffeomorphism $\phi$ of
$B_{r}(z)$ onto its image with
\[
\SSu\cap B_{cr}(z)=\phi\left(\mathscr{C}\right)\cap B_{cr}(z).
\]
In addition, for any given $\delta\in(0,1/2)$, there is $\eps > 0$ such that, 
if \eqref{e:eps} holds, then 
$\SSu\cap (B_{(1-\delta)r} (z)\setminus B_{\delta r} (z))$
is $\delta$-close, in the $C^1$ norm, to 
$\mathscr{C}\cap (B_{(1-\delta)r}(z)\setminus B_{\delta r} (z))$.
\end{theorem}
\begin{remark}\label{r:epsreg} 
The last sentence of Theorem~\ref{t:epsreg} is not contained in 
\cite[Corollary~51.17, Theorem~53.4]{Dav05}.
However it is a simple consequence of the theory developed in there. 
By scaling, we can assume $r=1$ and $x=0$. Fix a cone
$\mathscr{C}$, a $\delta > 0$ and 
a sequence $\{u_k\}\subset \mathcal{M} (B_1)$ for which the left hand side of 
\eqref{e:eps} goes to $0$.
If $\mathscr{C}$ is a segment, then it follows from \cite{Dav05}
(or \cite{AFP00}) that there are uniform $C^{1,\alpha}$ bounds on 
$\overline{S_{u_k}}\cap B_{1-\delta}$.
We can then use the Ascoli-Arzel\`a Theorem to conclude that 
$\overline{S_{u_k}}$ is converging in $C^1$ to $\mathscr{C}$.

In case the minimal cone $\mathscr{C}$ is a triple junction, then observe that 
$\mathscr{C}\cap (B_1 \setminus B_{\delta/2})$ consists of a three distinct 
segments at distance $\delta/2$ from each other. Covering each of these segments with balls of
radius comparable to $\delta$ and centered in a point belonging to the segment itself, we can argue as above
and conclude that, for $k$ large enough, $\overline{S_{u_k}}\cap 
(B_{1-\delta}\setminus B_\delta)$ consist of three arcs,
with uniform $C^{1,\alpha}$ estimates.  Once again the Ascoli-Arzel\`a Theorem shows that $\overline{S_{u_k}}\cap (B_{1-\delta}\setminus B_\delta)$ is converging 
in $C^1$ to $\mathscr{C}\cap (B_{1-\delta} \setminus B_\delta)$.
\end{remark}
\begin{remark}
 Actually assumption \eqref{e:eps} can be relaxed to
 \[
  \int_{B_{r}(x)}|\nabla u|^2\,dx\leq\varepsilon\,r
 \]
(see \cite[Proposition~60.1]{Dav05} or Theorem~\ref{t:cptness} below).
\end{remark}
Remarkably, Lemenant \cite{lemenant} extended such a result to the $3$-dimensional 
case with suitable changes in the statement (see also \cite{Lem15} for a sketch 
of the proof).

The techniques developed by David are also capable to describe in details the structure 
of $\SSu$ around points that corresponds to the model case of crack-tips despite uniqueness 
of blow~ups is not ensured.
\begin{theorem}[David, Theorem~69.29 \cite{Dav05}]\label{t:David-zurrione}
For all $\eps_0>0$ there exists $\varepsilon>0$ such that if $u\in\MM(\Omega)$ and 
\[
 \mathrm{dist}_{\HH}(\SSu\cap B_{r}(z),\sigma)<\eps\,r
\]
for some radius $\sigma$ of $B_r(z)\subset\Omega$, then $\SSu \cap B_{\sfrac r2}(z)$ 
consists of a single connected arc which joins some point $y_0\in B_{\sfrac r4}(z)$ with 
$\partial B_{\sfrac r2}(z)$ and which is smooth in $B_{\sfrac r2}(z)\setminus \{y_0\}$. 

More precisely, there is a point $y_0\in B_{\sfrac r{4}}(z)$ such that 
\[
\SSu \cap B_{\sfrac r2} (z)=\big\{y_0+ \rho (\cos \alpha (\rho), \sin \alpha (\rho))\big\}
\]
for some smooth function $\alpha: (0,\sfrac r2)\to \R$ which satisfies
\begin{equation}
\rho|\alpha'(\rho)|\leq \eps_0 \quad \mbox{for all $\rho\in (0,\sfrac r2)$,}
\qquad\lim_{\rho\downarrow 0}\rho|\alpha'(\rho)|=0.
\end{equation}
In addition, there is a constant $C$ such that, up to a change of sign, 
\[
 u\big(y_0+ \rho (\cos \theta, \sin \theta)\big)=
 \sqrt{\frac{2}{\pi}\rho}\,\sin\Big(\frac{\theta-\alpha(\rho)}{2}\Big)+C+\rho^{\sfrac 12}\omega(\rho,\theta)
\]
for all $\rho\in (0,\sfrac r2)$ and $\theta\in(\alpha(\rho)-\pi,\alpha(\rho)+\pi)$, with
$\lim_{\rho\downarrow 0}\sup_\theta|\omega(\rho,\theta)|=0$. Finally,
\[
 \lim_{\rho\downarrow 0}\frac{1}{\rho}\int_{B_{\rho}(y_0)}|\nabla u|^2dx=
 \lim_{\rho\downarrow 0}\frac{1}{\rho}\,\HH^1(\SSu\cap B_{\rho}(y_0))=1.
 \]
\end{theorem}
As remarked above, the latter theorem does not guarantee that such arc is $C^1$ {\em up to the loose end} 
$y_0$: in particular it leaves the possibility that the arc spirals infinitely many times around it.
Hence, it does not establish the uniqueness of the blowup in the point $y_0$.

\subsection{Higher integrability of the gradient and the Mumford and Shah conjecture}\label{ss:HI}

Theorem~\ref{t:AFP97}, or better the characterization of the
{\em singular set} $\Sigma_u$ in \eqref{e:sigma}, can be employed 
to subdivide $\Sigma_u$ as follows: $\Sigma_u=\Sigma_u^{(1)}\cup\Sigma_u^{(2)}\cup\Sigma_u^{(3)}$,
where 
\begin{eqnarray*}
&&\Sigma_u^{(1)}:=\{x\in\Sigma_u:\,\lim_{\rho\downarrow 0} 
\mathscr{D}_u(x,\rho)=0\},\\
&&\Sigma_u^{(2)}:=\{x\in\Sigma_u:\,\lim_{\rho\downarrow 0} 
\mathscr{A}_u(x,\rho)=0\},\\
&&\Sigma_u^{(3)}:=\{x\in\Sigma_u:\,
\liminf_{\rho\downarrow 0}\mathscr{D}_u(x,\rho)>0,\,
\liminf_{\rho\downarrow 0}\mathscr{A}_u(x,\rho)>0\}.
\end{eqnarray*}
According to the Mumford and Shah conjecture~\ref{c:MS1} we should have $\Sigma_u^{(3)}=\emptyset$ if 
$n=2$. Furthermore, being inspired by the $2$d case, we shall refer to $\Sigma_u^{(1)}$ as the set of 
triple junctions, and to $\Sigma_u^{(2)}$ as the set of crack-tips.
Actually, in $2$-dimensions we will fully justify the latter terminology in Proposition~\ref{p:bupchar}
(see also Remark~\ref{r:triplen2}).

In the general $n$-dimensional setting De Giorgi conjectured that
$\HH^{n-2}(\Sigma_u\cap\Omega^\prime)<\infty$ for all $\Omega^\prime\subset\hskip-0.125cm\subset\Omega$
(cf. \cite[conjecture~6]{DeG91}).
In particular, the validity of the latter conjecture would imply $\dimh\Sigma_u\leq n-2$. 

A first breakthrough in this direction has been obtained by 
Ambrosio, Fusco and Hutchinson in \cite{AFH}.
\begin{theorem}[Ambrosio, Fusco and Hutchinson, \cite{AFH}]\label{t:AFH2}
For every $u\in\MM(\Om)$ 
\[
\dimh \Sigma_u^{(1)}\leq n-2.
\]
\end{theorem}
Actually, the set $\Sigma_u^{(1)}$ turns out to be countable in $2$-dimensions (see 
Remark~\ref{r:triplen2} below).

In the same paper \cite{AFH}, Ambrosio, Fusco and Hutchinson investigated the 
connection between the higher integrability of $\nabla u$ and the 
Mumford and Shah conjecture. 

If conjecture~\ref{c:MS1} does hold, then $\nabla u\in L^p_{loc}$ 
for all $p<4$ (cf. with \cite[Proposition~6.3]{AFH} under $C^{1,1}$ 
regularity assumptions on $\SSu$, see also 
Theorem~\ref{t:MSconj} below). 
It was indeed conjectured by De Giorgi in all space dimensions that 
$\nabla u\in L^p_{loc}$ for all $p<4$ (cf. with \cite[conjecture 1]{DeG91}).
So far only a first step into this direction has been established. 
\begin{theorem}\label{t:hi}
There is $p>2$ such that $\nabla u\in L^p_{\mathrm{loc}}(\Om)$ for all $u\in\MM(\Om)$ and for 
all open sets $\Om\subseteq\R^n$.
\end{theorem}
A proof of Theorem~\ref{t:hi} in $2$-dimension has been given by De Lellis and Focardi in \cite{DF1}, 
shortly after De Philippis and Figalli established the result without any dimensional limitation in 
\cite{DePFig14}.
The exponent $p$ in both papers is not explicitly computed despite some suggestions to do that 
are also proposed.

The higher integrability can be translated into an estimate for the size of the 
singular set $\Sigma_u$ of $\SSu$ (see \cite[Corollary 5.7]{AFH}) that improves 
upon the conclusion of Theorem~\ref{t:AFP97} (cf. \cite{Min03}, \cite{Min03b}, 
\cite{KrMin06} and the survey \cite{Min08} for related issues for minima of 
variational integrals and for solutions to nonlinear elliptic systems in divergence form).
\begin{theorem}[Ambrosio, Fusco and Hutchinson \cite{AFH}]\label{t:AFH1}
If $u\in\MM(\Om)$ and $|\nabla u|\in L^p_{{loc}}(\Om)$ for some $p>2$, then 
\begin{equation}\label{e:stimaSigmau}
\mathrm{dim}_\HH\Sigma_u\leq\max\{n-2,n-\sfrac p2\}\in(0,n-1).
\end{equation}
 
\end{theorem}
The estimate $\mathrm{dim}_\HH\Sigma_u<n-1$ has also been established by David \cite{Dav96} 
for $n=2$, and lately by Rigot \cite{Rig00} and by Maddalena and Solimini \cite{MadSol01} 
in general by establishing the porosity of $\Om\cap\SSu$. Despite this, it was not related to the 
higher integrability property of the gradient. In Section~\ref{s:porosity} below we shall 
comment more in details on how porosity implies such an estimate and moreover on how porosity 
can be employed to prove the higher integrability of the gradients of minimizers following 
De Philippis and Figalli \cite{DePFig14} (see Section~\ref{s:hi}). 

For the time being few remarks are in order:
\begin{itemize}
\item[(i)] the upper bound $p<4$ is motivated not only because we need the 
rhs in the estimate \eqref{e:stimaSigmau} to be positive if $n=2$, but also because explicit 
examples show that it is the best exponent one can hope for: consider in $2$ 
dimensions a crack-tip minimizer (Bonnet and David \cite{BD01}), i.e.~a function 
that up to a rigid motion can be written as 
\[
u(\rho,\theta)=C\pm \sqrt{\frac 2\pi \rho}\cdot\sin(\theta/2)
\] 
for $\theta\in(-\pi,\pi)$ and $\rho>0$, and some constant $C\in\R$.
Simple calculations imply that crack-tip minimizers  satisfy 
\[
|\nabla u|\in L^p_{{loc}}\setminus L^4_{{loc}}(\R^2)\quad
\text{for all }p<4;
\]

\item[(ii)] If we were able to prove the higher integrability property
for every $p<4$ then we would infer that $\mathrm{dim}_\HH\Sigma_u\leq n-2$, 
and actually in $2$-dimensions $\mathrm{dim}_\HH\Sigma_u=0$. 
Clearly, this would be a big step towards the solution in positive of the Mumford and 
Shah conjecture. For further progress in this direction see Theorem~\ref{t:MSconj} below.
\end{itemize}
Given Theorems~\ref{t:AFH2} and \ref{t:hi} for granted, Theorem~\ref{t:AFH1} is a simple consequence 
of soft measure theoretic arguments. We shall establish Theorem~\ref{t:AFH2} in Section~\ref{s:cpt}. 
Instead, here we prove Theorem~\ref{t:AFH1} to underline the role of the higher 
integrability that, in turn, shall be established in Sections~\ref{s:hi2} and \ref{s:hi} following 
two different paths. 
\begin{proof}[Proof of Theorem~\ref{t:AFH1}]
Suppose that $|\nabla u|\in L^p_{loc}(\Om)$ for some $p>2$, then for all $s\in(n-\sfrac p2,n-1)$ 
the set 
\[
\Lambda_s:=\left\{x\in\Om:\,
\limsup_\rho\rho^{-s}\int_{B_\rho(x)}|\nabla u|^pdy=\infty\right\}
\]
satisfies $\HH^s(\Lambda_s)=0$ by elementary density estimates for the Radon measure
\[
\mu(A):=\int_{A}|\nabla u|^pdy\quad \text{$A\subseteq\Om$ open subset}.
\]
Indeed, 
for all $\delta>0$, Proposition~\ref{p:strongweak} gives that
\[
\delta\,\HH^s(\Lambda_s)\leq \mu(\Lambda_s)\leq\mu(\Om)<\infty. 
\]
Therefore, $\HH^s(\Lambda_s)=0$. 

Hence, if we rewrite $\Sigma_u$ as the disjoint union of $\Sigma_u\cap \Lambda_s$
and of $\Sigma_u\setminus\Lambda_s$, we deduce $\HH^s(\Sigma_u\cap \Lambda_s)=0$ 
and thus the estimate $\mathrm{dim}_\HH(\Sigma_u\cap \Lambda_s)\leq s$.
 
Furthermore, it is easy to prove that $\Sigma_u\setminus\Lambda_s\subseteq\Sigma_u^{(1)}$. 
If $x\in\Sigma_u\setminus\Lambda_s$ by the higher integrability 
and H\"older inequality it follows that
\[
\mathscr{D}_u(x,\rho)=\rho^{1-n}\int_{B_\rho(x)}|\nabla u|^2dy
\leq\, \omega_n^{1-\frac 2p}\,\rho^{\frac 2p\big(s-n+\frac p2\big)}
\left(\rho^{-s}\int_{B_\rho(x)}|\nabla u|^pdy\right)^{\frac 2p}
\stackrel{\rho\downarrow 0^+}{\longrightarrow}0,
\]
since $s>n-\sfrac p2$. 
By taking into account Theorem~\ref{t:AFH2} we have that 
$\mathrm{dim}_\HH(\Sigma_u\setminus \Lambda_s)\leq n-2$.

In conclusion, we infer that for all $s\in(n-\sfrac p2,n-1)$
\[
\mathrm{dim}_\HH\Sigma_u=\max\{\mathrm{dim}_\HH(\Sigma_u\cap \Lambda_s), 
\mathrm{dim}_\HH(\Sigma_u\setminus\Lambda_s)\}\leq \max\{n-2,s\},
\] 
by letting $s\downarrow(n-\sfrac p2)$ we are done.
\end{proof}

\subsection{An energetic characterization of the Mumford and Shah conjecture~\ref{c:MS2}}\label{ss:enconj}

Let us go back to the upper bound $p<4$ in the higher integrability result. 
We consider again the crack-tip function in item (i) after Theorem~\ref{t:AFH1}. 
A simple calculation shows that its gradient belongs to $L^p_{loc}\setminus L^4_{loc}(\R^2)$ 
for all $p<4$. Beyond the scale of $L^p$ spaces something better holds true: 
$|\nabla u|\in L^{4,\infty}_{{loc}}(\R^2)$. The latter is a weak-Lebesgue 
space, i.e.~if $U\subseteq\R^2$ is open then $f\in L^{4,\infty}_{{loc}}(U)$ 
if and only if for all $U^\prime\subset\hskip-0.125cm\subset U$ there exists $K=K(U^\prime)>0$ 
such that 
\[
\LL^2\big(\{x\in U^\prime:\,|f(x)|>\lambda\}\big)\leq K\lambda^{-4}\quad\text{ for all }
\lambda>0.
\]
As a side effect of the considerations in \cite{DF1} 
one deduces an energetic characterization of the modified Mumford and Shah conjecture~\ref{c:MS2} 
(see \cite[Proposition~5]{DF1}). 
Indeed, the validity of the latter 
is equivalent to a sharp integrability property of the gradient of the minimizers.
\begin{theorem}[De Lellis and Focardi \cite{DF1}]\label{t:MSconj}
If $\Om\subseteq\R^2$, conjecture~\ref{c:MS2} is true for $u\in\MM(\Om)$
if and only if $\nabla u\in L^{4,\infty}_{loc}(\Om)$. 
\end{theorem}
The characterization above would hold for the original Mumford and Shah conjecture~\ref{c:MS1} if 
$C^{1}$ regularity of $\SSu$ up to crack-tip points had been established.

To prove Theorem~\ref{t:MSconj} we need the following preliminary observation.
\begin{lemma} 
Let $f\in L^{4,\infty}_{loc}(\Om)$, $\Om\subseteq\R^2$, then for all $\e>0$ the set
\begin{equation}\label{e:ctip}\
D_\e:=\left\{x\in\Om:\,\liminf_r\frac 1r\int_{B_r(x)}f^2(y)\,dy\ge\e\right\}
\end{equation}
is locally finite.
\end{lemma}
\begin{proof}
We shall show in what follows that if $f\in L^{4,\infty}(\Om)$ then $D_\e$
is finite, an obvious localization argument then proves the general case.

Let $\e>0$ and consider the set $D_\e$ in \eqref{e:ctip} above.
First note that, for any $B_r (x)\subset \Omega$ and any $\lambda>0$ we have the estimate
\begin{align}\label{e:stimasublev}
\int_{\{y\in B_r(x):\,|f(y)|\geq \lambda\}}f^2(y)\,dy&\leq
\int_{\{y\in\Om:\,|f(y)|\geq \lambda\}}f^2(y)\,dy\notag\\
&=2\int_\lambda^{+\infty}t\,\LL^2\big(\{y\in\Om:\,|f(y)|\geq t\}\big)dt
\leq\int_\lambda^{+\infty}\frac{2K}{t^3}dt=\frac K{\lambda^2}.
\end{align}
If $x\in D_\e$ and $r>0$ satisfy 
\begin{equation}\label{e:stimamisura}
\int_{B_r(x)}f^2(y)\,dy\geq \frac \e2 r,
\end{equation}
choosing $\lambda=2(K/r\e)^{1/2}$ in \eqref{e:stimasublev} we conclude 
\begin{equation}\label{e:stimasuplev}
\int_{\{y\in B_r(x):\,|f(y)|<\,2( \frac K{r\e})^{1/2}\}}f^2(y)\,dy\geq
\frac \e4 r.
\end{equation}
Furthermore, the trivial estimate
\[
\int_{\{y\in B_r(x):\,|f(y)|<\lambda\}}f^2(y)\,dy<\pi\lambda^2r^2,
\]
implies for $\lambda=(\e/8\pi r)^{1/2}$
\begin{equation}\label{e:stimasublev1}
\int_{\{y\in B_r(x):\,|f(y)|<(\frac \e{8\pi r})^{1/2}\}}f^2(y)\,dy<\frac\e 8r.
\end{equation}
By collecting \eqref{e:stimasuplev} and \eqref{e:stimasublev1}
we infer
\[
\int_{\{y\in B_r(x):\,(\frac \e{8\pi r})^{1/2}\leq
|f(y)|<\,2(\frac K{r\e})^{1/2}\}}f^2(y)\,dy\geq\frac\e 8r,
\]
that in turn implies
\begin{equation}\label{e:stimafond}
\LL^2\big(\{y\in B_r(x):\,|f(y)|\geq (\frac \e{8\pi r})^{1/2}\}\big)
\geq\frac{\e^2r^2}{32K}.
\end{equation}
Let $\{x_1,\ldots,x_N\}\subseteq D_\e$ and $r>0$ be a radius such 
that the balls $B_r(x_i)\subseteq\Om$ are disjoint and 
\eqref{e:stimamisura} holds for each $x_i$. Then, from \eqref{e:stimafond} 
and the fact that $f\in L^{4,\infty}(\Om)$, we infer
\[
N\frac{\e^2r^2}{32K}\leq \LL^2\big(\{y\in\Om:\,
|f(y)|\geq(\frac \e{8\pi r})^{1/2}\}\big)
\leq\frac {K(8\pi r)^2}{\e^2}\Longrightarrow 
N\leq\frac {2^{11}K^2\pi^2}{\e^4},
\]
and the conclusion follows at once.
\end{proof}

We are now ready to give the proof of Theorem~\ref{t:MSconj}.
The ``if'' direction is achieved by first proving that $\SSu$ has locally finitely many 
connected components and then invoking the regularity theory developed by Bonnet \cite{B96}. 
In turn, the proof that the connected components are locally finite 
is a fairly simple application of David's $\eps$-regularity theory (see Theorem~\ref{t:epsreg}).
Vice versa, the ``only if'' direction is proved by means of Bonnet blow~up analysis
and standard elliptic regularity theory.
\begin{proof}[Proof of Theorem~\ref{t:MSconj}.]
To prove the direct implication we assume without loss of generality 
that $\Om=B_R$ for some $R>1$, being the result local. In addition, 
we may also suppose that $\SSu\cap\partial B_1=\{y_1,\ldots,y_M\}$.
Theorem~\ref{t:epsreg} and Theorem~\ref{t:cptness} in Section~\ref{s:cpt} below yield 
that there exists some $\e_0>0$ such that for all points 
$x\in B_R\setminus D_{\e_0}$ the set $\SSu\cap B_r(x)$ is 
either empty or diffeomorphic to a minimal cone, for some $r>0$. 
In particular, in the latter event $B_r(x)\setminus\SSu$ 
is not connected.

Supposing that $D_{\e_0}\cap B_1=\{x_1, \ldots, x_N\}$, and setting
\[
\Omega_k:= B_{1-1/k}\setminus \bigcup_{i=1}^N B_{1/k}(x_i)\,,
\]
a covering argument and the last remark give that for every 
$x\in\Omega_k\cap\SSu$ there 
is a continuous arc $\gamma_k:[0,1]\to\SSu$ 
with $\gamma_k(0)=x$ and $\gamma_k(1)=y\in\partial\Om_k$.
Then, the sequence $(\widetilde{\gamma}_k)_{k\in\N}$ of 
reparametrizations of the $\gamma_k$'s by arc length converges 
to some arc $\gamma:[0,1]\to\SSu$ with $\gamma(0)=x$ 
and $\gamma(1)\in\{x_1,\ldots,x_N,y_1,\ldots,y_M\}$.

From this, we deduce that $\overline{B_1}\cap\SSu$ has 
a finite number of connected components.
Bonnet's regularity results \cite[Theorems 1.1 and 1.3]{B96} 
then provide the thesis.

To conclude we prove the opposite implication. To this aim we 
consider $\Om^\prime\subset\hskip-0.125cm\subset\Om^{\prime\prime}\subset\hskip-0.125cm\subset\Om$ 
and suppose that $\SSu\cap\Om^{\prime\prime}$ is a finite 
union of $C^1$ arcs of finite length. Denote by $\{x_1,\ldots,x_N\}$ 
the end points of the arcs in $\Om^\prime$ and let $r>0$ be such that 
$B_{4r}(x_i)\subseteq\Om^\prime$ for all $i$, and 
$B_{4r}(x_i)\cap B_{4r}(x_j)=\emptyset$ if $i\neq j$.
Theorem~\ref{t:uSSu} implies 
that $\nabla u$ has a $C^{0,\alpha}$ extension on both sides of 
$(\Om^{\prime\prime}\cap\SSu)\setminus\cup_i\overline{B_{r}(x_i)}$
for all $\alpha<1$. In particular, $\nabla u$ is bounded on 
$\overline{\Om^\prime}\setminus\cup_iB_{2r}(x_i)$.

Next consider the sequence $r_k=r/2^{k-1}$, $k\geq 0$, and fix 
$i\in\{1,\ldots,N\}$. Then, by \cite[Proposition 37.8]{Dav05} 
(or \cite[Theorem 2.2]{B96}) we can extract a subsequence 
$k_j\uparrow\infty$ along which the blow~up functions
$u_{j}(x):=r_{k_j}^{-1/2}\big(u(x_i+r_{k_j}x)-c_{j}(x)\big)$ converge to some  
$w$ in $W^{1,2}_{loc}(B_4\setminus K)$, for some piecewise constant 
function $c_{j}:\Om\setminus\overline{S_{u_{j}}}\to\R$, and 
$(\overline{S_{u_{j}}})_{j\in\N}$ converges to some set $K$ in the 
Hausdorff metric.

By Bonnet's blow~up theorem \cite[Theorem 4.1]{B96} only two possibilities 
occur: either $x_i$ is a triple junction point, i.e., $K$ is a triple junction and $w$ is locally 
constant on $B_4\setminus K$, or $x_i$ is a crack-tip, i.e., up 
to a rotation $K=\{(x,0):\,x\le 0\}$ and 
$w(\rho,\theta)=C\pm\sqrt{\frac 2\pi \rho}\cdot\sin(\theta/2)$ 
for $\theta\in(-\pi,\pi)$, $\rho>0$ and some constant $C\in\R$ 
(note that in this argument we do not need to know that the blow~up limit is unique).

In both cases, we claim that $\nabla u_j$ has a $C^{0,\alpha}$ 
extension on the closure of each connected component of 
$U_j:=(B_3\setminus\overline{B_1})\setminus\overline{S_{u_{j}}}$ with 
$\sup_j\|\nabla u_{j}\|_{L^\infty(U_j)}\le C$. This follows as in 
\cite[Theorem 7.49]{AFP00} (or \cite[Proposition 17.15]{Dav05},
see also Remark~\ref{r:epsreg}) locally straightening 
$\overline{S_{u_{j}}}\cap(B_4\setminus\overline{B_{1/2}})$ onto 
$K\cap(B_4\setminus\overline{B_{1/2}})$ via a $C^{1,\alpha}$ 
conformal map, a reflection argument and standard Schauder 
estimates for the laplacian.
Scaling back the previous estimate gives 
\[
|\nabla u(x)|\leq C\,|x-x_i|^{-1/2}\quad\text{ for } x\in
\cup_{j\in\N}(\overline{B_{3r_{k_j}}(x_i)}\setminus B_{r_{k_j}}(x_i)),
\] 
in turn from this, the maximum principle and Hopf's lemma we infer 
\[
|\nabla u(x)|\leq C\,r_k^{-1/2}\quad\text{ for } x\in 
B_{2r}(x_i)\setminus B_{r_k}(x_i).
\]
The latter inequality finally implies $\nabla u\in L^{4,\infty}(B_{2r}(x_i))$. 

Eventually, we are able to conclude $\nabla u\in L^{4,\infty}(\Om^\prime)$, 
being on one hand $\nabla u$ bounded on 
$\overline{\Om^\prime}\setminus\cup_iB_{2r}(x_i)$, and on the other hand 
belonging to $L^{4,\infty}(\cup_iB_{2r}(x_i))$.
\end{proof}

\section{Hausdorff dimension of the set of triple-junctions }\label{s:cpt}

In order to prove Theorem~\ref{t:AFH2} we need to analyze the asymptotic behavior
of MS-minimizer in points of vanishing Dirichlet energy.
This issue has been first investigated in \cite[Proposition 5.3, Theorem 5.4]{AFH}. 
Those results hinge upon the notion of \emph{Almgren's area minimizing sets}, i.e.~a 
$\HH^{n-1}$ rectifiable set $S\subset B_1$ such that 
\[
\HH^{n-1}(S)\leq\HH^{n-1}(\varphi(S)),\quad
\forall\varphi\in\mathrm{Lip}(\R^n,\R^n),\,\,
\{\varphi\neq\mathrm{Id}\}\subset\hskip-0.125cm\subset B_1.
\]
Following this approach to infer Theorem~\ref{t:AFH2} requires a delicate study of the behavior 
of the composition of $SBV$ functions with Lipschitz deformations that are not necessarily 
one-to-one, and some specifications on the regularity theory for Almgren's area minimizing sets 
are needed (cf. \cite{AFH}). Therefore, following Ambrosio, Fusco and Hutchinson, Theorem~\ref{t:AFH1} 
is a straightforward corollary of a much deeper and technically demanding result
(given the higher integrability for granted). 

Instead, in Theorem~\ref{t:cptness} below (cf. \cite[Proposition 5.1]{DF1}) we set 
the analysis into the more natural framework of Caccioppoli partitions.
\begin{definition}\label{d:CP}
A Caccioppoli partition of $\Omega$ is a countable partition
$\mathscr{E} = \{E_i\}_{i=1}^\infty$ of $\Omega$ in sets of 
(positive Lebesgue measure and) finite perimeter with
$\sum_{i=1}^\infty \Per (E_i, \Omega) < \infty$. 

For each Caccioppoli partition $\mathscr{E}$ the \emph{set of interfaces} is given by
$J_{\mathscr{E}} := \bigcup_i \partial^* E_i$.

The partition $\mathscr{E}$ is said to be \emph{minimal} if 
\[
\HH^{n-1}(J_{\mathscr E})\leq \HH^{n-1}(J_{\mathscr F})
\]
for all Caccioppoli partitions ${\mathscr F}$ for which 
$\sum_{i=1}^\infty\LL^n\left((F_i\triangle E_i)
\cap(\Omega\setminus \Omega')\right)=0$,
for some open subset $\Omega'\subset\hskip-0.125cm\subset \Omega$.
\end{definition}
There is an important  correspondence between Caccioppoli partitions and 
the subspace $SBV_0$ of ``piecewise constant'' $SBV$ functions recalled 
as prototype example in subsection~\ref{ss:fsetting}, in such a 
way that minimizing the Mumford and Shah energy over $SBV_0$ 
corresponds exactly to the minimal area problem for Caccioppoli partitions
(see \cite[Theorems 4.23, 4.25 and 4.39]{AFP00}). 

Existence of minimal Caccioppoli partitions is guaranteed by Ambrosio's
$SBV$ closure and compactness Theorem~\ref{t:luigi} without imposing any 
$L^\infty$ bound simply by composition with $\arctan(t)$, provided the
partitions are either equi-finite or ordered, i.e.~if $\mathscr{E}=\{E_i\}_{i=1}^\infty$
then $\LL^n(E_i)\geq\LL^n(E_j)$ for $j\geq i$.

A regularity theory for minimal Caccioppoli partitions has been established 
by Massari and Tamanini \cite{MT}. We limit ourselves here to the ensuing statement.
\begin{theorem}[Massari and Tamanini \cite{MT}]\label{t:Caccest}
Let $\mathscr{E}$ be a minimal Caccioppoli partition in $\Om$, 
\[
\omega_{n-1}\leq \liminf_{r\downarrow 0}\frac{\HH^{n-1}(J_\mathscr{E}\cap B_r(x))}{r^{n-1}}\leq 
\limsup_{r\downarrow 0}\frac{\HH^{n-1}(J_\mathscr{E}\cap B_r(x))}{r^{n-1}}\leq n\omega_n. 
\]
In particular, $J_\mathscr{E}$ is essentially closed, i.e.~$\HH^{n-1}\big((\Om\cap\overline{J_\mathscr{E}})\setminus J_\mathscr{E}\big)=0$. 

Moreover, there exists a relatively closed subset $\Sigma_\mathscr{E}$ of $J_{\mathscr{E}}$ 
such that $J_{\mathscr{E}}\setminus \Sigma_\mathscr{E}$ is a $C^{1,\sfrac 12}$ hypersurface and 
$\dimh\Sigma_\mathscr{E}\le n-2$.
If, in addition, $n=2$, then $\Sigma_\mathscr{E}$ is locally finite.
\end{theorem}
We are now ready to prove a compactness result for sequences of MS-minimizers 
with vanishing $L^1$-gradient energy.  
\begin{theorem}[De Lellis and Focardi \cite{DF1}]\label{t:cptness}
Let $(u_k)_{k\in\N}\subset\MM(B_1)$ be such that 
\begin{equation}\label{e:gradvanq}
\lim_{k}\|\nabla u_k\|_{L^1(B_1)}= 0.
\end{equation}
Then, (up to the extraction of a subsequence not relabeled for convenience) 
there exists a  minimal Caccioppoli partition 
${\mathscr E}=\{E_i\}_{i\in\N}$
such that $(\overline{S_{u_k}})_{k\in\N}$ converges locally in the 
Hausdorff distance on $\overline{B_1}$ to $\overline{J_{\mathscr E}}$ and for all open sets $A\subseteq B_1$
\begin{equation}\label{e:enrgconv}
\lim_k\MS(u_k,A)=\lim_k\HH^{n-1}(S_{u_k}\cap A)=\HH^{n-1}(J_{\mathscr E}\cap A).
\end{equation}
\end{theorem}
Though this last statement is, intuitively, quite clear, it is technically 
demanding, because we do not have any a priori control of the norms 
$\|u_k\|_{L^1}$, thus preventing the use of Ambrosio's $(G)SBV$ compactness 
theorem. We can not even expect to gain pre-compactness via De Giorgi's $SBV$ 
Poincar\'e-Wirtinger type inequality, since the latter holds true in a regime 
of small jumps rather than of small gradients as the current one.
\begin{proof}[Proof of Theorem~\ref{t:cptness}] 
The sequence $(u_k)_{k\in\N}$ does not satisfy, a priori, any $L^p$ bound, 
thus in order to gain some insight on the asymptotic behavior of the 
corresponding jump sets we first construct a new sequence 
$(w_k)_{k\in\N}$ with null gradients introducing an infinitesimal error on the 
length of the jump set of $w_k$ with respect to that of $u_k$.
Then, we investigate the limit behavior of the corresponding Caccioppoli
partitions.

\noindent {\bf Step 1.} \emph{There exists a sequence 
$(w_k)_{k\in\N}\subseteq SBV(B_1)$ satisfying}
\begin{itemize}
\item[(i)] $\nabla w_k=0$ $\LL^n$ \emph{a.e. on} $B_1$,

\item[(ii)] $\|u_k-w_k\|_{L^\infty(B_1)}\leq 2\|\nabla u_k\|_{L^1(B_1)}^{\sfrac12}$,

\item[(iii)] $\HH^{n-1}\left(S_{w_k}\setminus(S_{u_k}\cup H_k)\right)=0$
\emph{for some Borel measurable set $H_k$, with $\HH^{n-1}(H_k)=o(1)$ as} 
$k\uparrow\infty$.
\end{itemize}
Note that in turn item (iii) implies that  
\begin{equation}\label{e:salti}
\MS(w_k)=\HH^{n-1}(S_{w_k})\leq\HH^{n-1}(S_{u_k})+o(1)\leq\MS(u_k)+o(1).
\end{equation}
In Step 2 below we shall eventually show that $|\MS (w_k) - \MS (u_k)| \leq o(1)$. 

Recall that the $BV$ Co-Area formula (see \cite[Theorem 3.40]{AFP00}) 
establishes
\begin{equation}\label{e:coareabv}
\int_{B_1}|\nabla u_k|dx=|Du_k|(B_1\setminus S_{u_k})=
\int_{\R}\Per\left(\{u_k\geq t\}\setminus S_{u_k}\right)dt.
\end{equation}
Denote by $I_i^k$ a partition of $\R$ of intervals of equal length 
$\|\nabla u_k\|_{L^1(B_1)}^{\sfrac12}$. Equation \eqref{e:coareabv} and the 
Mean value Theorem provide the existence of levels $t_i^k\in I_i^k$ 
satisfying 
\begin{equation}\label{e:wjbv}
\sum_{i=1}^\infty\Per\left(\{u_k\geq t_i^k\}\setminus S_{u_k}\right)
\leq\|\nabla u_k\|_{L^1(B_1)}^{\sfrac12}.
\end{equation}
Then define the functions $w_k$ to be equal to $t_i^k$ on 
$\{u_k\geq t_i^k\}\setminus\{u_k\geq t_{i+1}^k\}$. The choice of the $I_i^k$'s, 
\eqref{e:wjbv} and the very definition yield that $w_k$ belongs to $SBV(B_1)$ 
and that it satisfies properties (i) and (ii). 
To conclude, note that $\HH^{n-1}\left(S_{w_k}\setminus (\cup_i\partial^\ast\{u_k\geq t_i^k\}
\cup S_{u_k})\right)=0$ by construction, thus item (iii) 
follows at once from \eqref{e:wjbv}. 
\medskip

\noindent {\bf Step 2.} \emph{Compactness for the jump sets.}

Each function $w_k$ determines a Caccioppoli partition 
$\mathscr{E}_k=\{E_i^k\}_{i\in\N}$ of $B_1$ (see \cite[Lemma 1.11]{CT}). 
In addition, upon reordering the sets $E_i^k$'s, we may assume that 
$\LL^n(E_i^k)\geq\LL^n(E_j^k)$ if $i<j$.
Then, the compactness theorem for Caccioppoli partitions (see 
\cite[Theorem 4.1, Proposition 3.7]{LT} 
and \cite[Theorem 4.19]{AFP00}) 
provides us with a subsequence  (not relabeled) and a Caccioppoli 
partition $\mathscr{E}:=\{E_i\}_{i\in\N}$ such that 
\begin{equation}\label{e:lscCp}
\lim_j\sum_{i=1}^\infty\LL^n(E_i^k\triangle E_i)=0,\quad\text{ and }\quad
\sum_{i=1}^\infty\Per(E_i,A)\leq\liminf_k\sum_{i=1}^\infty\Per(E_i^k,A)
\end{equation}
for all open subsets $A$ in $B_1$.
We claim that $\mathscr{E}$ determines a minimal Caccioppoli partition
and in proving this we will also establish \eqref{e:enrgconv}. 

We start off observing that the first identity \eqref{e:lscCp} and the Co-Area formula yield 
the existence of a set $I\subset (0,1)$ of full measure such that
\begin{equation}\label{e:taglio}
\liminf_k\sum_{i=1}^\infty
\HH^{n-1}\left((E_i^k\triangle E_i)\cap\partial B_\rho\right)=0 \qquad \forall \rho\in I\, .
\end{equation}
Define the measures $\mu_k$ as
$\mu_k(A):=\MS(u_k,A)+\MS(w_k,A)$ ($A$ being an arbitrary Borel subset of $B_1$).
Proposition~\ref{p:dub} and item (iii) in Step 1 ensure that, 
upon the extraction of a further subsequence, $\mu_k$ converges weakly$^*$ 
to a finite measure $\mu$ on $B_1$. 
W.l.o.g. we may assume that for all $\rho\in I$ we have, in addition, 
$\mu(\partial B_\rho)=0$.

Let us now fix a Caccioppoli partition ${\mathscr F}:=\{F_i\}_{i\in\N}$  
suitable to test the minimality of ${\mathscr E}$, i.e.~$\sum_{i=1}^\infty\LL^n\left((F_i\triangle E_i)\cap(B_1\setminus\overline{B_t})\right)=0$ 
for some $t\in (0,1)$. Moreover, we may also suppose that  
$\sum_{i=1}^\infty\HH^{n-1}\left((F_i\triangle E_i)\cap\partial B_\rho\right)=0$
for all $\rho\in I\cap(t,1)$.
Let then $\rho$ and $r$ be radii in $I\cap(t,1)$ with $\rho<r$ and assume,
after passing to a subsequence (not relabeled) that the inferior limit in \eqref{e:taglio} is
actually a limit for these two radii. We define
\[
\omega_k:=\begin{cases}
w_k & \quad\mbox{on $B_1\setminus \overline{B_\rho}$}\cr
t_i^k & \quad \mbox{on $F_i\cap B_\rho$}.
\end{cases}
\]
Note that $\omega_k\in SBV(B_1)$ with $\nabla\omega_k=0$ $\LL^n$ a.e. 
on $B_1$, and since $t<\rho\in I$ it follows
\[
\HH^{n-1}\Big(S_{\omega_k}\setminus\big((J_{\mathscr F}\cap B_\rho)
\cup(\cup_{i\in\N}(E_i^k\triangle E_i)\cap\partial B_\rho)
\cup(S_{w_k}\cap(B_1\setminus\overline{B_\rho}))\big)\Big)=0.
\]
Consider $\varphi\in \mathrm{Lip}\cap C_c(B_1,[0,1])$ with 
$\varphi|_{B_r}\equiv 1$, and $|\nabla\varphi|\leq (1-r)^{-1}$ on 
$B_1$, and set $v_k:=\varphi\,\omega_k+(1-\varphi)\,u_k$. Clearly, $v_k$ 
is admissible to test the minimality of $u_k$. Then, simple calculations lead to  
\begin{eqnarray}\label{e:apprmin}
\lefteqn{\MS(u_k)\leq\MS(v_k)}\notag\\&&
\leq\MS(\omega_k)+2\MS(u_k,B_1\setminus\overline{B_r})
+\frac{2}{(1-r)^2}\|u_k-\omega_k\|^2_{L^2(B_1\setminus\overline{B_r})}\notag\\
&&\leq\HH^{n-1}\left(J_{\mathscr{F}}\right)
+\sum_{i\in\N}\HH^{n-1}\left((E_i^k\triangle E_i)\cap\partial B_\rho\right)
+\HH^{n-1}\left(S_{w_k}\setminus\overline{B_\rho}\right)\notag\\
&&+2\MS(u_k,B_1\setminus\overline{B_r})
+\frac{2}{(1-r)^2}\|u_k-w_k\|^2_{L^2(B_1\setminus\overline{B_r})}\notag\\
&&\leq\HH^{n-1}\left(J_{\mathscr{F}}\right)
+\sum_{i\in\N}\HH^{n-1}\left((E_i^k\triangle E_i)\cap\partial B_\rho\right)
+3\mu_k(B_1\setminus\overline{B_\rho})\notag\\&&
+\frac{2}{(1-r)^2}\|u_k-w_k\|^2_{L^\infty(B_1)}.
\end{eqnarray}
Note that in the third inequality we have used that $\omega_k$ and $w_k$ 
coincide on $B_1\setminus\overline{B_\rho}$, 
and that $\rho<r$. By letting $k\uparrow\infty$ in \eqref{e:apprmin}, 
we infer
\begin{multline*}
\HH^{n-1}(J_{\mathscr{E}})\leq\liminf_j\HH^{n-1}(S_{u_k})\leq
\liminf_j\MS(u_k)\leq\limsup_j\MS(u_k)\\
\leq\limsup_k\MS(v_k)\leq\HH^{n-1}\left(J_{\mathscr{F}}\right)
+3\mu(B_1\setminus\overline{B_\rho}),
\end{multline*} 
where we have used that $r$ and $\rho$ belong to $I$, inequality \eqref{e:salti}, 
the convergence $\mu_k\rightharpoonup^* \mu$, the estimate $\sup_k\MS(u_k,B_1\setminus B_t)
\leq n\omega_n(1-t^{n-1})$ for all $t\in(0,1)$, that is derived as inequality \eqref{e:dub} in
Proposition~\ref{p:dub}, and the limit
\eqref{e:taglio}. Finally, by letting $\rho\in I$ tend to $1^-$ we conclude 
\begin{equation}\label{e:Eminimal}
\HH^{n-1}(J_{\mathscr{E}})\leq\liminf_k\HH^{n-1}(S_{u_k})
\leq\liminf_k\MS(u_k)\leq\limsup_k\MS(u_k)\leq
\HH^{n-1}\left(J_{\mathscr{F}}\right),
\end{equation}
which proves the minimality of $\mathscr{E}$. In addition, choosing $\mathscr{E} = \mathscr{F}$, 
we infer \eqref{e:enrgconv} for $A=B_1$. Actually, for $\mathscr{E} = \mathscr{F}$ the same
same argument employed above gives \eqref{e:enrgconv} (it suffices to take $v_k=\varphi w_k
+(1-\varphi)u_k$).

In particular, $J_{\mathscr{E}}$ is essentially closed  (by Theorem~\ref{t:Caccest}) and it satisfies 
a density lower bound estimate.
Using this and the De Giorgi, Carriero, Leaci density lower bound in formula \eqref{e:dlb} we 
conclude that $(\overline{S_{u_k}})_{k\in\N}$ converges to $\overline{J_{\mathscr{E}}}$ in the 
local Hausdorff topology on $\overline{B_1}$.
\end{proof}
Interesting (immediate) consequences of Theorem~\ref{t:cptness} are contained in the ensuing two statements.
\begin{corollary}
 Let $(u_k)_{k\in\N}\subset\MM(B_1)$ be as in the statement of Theorem~\ref{t:cptness}, then 
 \[
\lim_{k}\|\nabla u_k\|_{L^2(B_1)}= 0,\quad\text{and }\quad 
\HH^{n-1}\res S_{u_k}\stackrel{\ast}{\to} \HH^{n-1}\res J_{\mathscr E}.
 \]
 Actually, $\HH^{n-1}\res S_{u_k}\to\HH^{n-1}\res J_{\mathscr E}$ in the 
 narrow convergence of measures, i.e.~in the duality with $C_b(\Omega)$.
\end{corollary}
\begin{proof}
It is an easy consequence of the equalities in \eqref{e:enrgconv}. 
\end{proof}
\begin{corollary}\label{c:triple}
Let $x\in \Sigma_u^{(1)}$ and $\rho_k\downarrow 0$, then (up to subsequences not relabeled)
there exists a minimal Caccioppoli partition ${\mathscr E}$ such that $(S_{u_{x,\rho_k}})_{k\in\N}$,
$u_{x,\rho_k}$ defined in \eqref{e:blowup}, converges locally in the Hausdorff distance to 
$\overline{J_{\mathscr E}}$ and
\[
 \HH^{n-1}\res S_{u_{x,\rho_k}}\stackrel{\ast}{\to} \HH^{n-1}\res J_{\mathscr E}.
\]
\end{corollary}
\begin{remark}
Under the assumptions of Corollary~\ref{c:triple}, it is natural to expect the limit partition ${\mathscr E}$ 
to be \emph{conical}, i.e. ${\mathscr E}=\{E_i\}_{i=1}^\infty$ with the $E_i$'s cones with vertices in the origin,
as a result of the blow up procedure. 
In general this latter property can be proven only for suitable sequences $\rho_k\downarrow 0$ by combining
a blow~up argument and Theorem~\ref{t:cptness} (cf. \cite[Proposition~5.8]{AFH}).

Actually, in $2$-dimensions the result is true for every sequence $\rho_k\downarrow 0$ since a structure 
theorem for minimal Caccioppoli partitions assures that (locally) they are minimal connections (cf. 
Proposition~\ref{p:min=con}).
The lack of monotonicity formulas for the Mumford and Shah problem prevents the derivation of such a statement 
in the general case.
\end{remark}

A more precise result in the $2$d case will be established in Proposition~\ref{p:bupchar}.
Note that no uniqueness is ensured for the limits except for $2$d in view of David's 
$\eps$-regularity result (see Theorem~\ref{t:epsreg}), and in $3$d in view of the analogous 
result established by Lemenant in \cite{lemenant}.

By means of Theorems~\ref{t:Caccest}, \ref{t:cptness} and standard blow~up arguments we are 
able to establish Theorem~\ref{t:AFH2}. Let us first recall a technical lemma.
\begin{lemma}[Section~3.6 \cite{Sim83}]\label{l:haus}
 Let $s\geq 0$, then
 \begin{itemize}
  \item[(i)] $\HH^s(\Sigma)=0\Longleftrightarrow\HH^{s,\infty}(\Sigma)=0$, for all sets $\Sigma\subseteq\R^n$;
  \item[(ii)] if $\Sigma_j$ and $\Sigma$ are compact sets such that 
  $\sup_{\Sigma_j}\mathrm{dist}\big(\cdot,\Sigma\big)\downarrow 0$ as $j\uparrow\infty$, then
\[
\HH^{s,\infty}(\Sigma)\geq \limsup_j\HH^{s,\infty}(\Sigma_j);
\]
  \item[(iii)] if $\HH^s(\Sigma)>0$, then for $\HH^s$-a.e. $x\in\Sigma$
  \[
  \limsup_{\rho\downarrow 0^+}\frac{\HH^{s,\infty}(\Sigma\cap B_\rho(x))}{\omega_s\rho^s}
\geq 2^{-s}.
  \]
 \end{itemize}
\end{lemma}
The strategy of proof below is essentially that by Ambrosio, Fusco and Hutchinson as reworked by 
De Lellis, Focardi and Ruffini in light of \cite[Theorem~\ref{t:Caccest}]{DFR}.
\begin{proof}[Proof of Theorem~\ref{t:AFH2}]
We argue by contradiction: suppose that there exists $s> n-2$ such that 
$\HH^s\big(\Sigma_u^{(1)}\big)>0$. 
From this we infer that $\HH^{s,\infty}\big(\Sigma_u^{(1)}\big)>0$, 
and moreover that for $\HH^s$-a.e. $x\in \Sigma_u^{(1)}$ it holds
\begin{equation}\label{e:limsup}
\limsup_{\rho\downarrow 0^+}
\frac{\HH^{s,\infty}(\Sigma_u^{(1)}\cap B_\rho(x))}{\rho^s}
\ge \frac{\omega_s}{2^s}
\end{equation}
(see for instance \cite[Theorem 2.56 and formula (2.43)]{AFP00}). 
Without loss of generality, suppose that \eqref{e:limsup} holds at $x=0$, and 
consider a sequence $\rho_k\downarrow 0$ for which
\begin{equation}\label{e:assurdo}
\HH^{s,\infty}(\Sigma_u^{(1)}\cap B_{\rho_k})\ge\frac{\omega_s}{2^{s+1}}\rho_k^s
\qquad\text{for all $k\in\N$}.
\end{equation}
Theorem~\ref{t:cptness} provides a subsequence, not relabeled for convenience, 
and a minimal Caccioppoli partition $\mathscr{E}$ such that 
\begin{equation}\label{e:narrow}
\lim_k\HH^{n-1}(S_{u_k}\cap A)=\HH^{n-1}(J_{\mathscr E}\cap A)
\qquad \mbox{for all open sets $A\subseteq B_1$.} 
\end{equation}
and that
\begin{equation}\label{e:DeLF}
\rho_k^{-1}\overline{S_u}\to\overline{J_\mathscr{E}}\,\,\text{ locally Hausdorff}.
\end{equation}
In turn, from the latter we claim that if $\mathcal{F}$ is any open cover  
of $\Sigma_\mathscr{E}\cap \overline B_1$, then for some $h_0\in\N$ 
\begin{equation}\label{e:usc}
\rho_k^{-1}\Sigma_u^{(1)}\cap\overline B_1\subseteq 
\cup_{F\in\mathcal{F}}F\qquad \text{for all $k\ge k_0$}. 
\end{equation}
Indeed, if this is not the case we can find a sequence 
$x_{k_j}\in \rho_{k_j}^{-1}\Sigma_u^{(1)}\cap\overline B_1$ 
converging to some point $x_0\notin\Sigma_\mathscr{E}$. 
If $T^\mathscr{E}_{x_0}$ is the tangent plane to $J_\mathscr{E}$ at $x_0$
(which exists by the property of $\Sigma_\mathscr{E}$ in Theorem~\ref{t:Caccest}), 
then for some $\rho_0$ we have
\[
\rho^{-1-n}\int_{B_\rho(x_0)\cap J_\mathscr{E}}
\mathrm{dist}^2(y,T^\mathscr{E}_{x_0})d\HH^{n-1}<\varepsilon_0,
\quad \text{ for all $\rho\in(0,\rho_0)$}.
\]
In turn, from the latter inequality and the convergence in \eqref{e:narrow},
it follows that, for $\rho\in(0,\rho_0\wedge 1)$,
\[
\limsup_{j\uparrow\infty}\rho^{-1-n}\int_{B_\rho(x_{k_j})\cap\rho_{k_j}^{-1}S_u}
\mathrm{dist}^2(y,T^\mathscr{E}_{x_0})\,d\HH^{n-1}<\varepsilon_0.
\]
Therefore, as $x_{k_j}\in\rho_{k_j}^{-1}\Sigma_u^{(1)}$, we get for 
$j$ large enough 
\[
\limsup_{\rho\downarrow 0}
\left(\mathscr{D}(x_{k_j},\rho)+\mathscr{A}(x_{k_j},\rho)\right)<\varepsilon_0,
\]
a contradiction in view of the characterization of the singular set in \eqref{e:sigma}. 

To conclude, we note that by \eqref{e:usc} we get
\[
\HH^{s,\infty}(\Sigma_\mathscr{E}\cap\overline{B_1})\geq
\limsup_{k\uparrow\infty}\HH^{s,\infty}(\rho_k^{-1}\Sigma_u^{(1)}\cap\overline{B_1});
\]
given this, \eqref{e:assurdo} and \eqref{e:DeLF} yield that
\[
\HH^s(\Sigma_\mathscr{E}\cap \overline{B_1})\ge\HH^{s,\infty}(\Sigma_\mathscr{E}\cap\overline{B_1})\ge
\limsup_{k\uparrow\infty}\HH^{s,\infty}(\rho_k^{-1}\Sigma_u^{(1)}\cap\overline{B_1})\ge
\frac{\omega_s}{2^{s+1}},
\]
thus contradicting Theorem~\ref{t:Caccest}.
\end{proof}
\begin{corollary}\label{c:tj}
 If $\Omega\subseteq\R^2$ and $u\in\MM(\Om)$, then $\Sigma_u^{(1)}$ is at most countable.
\end{corollary}
\begin{proof}
This claim follows straightforwardly from the compactness result Theorem~\ref{t:cptness}, 
David's $\varepsilon$-regularity Theorem~\ref{t:epsreg}, and 
a direct application of Moore's triod theorem showing that in the plane 
every system of disjoint triods, i.e.~unions of three Jordan arcs that 
have all one endpoint in common and otherwise disjoint, is at most countable
(see \cite[Theorem~1]{Moore} and \cite[Proposition~2.18]{pommerenke}). 
\end{proof}
\begin{remark}
 Analogously, in $3$-dimensions the set of points with blow~up a $\mathbb{T}$ cone, 
i.e.~a cone with vertex the origin constructed upon the $1$-skeleton of a regular tetrahedron, is at most countable. 
The latter claim follows thanks to Theorem~\ref{t:cptness}, the $3$d extension of David's $\varepsilon$-regularity 
result by Lemenant in \cite[Theorem~8]{lemenant}, and a suitable extension of Moore's theorem on triods established 
by Young in \cite{Young}.

Let us point out that we employ topological arguments to compensate for the lack of monotonicity
formulas. The latter would allow one to exploit Almgren's stratification type results and get, actually, 
a more precise picture of the set $\Sigma_u^{(1)}$ (cf.~with \cite[Theorem 3.2]{White} and \cite{FMS15}).
\end{remark}

\begin{remark}\label{r:triplen2}
In $2$d Theorem~\ref{t:Caccest} provides the local finiteness of the singular set for 
minimal Caccioppoli partitions, the blow~up limits of $\rho^{-1}(\SSu-x)$ in points $x\in \Sigma_u^{(1)}$. 
This conclusion is far from being established for the set $\Sigma_u^{(1)}$ itself.
With the results at hand one can prove that every convergent sequence $(x_j)_{j\in\N}\subset\Sigma_u^{(1)}$ 
has a limit $x_0\notin\big(\Sigma_u^{(1)}\cup \Sigma_u^{(2)}\big)$. 
To show this, first note that $x_0\notin\Sigma_u^{(1)}$ thanks to item (iii) in Proposition~\ref{p:min=con} 
or Theorem~\ref{t:epsreg};  moreover $x_0\notin\Sigma_u^{(2)}$ thanks to item (ii) in Proposition~\ref{p:bupchar} 
below and Theorem~\ref{t:David-zurrione}.
Similarly, any converging sequence $(x_j)_{j\in\N}\subset\Sigma_u^{(2)}$ has a limit $x_0\notin\big(\Sigma_u^{(1)}\cup\Sigma_u^{(2)}\big)$.
Therefore, in both instances, it might happen that the limit point $x_0$ belong to 
\[
\Sigma_u^{(3)}=\{x\in\Sigma_u:\,\liminf_{\rho\downarrow 0}\mathscr{D}_u(x,\rho)>0,\,
\liminf_{\rho\downarrow 0}\mathscr{A}_u(x,\rho)>0\}.
\]
\end{remark}

We conclude the section by justifying the denomination used for the sets $\Sigma_u^{(1)}$ and 
$\Sigma_u^{(2)}$ in $2$-dimensions.
\begin{proposition}\label{p:bupchar}
 Let $\Om\subseteq\R^2$ and $u\in\MM(\Om)$, then
 \begin{itemize}
  \item[(i)] $x\in\Sigma_u^{(1)}$ if and only if every blow~up of $u$ in $x$ is a triple junction function;
  \item[(ii)] $x\in\Sigma_u^{(2)}$ if and only if  every blow~up of $u$ in $x$ is a crack-tip function.
  \end{itemize}
\end{proposition}
\begin{proof}
We start off recalling that $(u,\Om\cap\SSu)$ is an \emph{essential pair}, 
 i.e.~$\HH^1(\SSu\cap B_r(x))>0$ for all $x\in \SSu$ and $B_r(x)\subseteq\Om$ (see Theorem~\ref{t:DGCL}).
Then the existence and several properties of blow~up limits are guaranteed by 
\cite[Propositions~37.8, 40.9, Corollary~38.48]{Dav05}. 
More precisely, with fixed a point $x\in\Om$ and a sequence $\rho_k\downarrow 0$, up to subsequences not relabeled, 
we may assume that the sets $K_k:=\rho_k^{-1}(\SSu-x)$ locally Hausdorff converges in $\R^2$ to some closed 
set $K$ as $k\uparrow\infty$. 
Then there is a subsequence (not relabeled for convenience)
and continuous piecewise constant functions $c_k$ on $\R^2\setminus K_k$, such that the pairs $(u_{x,\rho_k},K_k)$
with $u_{x,\rho_k}(y):=\rho_k^{-\sfrac12}\big(u(x+\rho_k\,y)-c_k(y)\big)$ satisfy: 
\begin{itemize}
 \item[(a)] $(u_{x,\rho_k})_k$ converges to some $w$ in $W^{1,2}_{\mathrm{loc}}(\R^2\setminus K)$\footnote{
 As $(K_k)_k$ locally Hausdorff converges to $K$, every $O\subset\hskip-0.125cm\subset\R^2\setminus K$ is 
 contained for $k$ sufficiently big in $\rho_k^{-1}(\Om-x)\setminus K$, so that the convergence of 
 $(u_{x,\rho_k})$ in $W^{1,2}(O)$ is well defined.},
 \item[(b)] $(w,K)$ is a global Bonnet minimizer according to Definition~\ref{d:Bmin}, and an essential pair
  (see \cite[Remark~54.8]{Dav05}),
 \item[(c)] for $\LL^1$ a.e. $r>0$
\[
 \lim_k\int_{B_r\setminus K_k}|\nabla u_{x,\rho_k}|^2dy=\int_{B_r\setminus K}|\nabla w|^2dy,\qquad
 \lim_k\HH^1\big(B_r\setminus K_k\big)=\HH^1\big(B_r\setminus K\big).
\]
 \end{itemize}

To prove the direct implication in case (i) note that by Corollary~\ref{c:triple}, the jump set of any 
blow~up limit in a point $x\in\Sigma_u^{(1)}$ is a minimal Caccioppoli partition ${\mathscr{E}}$, and that 
$0$ is a singular 
point for it (cf. the argument leading to \eqref{e:usc} in the proof of Theorem~\ref{t:AFH2}). Finally, 
Proposition~\ref{p:min=con} below ensures then that $J_{\mathscr{E}}$ is (locally) a triple junction
around $0$. 

For the direct implication in item (ii), as by the very definition of $\Sigma_u^{(2)}$ for $\LL^1$ a.e. $r>0$
\[
\lim_{k}\mathscr{A}_u(x,r\,\rho_k)=\lim_{k}\mathscr{A}_{u_{x,\rho_k}}(0,r)=
\mathscr{A}_w(0,r)=0,
\] 
$\overline{S_w}$ is actually contained in a $1$-dimensional vector space. 
In this case a result by L\'eger \cite{Leg99} ensures 
that $\overline{S_w}$ is either empty or a line or a half line (cf. \cite[Theorem~64.1]{Dav05}).
Therefore, the energy upper bound in \eqref{e:dub} and item (iii) above yield for $\LL^1$ a.e. $r>0$
\begin{equation}\label{e:dircontra}
\lim_{k}\mathscr{D}_u(x,r\,\rho_k)=\lim_{k}\mathscr{D}_{u_{x,\rho_k}}(0,r)=
\mathscr{D}_w(0,r)\in[\e_0,2\pi\,r].
\end{equation}
The possibility that $\overline{S_w}=\emptyset$ is ruled out as follows: in such a case $|\nabla w|^2$ 
would be subharmonic on $\R^2$, being $w$ harmonic there, and thus we would deduce that
\[
 \sup_{B_{\sfrac r2}}|\nabla w|^2\leq \frac{4}{\pi r^2}\int_{B_r}|\nabla w|^2dx
 \stackrel{\eqref{e:dub}}{\leq}\frac 8r.
\]
By letting $r\uparrow\infty$ we would conclude $w$ to be constant, in contrast to \eqref{e:dircontra}. 
Analogously, if $\overline{S_w}$ would be a line, $w$ would be harmonic in $\R^2\setminus\overline{S_w}$.
Considering the restriction of $w$ to one of the two half-spaces forming $\R^2\setminus\overline{S_w}$ and 
performing an even reflection, since $\frac{\partial w}{\partial\nu}=0$ on $\overline{S_w}$ we would get 
an harmonic function $\widetilde{w}$ on $\R^2$ satisfying for all $r>0$
\[
 \int_{B_r}|\nabla \widetilde{w}|^2dx\leq 4\pi r.
\]
Arguing as before $\widetilde{w}$ would be constant. Hence, $w$ would be locally constant on 
$\R^2\setminus\overline{S_w}$, leading again to a contradiction to \eqref{e:dircontra}. 
Therefore, $\overline{S_w}$ is a half-line, that up to a rotation can be written as 
$\overline{S_w}=\{(x,0):\,x\leq 0\}$. Then the map $\widetilde{w}:\{z\in\mathbb{C}:\,\mathrm{Re}z\geq 0\}\to\R$
defined by $\widetilde{w}(z):=w(z^2)$ is harmonic,  $\frac{\partial \widetilde{w}}{\partial\nu}=0$ on 
$\{\mathrm{Re}z= 0\}$ and it satisfies
\begin{equation}\label{e:tildew}
 \int_0^r\int_{-\sfrac\pi 2}^{\sfrac\pi 2}\Big(\rho\Big|\frac{\partial\widetilde{w}}{\partial r}\Big|^2
 +\frac1\rho\Big|\frac{\partial\widetilde{w}}{\partial \theta}\Big|^2\Big)d\rho\, d\theta
=\int_{B_{r^2}}|\nabla w|^2\leq 2\pi r^2.
\end{equation}
In view of \eqref{e:tildew}, the even extension of $\widetilde{w}$ on $B_1$, 
that we still denote by $\widetilde{w}$, has Fourier decomposition
\[
 \widetilde{w}(r,\theta)=\alpha_0+\beta_1\,r\,\sin\theta.
\]
Changing back coordinates and taking into account the minimality of $w$ we get
 \[
 w(r\cos\theta,r\sin\theta)=\alpha_0\pm\sqrt{\frac 2\pi r}\,\sin\frac\theta 2,
\]
the conclusion follows at once.

The reverse implications in both cases are easily concluded. Indeed, in case (i) 
if $(\rho_k)_k$ satisfies 
\[
\lim_k\mathscr{D}_u(x,\rho_k)=\limsup_{\rho\downarrow 0} \mathscr{D}_u(x,\rho),
\]
then the blow~up limit $w$ of $(u_{x,\rho_k})_k$ is a triple junction function
by assumption. By taking into account item (c) above we infer for $\LL^1$ a.e. $r>0$
\[
 \lim_k\mathscr{D}_u(x,r\rho_k)=\lim_k\mathscr{D}_{u_{x,\rho_k}}(0,r)=\mathscr{D}_w(0,r)=0,
\]
thus implying that $x\in \Sigma_u^{(1)}$.

Similarly, one can prove the reverse implication in case (ii).
\end{proof}

Actually, in the first instance of Proposition~\ref{p:bupchar} uniqueness of the blow~up limit 
is ensured by Theorem~\ref{t:epsreg}. As already remarked, case (ii) is still open.

\section{Higher integrability of the gradient in dimension $2$}\label{s:hi2}

The higher integrability of the gradient has been first established by De Lellis and Focardi~\cite{DF1} 
in dimension $2$. 
Following a classical path, the key ingredient to establish Theorem~\ref{t:hi}
is a reverse H\"older inequality for the gradient, which we state 
independently (see \cite[Theorem 1.3]{DF1}).
\begin{theorem}[De Lellis and Focardi \cite{DF1}]\label{t:rH}
For all $q\in(1,2)$ there exist $\rho\in(0,1)$ and $C>0$ such that
\begin{equation}\label{e:rH}
\|\nabla u\|_{L^2(B_\rho)}\leq C\|\nabla u\|_{L^q(B_1)} \qquad 
\text{for any } u\in\MM(B_1).
\end{equation}
\end{theorem}
Using the obvious scaling invariance of \eqref{e:ms}, Theorem~\ref{t:rH} 
yields a corresponding reverse H\"older inequality for balls of arbitrary 
radius. Theorem~\ref{t:hi} is then a consequence of a by now classical result. 
\begin{theorem}[Giaquinta and Modica~\cite{GM}]
 Let $v\in L^q_{loc}(\Omega)$, $q>1$, be nonnegative such that for some constants 
 $\beta>0$,  $\lambda\geq 1$ and $R_0>0$
 \[
  \left(\fint_{B_r(z)}v^q\,dy\right)^{\sfrac 1q}\leq \beta\fint_{B_{\lambda\,r}(z)}v\,dy
 \]
for all $z\in\Omega$, $r\in\big(0,R_0\wedge \mathrm{dist}(z,\partial\Omega)\big)$.

 Then $v\in L^p_{loc}(\Omega)$ for some $p>q$ and there is $C=C(\beta,n,q,p,\lambda)>0$
 such that 
 \[
 \left(\fint_{B_r(z)}v^p\,dy\right)^{\sfrac 1p}\leq C\,
 \left(\fint_{B_{2r}(z)}v^q\,dy\right)^{\sfrac 1q}.
 \]
\end{theorem}
The exponent $p$ could be explicitly estimated in terms of $q$, $C$ and $\rho$. 
However, since our argument for Theorem~\ref{t:rH} is indirect, we do not have any explicit estimate for $C$ 
($\rho$ can instead be computed). Hence, combining Theorem~\ref{t:hi} with 
\cite{AFH} we can only conclude that the dimension of the singular set of 
$\SSu$ is strictly smaller than $1$. 
Guy David pointed out that the corresponding dimension estimate could be made explicit. 
In fact, he suggested to the Authors of \cite{DF1} that also the constant $C$ in Theorem~\ref{t:rH} 
might be estimated: a viable strategy would combine the core argument of this 
paper with some ideas from \cite{Dav05} (the proof of Theorem~\ref{t:rH} given 
here makes already a fundamental use of the paper \cite{Dav05}, but depends 
only on the $\eps$-regularity theorems for ``triple junctions'' and ``segments'' stated in Section~\ref{s:regSu}). 
However, the resulting estimate would give an extremely small number, whereas 
the proof would very likely become much more complicated. 

In spite of the dimensional restriction, the indirect proof has as interesting side 
results Theorem~\ref{t:cptness} and its related consequences highlighted in 
Section~\ref{s:cpt}.
No dimensional limitation is present in Theorem~\ref{t:cptness}, instead dimension $2$ 
enters dramatically in the proof of Theorem~\ref{t:rH} as the structure of minimal Caccioppoli 
partitions in $\R^2$ can be described precisely via minimal connections. Recall that 
a \emph{minimal connection} of $\{q_1,\ldots,q_N\}\subset \R^2$ is any minimizer of the Steiner problem
\[
\min\left\{\HH^1(\Gamma):\,\Gamma 
\text{ closed and connected, and }q_1,\ldots,q_N\in\Gamma\right\}.
\]
\begin{proposition}[Proposition~11, Lemma~12 \cite{DF1}]\label{p:min=con}
Let $\mathscr{E}$ be a minimal Caccioppoli partition in $\Om\subset\R^2$, then 
\begin{itemize}
\item[(i)] $\HH^0(\overline{J_\mathscr{E}}\cap\partial B_\rho(x))<+\infty$
if $B_\rho(x)\subset\hskip-0.125cm\subset \Omega$;

\item[(ii)] $\HH^0(K\cap \partial B_\rho(x))\geq 2$ for each connected component $K$ of 
$\overline{J_\mathscr{E}}\cap \overline{B}_\rho(x)$, and it is a minimal connection
of $K\cap \partial B_\rho(x)$;

\item[(iii)] if $\Om=B_1$, then there exists $\rho_0\in(0,1)$ such that 
for all $t\in(0,\rho_0)$
\[
\HH^0(\overline{J_{\mathscr{E}}}\cap\partial B_{t})\leq 3,\; \text{ and }\;
\HH^1(\overline{J_{\mathscr{E}}}\cap B_{t})\leq 3t.
\]
\end{itemize}
\end{proposition}
We are now ready to sketch the proof of Theorem~\ref{t:rH} in $2$-dimensions following
De Lellis and Focardi \cite{DF1}.
\begin{proof}[Proof of Theorem~\ref{t:rH}]
We fix an exponent $q\in (1,2)$ and a suitable radius $\rho$ (whose choice 
will be specified later) for which \eqref{e:rH} is false, that is
\begin{equation}\label{e:rH_contra}
\|\nabla u_k\|_{L^2(B_\rho)}\geq k\|\nabla u_k\|_{L^q(B_1)}\quad
\text{for a sequence }(u_k)_{k\in\N}\in \MM (B_1).
\end{equation}
Since the Mumford and Shah energy of any $u\in \MM (B_1)$ can be easily
bounded a priori by $2\pi$ (cf. Proposition~\ref{p:dub}), we have 
$\|\nabla u_k\|_{L^q(B_1)}\to 0$. Theorem~\ref{t:cptness} and 
Proposition~\ref{p:min=con} then show that:
\begin{itemize}
\item[(i)] The $L^2$ energy of the gradients of $u_k$ converge to $0$;
\item[(ii)] $\overline{S_{u_k}}$  converge in the local Hausdorff metric 
to the (closure of) set of interfaces of a minimal Caccioppoli partition 
$\overline{J_{\mathscr{E}}}$;
\item[(iii)] $\HH^0(\overline{J_{\mathscr{E}}}\cap\partial B_{t})\leq 3$
for $t\in(0,\rho_0)$.
\end{itemize}
An elementary argument shows the existence of $t\geq \rho_0/4$ such that
\begin{itemize}
  \item[(a)] either $\overline{J_{\mathscr{E}}}\cap B_{t}=\emptyset$; 
  
  \item[(b)] or $\overline{J_{\mathscr{E}}}\cap B_{t}$ is a 
  segment and $\partial B_{t}\setminus \overline{J_{\mathscr{E}}}$ 
  is the union of two arcs each with length $<\frac{4\pi}{3}t$; 
  
  \item[(c)] or $\overline{J_{\mathscr{E}}}\cap B_{t}$ is a triple junction
  and $\partial B_{t}\setminus\overline{J_\mathscr{E}}$ the union of  
  three arcs each with length $<(2\pi-\frac{1}{8})t$. 
  
\end{itemize}
In any case we set $\rho:=\rho_0/9$ (cf. with \eqref{e:rH_contra}).
By Theorem~\ref{t:epsreg} (we keep the notation introduced there), we may find a constant 
$\beta\in(0,1/3)$ such that for all $k$ sufficiently big one of the following alternatives happens
\begin{itemize}
  \item[(a$_1$)] $\overline{S_{u_k}}\cap B_{t}=\emptyset$;

  \item[(b$_1$)] For each $s\in ((1-\beta)t,t)$,  
  $\partial B_{s}\setminus \overline{S_{u_k}}$ is the union of two arcs 
  $\gamma_1^k$ and $\gamma_2^k$ each with length $< (2\pi-\frac{1}{9})s$, 
  whereas $\overline{S_{u_k}}\cap B_s$ is connected and divides $B_s$ in two components 
  $B_1^k$, $B_2^k$ with $\partial B_i^k=\gamma_i^k\cup(\overline{S_{u_k}}\cap\overline{B_s})$;

  \item[(c$_1$)] For each $s\in ((1-\beta)t,t)$, $\partial B_s\setminus \overline{S_{u_k}}$ 
  is the union of three arcs $\gamma_1^k$, $\gamma_2^k$ and $\gamma^k_3$ each with 
  length $<(2\pi-\frac{1}{9})s$, whereas $\overline{S_{u_k}}\cap B_s$ is connected and divides 
  $B_s$ in three connected components $B_1^k$, $B_2^k$ and $B_3^k$ with 
  $\partial B_i^k\subset\gamma_i^k\cup  (\overline{S_{u_k}}\cap \overline{B_s})$.
\end{itemize}

Choose then $r\in(\sfrac 23t,t)$ and a subsequence (not relabeled) such that  
\[
g_k:=u_k|_{\partial B_r}\in W^{1,q}(\gamma,\HH^1)\qquad\text{for any connected component 
of $\partial B_r\setminus S_{u_k}$}
\] 
and
\[
\int_{\partial B_r\setminus S_{u_k}}|g_k^\prime|^qd\HH^1\leq
\frac 3{t}\int_{B_{t}}|\nabla u_k|^q\,dx
\leq\frac {12}{\rho_0}\int_{B_1}|\nabla u_k|^q\,dx.
\]

Let us first deal with the (easier) case (a). By compactness, as  
$\overline{J_{\mathscr{E}}}\cap B_t=\emptyset$ then
\[
\overline{S_{u_k}}\cap B_t=\emptyset\quad\text{for } k\gg 1.
\]
Hence, being $u_k\in\MM(B_1)$ we get that $u_k$ is the harmonic 
extension of its trace in $B_t$. 
In conclusion, as $\rho=\sfrac{\rho_0}9<\sfrac 23\,t<r$ we have
\begin{multline*}
\int_{B_{\sfrac{\rho_0}9}} |\nabla u_k|^2dx \leq \int_{B_r}|\nabla u_k|^2dx
\leq C\,\min_\lambda \|g_k-\lambda\|_{H^{\sfrac12}(\partial B_r)}^2\\ 
\stackrel{W^{1,q}\hookrightarrow H^{\sfrac12}}{\leq} 
C\,\left(\int_{\partial B_r}|g_k^\prime|^q\,d\HH^1\right)^{\sfrac2q}
\le C\,\left(\frac {12}{\rho_0}\int_{B_1}|\nabla u_k|^q\,dx\right)^{\sfrac2q},
\end{multline*}
for some $C>0$ (independent of $k$), contradicting \eqref{e:rH_contra}.

In case (b) or (c) hold the construction is similar. 
Denote by $K_k$ the minimal connection relative to $\overline{S_{u_k}}\cap\partial B_r$.
Then $K_k$ splits $\overline{B_r}$ into two (case (b$_1$)) or three
(case (c$_1$)) regions denoted by $B^i_r$.
Let $\gamma^i$ be the arc of $\partial B_r$ contained in the boundary 
of $B^i_r$. Having all the arcs length uniformly bounded from below, 
it is easy to check that for all $i$ we can find a function $w_k^i\in W^{1,2}(B_r)$ with 
boundary trace $g_k$ and satisfying for some absolute constant $C>0$
\begin{equation}\label{e:wij}
\int_{B_r}|\nabla w_k^i|^2\,dx\leq C
\left(\int_{\gamma^i}|g_k^\prime|^q\,d\HH^1\right)^{\sfrac 2q}
\end{equation}
(cf. \cite[Lemma~7]{DF1}). Denote by $w_k$ the function equal to $w^i_k$ on $B^i_k$, 
then $w_k\in SBV(B_r)$ and $S_{w_k}\subseteq K_k$. The minimality of $u_k$ implies then that
\begin{multline*}
\int_{B_{\sfrac{\rho_0}9}} |\nabla u_k|^2 \leq \int_{B_r}|\nabla u_k|^2\leq\int_{B_r}|\nabla w_k|^2+
\HH^1 (K_k) - \HH^1 (S_{u_k}\cap B_r)\\
\leq\int_{B_r}|\nabla w_k|^2\stackrel{\eqref{e:wij}}{\leq} 
C\left(\int_{\partial B_r\setminus S_{u_k}}|g_k^\prime|^q\,d\HH^1\right)^{\sfrac 2q}
\leq C\left(\frac {12}{\rho_0}\int_{B_1}|\nabla u_k|^q\,dx\right)^{\sfrac 2q}, 
\end{multline*}
contradicting \eqref{e:rH_contra}.
\end{proof}

\section{Higher integrability of the gradient in any dimension: Porosity of the Jump set}\label{s:porosity}

A central role in establishing the higher integrability of the gradient in any dimension shall
be played by the following improvement of Theorem~\ref{t:AFP97}.

\begin{theorem}[Rigot \cite{Rig00}, Maddalena and Solimini \cite{MadSol01}]\label{t:RMS}
There are dimensional constants $\eps(n),\, C_0(n)>0$ such that 
for every $\eps\in(0,\eps(n))$ there exists $\alpha_\eps\in(0,\sfrac 12)$ such that if $u\in\MM(B_2)$ 
and $B_\rho(x)\subset\Om$, with $x\in \SSu$ and $\rho\in(0,1)$, then there exists a ball 
$B_r(y)\subset B_\rho(x)$ with radius $r\in(\alpha_\eps\rho,\rho)$ such that 
\begin{itemize}
 \item[(i)] $\mathscr{D}_u(y,r)+\mathscr{A}_u(y,r)<\eps_0$, $\eps_0>0$ the constant in Theorem~\ref{t:AFP97};
 \item[(ii)] $\SSu\cap B_r(y)$ is a $C^{1,\gamma}$ graph, for all $\gamma\in(0,1)$, containing $y$;
 \item[(iii)] 
\begin{equation}\label{e:neureg}
 r\|\nabla u\|^2_{L^\infty(B_r(y))}\leq C_0\,\eps.
\end{equation} 
\end{itemize}
\end{theorem}

We can restate the result above by saying that $\Sigma_u$ is $(\alpha_\eps,1)$-\emph{porous} in 
$\SSu$ according to the following definition.
\begin{definition}
Given a metric space $(X,d_X)$, a subset $K$ is $(\alpha,\delta)$-\emph{porous} in $X$, with 
 $\alpha\in(0,\sfrac12)$ and $\delta>0$, if for every $x\in X$ and $\rho\in(0,\delta)$ we can 
 find $y\in B_\rho(x)$ and $r\in(\alpha\rho,\rho)$ such that 
 \[
 B_r(y)\subset B_\rho(x)\setminus K.          
 \]
\end{definition}
Clearly, in our case $X=\overline{S_u}$, $K=\Sigma_u$ and $d_X$ is the metric induced 
by the Euclidean one. 
The Hausdorff dimension estimate in the papers by David \cite{Dav96}, Rigot \cite{Rig00} and 
Maddalena and Solimini \cite{MadSol01} follows from the porosity property in Theorem~\ref{t:RMS} 
and Theorem~\ref{t:porosity} below. 

To this aim we recall that, given an Alfhors regular metric space $(X,d_X)$ of dimension $\ell$, 
i.e.~$(X,d_X)$ is complete and there is $\Lambda>0$ such that
\[
 \Lambda^{-1} r^\ell\leq \HH^\ell(B_r(z))\leq \Lambda r^\ell
\quad\text{for all $z\in X$ and $r>0$},
 \]
the \emph{lower/upper Minkowski dimension} of $K$ is defined as 
\[
\underline{\mathrm{dim}}_{\mathcal{M}}K:=\inf\{s\in(0,\ell]:\,\mathcal{M}_\ast^s(K)=0\},\quad
\overline{\mathrm{dim}}_{\mathcal{M}}K:=\inf\{s\in(0,\ell]:\,\mathcal{M}^{\ast\, s}(K)=0\},
\]
where the \emph{lower/upper Minkowski content} is given by
\[
\mathcal{M}_\ast^s(K):=\liminf_{r\downarrow 0}\frac{\HH^\ell\big((K)_r\big)}{r^{\ell-s}},
\qquad
\mathcal{M}_{\ast\,s}(K):=\limsup_{r\downarrow 0}\frac{\HH^\ell\big((K)_r\big)}{r^{\ell-s}}.
\]
and
\begin{equation}\label{e:rinvolucro}
(K)_r:=\{x\in X:\,d_X(x,K)<r\}.
 \end{equation}
 Let us first relate the $2$-dimensions introduced above.
\begin{lemma}\label{l:HMdim}
For all sets $K\subset X$ 
\[
\dimh K\leq\mathrm{dim}_\MM K.
\]
\end{lemma}
\begin{proof}
We may assume $\mathrm{dim}_\MM K<\ell$ since otherwise the inequality is trivial.

 Let $N(K,r)$ be the minimal number of balls of radius $r$ covering $K$, and $P(K,r)$ be the 
 maximal number of disjoint balls with centers belonging to $K$, then
\[
 N(K,2r)\leq P(K,r)\footnote{One can also prove that $P(K,r)\leq N(K,\sfrac r2)$}.
\]
Indeed, set $N:=N(K,2r)$, $P:=P(K,r)$ and let $\mathscr{B}=\{B_r(x_i)\}_{i=1}^P$ be 
the corresponding maximal family of disjoint balls with $x_i\in K$. 
If there exists $x\in K\setminus\cup_{i=1}^P B_{2r}(x_i)$, then $\{B_r(x)\}\cup\mathscr{B}$ 
is a disjoint family of balls with centers on $K$, a contradiction. 

Therefore, for all $s\in(\mathrm{dim}_\MM K,\ell)$
\[
N\omega_\ell r^\ell\leq P\omega_\ell r^\ell\leq \HH^\ell\big((K)_r\big)\Longrightarrow 
N\omega_\ell r^s\leq\frac{\HH^\ell\big((K)_r\big)}{r^{\ell-s}},
\]
from which one easily conclude that 
$2^{-s}\frac{\omega_\ell}{\omega_s}\HH^s(K)\leq\MM^s_\ast(K)=0$.
\end{proof}
 We are now ready to state an estimate on the Hausdorff dimension of porous sets.
 \begin{theorem}[David and Semmes~\cite{DavSem97}]\label{t:porosity}
 If $(X,d_X)$ is Alfhors regular of dimension $\ell$, 
%
then every $(\alpha,\delta)$-porous subset $K$ of $X$ of diameter $d$ satisfies 
 \[
  \HH^\ell\big((K)_r\big)\leq C\, r^{\ell-\eta}\quad \forall r\in(0,d),
 \]
 for some constant $C=C(\ell,\delta,d)>0$ and $\eta=\eta(\alpha,d,\Lambda)\in[0,\ell)$.
Hence,
 \[
 \dimh K\leq\mathrm{dim}_{\mathcal{M}}K \leq \eta.
 \]
\end{theorem}
The higher integrability property of the gradient for MS-minimizers will be (essentially) a consequence 
of the result above. Actually, we cannot take advantage directly of Theorem~\ref{t:porosity} since we 
are not able to relate $(\Sigma_u)_r$ and $\{|\nabla u|^2\geq r^{-1}\}$.
Therefore, we shall prove a suitable version of Theorem~\ref{t:porosity} and establish its links with 
the higher integrability property in the Section~\ref{s:hi} following the approach by De Philippis and 
Figalli \cite{DePFig14}.

In passing, we mention that recently porosity has been employed in several instances to estimate the Hausdorff 
dimension of singular sets of solutions to variational problems (cf. \cite{KKPS00}, \cite{KrMin07}, 
\cite{DePFig14bis}).

In the rest of the present section we shall comment on porosity in the more standard Euclidean setting, 
i.e.~$X=\R^n$, and prove analogous results to those of interest for us. This is done to get more acquainted with
porosity and it is intended as a warm up to the proof of Theorem~\ref{t:hi} by De Philippis and Figalli.

Let then $K\subseteq\R^n$ be a $(\alpha,\delta)$-porous set. Few remarks are in order:
\begin{itemize}
 \item[(i)] By Lebesgue's differentiation theorem clearly $\LL^n(K)=0$;
 \item[(ii)] $K$ is nowhere dense, i.e.~$\text{int}\overline{K}=\emptyset$, since the latter is equivalent to: 
 for every $x\in X$ and $\rho>0$ there exists $y\in X$ and $r>0$ such that $B_r(y)\subset B_\rho(x)\setminus K$.
 \item[(iii)] Zaj{\'{\i}}{\v{c}}ek \cite{Zaj87} actually proved that there are non-porous sets which are nowhere 
 dense and with zero Lebesgue measure.
\end{itemize}
An elementary covering argument actually provides an estimate on the Hausdorff dimension of $K$ and therefore
improves item (i) above.
\begin{proposition}[Salli~\cite{Sal91}]\label{p:salli}
 Suppose that $K$ is a bounded $(\alpha,\delta)$-porous set in $\R^n$ with 
$\diam K\leq d$, then 
\begin{equation}\label{e:salli}
\LL^n\big((K)_r\big)\leq C\, r^{n-\gamma}\quad\quad
\forall r\in(0,d).
\end{equation}
for some constants $C=C(n,\delta,d)>0$ and $\gamma=\gamma(\alpha,n)<n$. In particular,
 \begin{equation}\label{e:salli2}
\dimh K\leq\mathrm{dim}_\MM K\leq\gamma.
 \end{equation} 
 \end{proposition}
\begin{proof}
The building step of the argument goes as follow: consider a cube $Q$ of $\diam Q<\delta$ and center 
$x_Q$, then by taking into account the $(\alpha,\delta)$-porosity of $K$, we find a point $y_Q$ such 
that $B_{\alpha\,\diam Q/2\sqrt{n}}(y_Q)\subset B_{\diam Q/2\sqrt{n}}(x_Q)\setminus K$.
If $\{Q_i\}_i$ is a covering of $Q$ of $k^n$ sub cubes with $\diam Q_i=\diam Q/k$, $k\in\N$, then at 
least one of those cubes does not intersect $K$ if $k$ is sufficiently big. Indeed, it suffices to impose 
\[
 \alpha\,\diam Q>2\frac{\diam Q}{k}\Longleftrightarrow \alpha\,k>2.
\]
Hence, we may choose $k=k(\alpha)$ for which the previous condition is satisfied.

Therefore, given a covering of $K$ of $m=m(\delta,d)$ cubes with diameter 
$\sfrac\delta 2$, we can construct another covering made of $m(k^n-1)=mk^\gamma$ 
cubes of diameter $\sfrac\delta{2k}$, where $\gamma=\gamma(\alpha,n)\in(0,n)$ is such 
that $k^\gamma=k^n-1$.

Clearly, we can iterate this procedure in each of the new cubes, so that for all $N\in\N$ we may find 
a covering of $K$ made of $mk^{N\,\gamma}$ cubes $\{Q_i^k(x_i^k)\}_{i=1}^{mk^{N\,\gamma}}$ of diameter 
$\sfrac\delta{2k^N}$. In particular, each ball $B_{\sfrac\delta{4k^N}}(x_i^k)$ contains $Q_i^k$, and 
their union covers $K$. Thus,  
\[
(K)_{\sfrac\delta{4k^N}}\subset\cup_{i=1}^{mk^{N\,\gamma}}B_{\sfrac{\delta}{2k^N}}(x_i^k).
\]
Hence  
\[
 \LL^n\left((K)_{\sfrac\delta{4k^N}}
 \right)
 \leq \omega_nmk^{N\,\gamma}\left(\frac\delta{2k^N}\right)^n=o(1)\quad N\uparrow\infty.
\]
Estimate \eqref{e:salli} follows at once by a simple dyadic argument on the radii, i.e.~given $r>0$
choosing $k$ such that $r\in[\frac\delta{2k^{N+1}},\frac\delta{2k^N})$.

Instead, estimate \eqref{e:salli2} is an easy consequence of Lemma~\ref{l:HMdim}.
\end{proof}
\begin{remark}
Theorem~\ref{t:porosity} can be proved exactly as Proposition~\ref{p:salli} once 
the existence of a family of dyadic cubes in $(X,d_X)$ has been established 
(cf. \cite[Lemma 5.8]{DavSem97}). By this, we mean a collection $\{\triangle_j\}_{\Z\ni j<j_0}$
of families of measurable subsets of $X$, $j_0=\infty$ if $\diam X=\infty$ and otherwise 
$j_0\in\Z$ such that $2^{j_0}\leq\diam X<2^{j_0+1}$, having the following properties:
\begin{itemize}
 \item[(i)] each $\triangle_j$ is a partition of $X$, i.e.~$X=\cup_{Q\in\triangle_j}Q$ for any $j$ as above;
 \item[(ii)] $Q\cap Q^\prime=\emptyset$ whenever $Q$, $Q^\prime\in \triangle_j$ and $Q\neq Q^\prime$;
 \item[(iii)] if $Q\in\triangle_j$ and $Q^\prime\in \triangle_k$ for $k\geq j$, then either $Q\subseteq Q^\prime$
 or $Q\cap Q^\prime=\emptyset$;
 \item[(iv)] $\lambda^{-1}\,2^{j}\leq\diam Q\leq\lambda\, 2^j$ and 
 $\lambda^{-1}2^{j\ell}\leq\HH^\ell(Q)\leq\lambda 2^{j\ell}$ for all $j$ and all $Q\in\triangle_j$;
 \item[(v)] for all $j$ and all $Q\in\triangle_j$, and $\tau>0$
 \[
  \HH^\ell\left(\{x\in Q:\,\mathrm{dist}(x,X\setminus Q)\leq\tau 2^j\}\right)
 +\HH^\ell\left(\{x\in X\setminus Q:\,\mathrm{dist}(x,Q)\leq\tau 2^j\}\right)\leq \lambda\tau^{1/\lambda}\HH^\ell(Q).
 \]
\end{itemize}
with $\lambda=\lambda(\ell,\Lambda)$ (for the existence of such families see \cite{Dav91} and \cite{Sem00}).
\end{remark}
With fixed a given porosity $\alpha\in(0,\sfrac12)$, we are then interested in analyzing the worst case, i.e.
\[
 D(\alpha,n):=\sup\left\{\dimh K:\,K\subset\R^n\text{ is $(\alpha,\delta)$-porous for some $\delta>0$}\right\}.
\]
We can easily deduce the estimate 
\[
n-1\leq D(\alpha,n)\leq\gamma(\alpha,n)<n,
\]
as $(n-1)$-dimensional vector spaces are $(\alpha,\delta)$-porous for all $\delta>0$ and $\alpha\in(0,\sfrac12)$. 
Furthermore, Mattila (see, for instance, \cite[Theorem~11.14]{Matt95}) has shown that 
\begin{equation}\label{e:mattila}
 \lim_{\alpha\uparrow \sfrac12}D(\alpha,n)=n-1.
\end{equation}
\begin{remark}
Salli \cite{Sal91} has actually improved upon the previous result by showing that 
\[
 n-1+\frac{B(n)}{|\ln(1-2\alpha)|}\leq D(\alpha,n)\leq  n-1+\frac{A(n)}{|\ln(1-2\alpha)|}
 \quad\quad\text{for all $\alpha\in(0,\sfrac12)$},
\]
for some strictly positive dimensional constants $A$ and $B$.
\end{remark}

Finally, we note that the analogous property in \eqref{e:mattila} in the case of interest for us, 
if true, would then let us conclude another characterization of the conjectured estimate on the
Hausdorff dimension of $\Sigma_u$:
\emph{If $\Sigma_u$ is $(\alpha,\delta)$-porous in $\overline{S_u}$ for all $\alpha\in(0,\sfrac12)$ 
and some $\delta=\delta(\alpha)>0$, then $\dimh\Sigma_u\leq n-2$.}

\section{Higher integrability of the gradient in any dimension: the proof}\label{s:hi}

In this section we shall prove the higher integrability property 
of the gradient following De Philippis and Figalli \cite{DePFig14}.
We shall first establish in Proposition~\ref{p:alfpor} below a particular 
case of Theorem~\ref{t:porosity} that is sufficient for our purposes. 

To this aim we recall that the conclusions of Theorem~\ref{t:CL} and Proposition~\ref{p:dub} 
show the Alfhors regularity of $\Om\cap\SSu$: for some constants $C_0=C_0(n)>0$, $\rho_0=\rho_0(n)>0$
\begin{equation}\label{e:SuAR}
C_0^{-1}\,r^{n-1} \leq \HH^{n-1}(\SSu\cap B_r(z)) \leq C_0\,r^{n-1}
\end{equation}
for all 
$z\in \SSu$, and all $r\in(0,\rho_0\wedge \mathrm{dist}(z,\partial\Omega))$, $u\in\MM(B_2)$.

To prove Proposition~\ref{p:alfpor} we need two technical lemmas. The first one is obtained via
De Giorgi's slicing/averaging principle.
\begin{lemma}\label{l:technical}
 There are dimensional constants $M_1$, $C_1$ such that if $M\geq M_1$ 
 for every $u\in \MM(B_2)$ we can find three decreasing sequences of radii such that 
 \begin{itemize}
  \item[(i)] $1\geq R_h\geq S_h\geq T_h\geq R_{h+1}$;
  \item[(ii)] $8M^{-(h+1)}\leq R_h-R_{h+1}\leq M^{-(h+1)/2}$, and $S_h-T_h=T_h-R_{h+1}=4M^{-(h+1)}$;
  \item[(iii)] $\HH^{n-1}\big(\SSu\cap(\Bb_{S_h}\setminus\Bb_{R_{h+1}})\big)\leq C_1M^{-(h+1)/2}$;
  \item[(iv)] $R_\infty=S_\infty=T_\infty\geq 1/2$.
  \end{itemize}
\end{lemma}
\begin{proof}
Let $R_1=1$, given $R_h$ we construct $S_h$, $T_h$ and $R_{h+1}$ as follows.

Set $N_h:=\lfloor M^{(h+1)/2}/8\rfloor\in\N$ and fix $M_1\in\N$ such that 
$N_h\geq \lfloor M^{(h+1)/2}/16\rfloor$ for $M\geq M_1$. Here, 
$\lfloor\alpha\rfloor$ denotes the integer part of $\alpha\in\R$. 

The annulus $B_{R_h}\setminus \Bb_{R_h-8M^{-(h+1)/2}}$ contains the $N_h$ disjoint sub annuli 
$\Bb_{R_h-8(i-1)M^{-(h+1)}}\setminus \Bb_{R_h-8iM^{-(h+1)}}$, $i\in\{1,\ldots,N_h\}$, of equal 
width $8\,M^{-(h+1)}$. By averaging we can find an index $i_h\in\{1,\ldots,N_h\}$ such that
\begin{multline*}
\HH^{n-1}\big(K\cap(\Bb_{R_h-8(i_h-1)M^{-(h+1)/2}}\setminus\Bb_{R_h-8i_h\,M^{-(h+1)/2}})\big)\\
\leq \frac 1{N_h}\,\HH^{n-1}\big(K\cap (\Bb_{R_h}\setminus \Bb_{R_h-8M^{-(h+1)/2}})\big)
\stackrel{\text{d.u.b. in \eqref{e:SuAR}}}{\leq} 
C_0\,\frac{R_h^{n-1}}{N_h}\leq C_1\,M^{-(h+1)/2},
\end{multline*}
so that (iii) is established. Finally, set 
\[
 S_h:=R_h-8(i_h-1)M^{-(h+1)},\,\,R_{h+1}:=R_h-8i_h\,M^{-(h+1)},\,\,T_h:=\frac12(S_h+R_{h+1}),
\]
then items (i) and (ii) follow by the very definition, and item (iv) from (ii) if 
$M_1$ is sufficiently big.
\end{proof}
The second lemma has a geometric flavor.
\begin{lemma}\label{l:grafico}
Let $f:\R^{n-1}\to\R$ be Lipschitz with
\begin{equation}\label{e:lip}
f(0)=0,\quad\|\nabla f\|_{L^\infty}\leq\eta.
\end{equation}
If $G:=\mathrm{graph}(f)\cap B_{2}$ and $\eta\in(0,\sfrac{1}{15}]$, then 
for all $\delta\in(0,\sfrac12)$ and $x\in (\overline{B}_{1+\delta}\setminus B_{1})\cap G$
\[
\mathrm{dist}(x,(\overline{B}_{1+2\delta}\setminus B_{1+\delta})\cap G)\leq \frac 32 \delta.
\]
\end{lemma}
\begin{proof}
Clearly by \eqref{e:lip} we get
\[
 \|f\|_{W^{1,\infty}(B_2)}\leq 3\eta.
\]
Let $x=(y,f(y))\in (\overline{B}_{1+\delta}\setminus B_{1})\cap G$ and $\hat{x}:=(\lambda\,y,f(\lambda\,y))$,
with $\lambda$ to be chosen suitably. Note that as $|x|\geq 1$ we have
\begin{multline*}
 |f(\lambda\,y)-\lambda\,f(y)|\leq |f(\lambda\,y)-f(y)|+|\lambda-1|\,|f(y)|\\
 \leq |\lambda-1|\,\big(\|\nabla f\|_{L^\infty}\,|y|+\|f\|_{L^\infty}\big)\leq 3\eta|\lambda-1|\,|x|.
\end{multline*}
Hence,
\[
 |\hat{x}-x|\leq |\hat{x}-\lambda\,x|+|\lambda-1||x| \leq\big(3\eta+1\big)|\lambda-1|\,|x|.
\]
It is easy to check that the choice $\lambda=1+\frac{5}{4}\delta|x|^{-1}$ gives the conclusion.
\end{proof}

We are now ready to prove the version of Theorem~\ref{t:porosity} of interest for our purposes.
\begin{proposition}[De Philippis and Figalli \cite{DePFig14}]\label{p:alfpor}
Let $C_0,C_1,M_1$ be the constants in \eqref{e:SuAR} and Lemma \ref{l:technical}, 
respectively.

There exist dimensional constants $C_2,\, M_2>0$ and $\alpha\in(0,\sfrac 14)$, 
$\beta\in(0,\sfrac 14)$, with $M_2\geq M_1$, such that for every $M\geq M_2$, 
$u\in\MM(B_2)$, we can find families $\FF_j$ of disjoint balls
 \[
 \FF_j=\left\{B_{\alpha M^{-j}}(y_i):\,y_i\in \SSu,\,1\leq j\leq N_j\right\}
 \]
such that for all $h\in\N$ 
\begin{itemize}
 \item[(i)]  $B$, $B^\prime\in\cup_{j=1}^h\FF_j$ are distinct balls, then 
 $(B)_{4M^{-(h+1)}}\cap(B^\prime)_{4M^{-(h+1)}}=\emptyset$;
 \item[(ii)] if $B_{\alpha M^{-j}}(y_i)\subset\FF_j$, then 
 $\SSu\cap B_{2\alpha M^{-j}}(y_i)$ is a $C^{1,\gamma}$ graph, $\gamma\in(0,1)$ any, 
 containing $y_i$,
 \[
\mathscr{D}_u\big(y_i,2\alpha M^{-j}\big)+\mathscr{A}_u\big(y_i,2\alpha M^{-j}\big)<\eps_0;
 \]
\begin{equation}\label{e:gradbdd}
\|\nabla u\|_{L^\infty\big(B_{2\alpha M^{-j}}(y_i)\big)}<M^{j+1}; 
\end{equation}
 \item[(iii)] let $\{R_h\}$, $\{S_h\}$, $\{T_h\}$ be the sequences of radii in 
 Lemma~\ref{l:technical}, and let
 \[
  K_h:=(\SSu\cap\overline{B}_{S_h})\setminus\left(\cup_{j=1}^h\cup_{\FF_j}B\right),
 \]
 (note that by construction $K_{h+1}\subset K_h\setminus \cup_{\FF_{h+1}}B$), and 
 \[
 \widetilde{K_h}:= (\SSu\cap\overline{B}_{T_h})\setminus
\left(\cup_{j=1}^h\cup_{\FF_j}(B)_{2M^{-(h+1)}}\right)\subset K_h.
 \]
 Then, there exists a finite set of points $\CC_h:=\{x_i\}_{i\in I_h}\subset\widetilde{K_h}$ such that
 \begin{equation}\label{e:centres}
  |x_j-x_k|\geq 3M^{-(h+1)}\quad\forall j,k\in I_h,\,j\neq k;
 \end{equation}
 \begin{equation}\label{e:ingro}
  (K_h\cap \Bb_{R_{h+1}})_{M^{-(h+1)}}\subset 
	\cup_{i\in I_h}B_{8M^{-(h+1)}}(x_i);
 \end{equation}
 \begin{equation}\label{e:stimaKh}
  \HH^{n-1}(K_h)\leq C_1\,h\,M^{-2h\beta};
 \end{equation}
 \begin{equation}\label{e:measureingro}
  \LL^n\left((K_h\cap \Bb_{R_{h+1}})_{M^{-(h+1)}}\right)\leq C_2\,h\,M^{-h(1+2\beta)-1}.
 \end{equation}
 
 \item[(iv)] $\Sigma_u\cap B_{\sfrac12}\subset K_h$ for all $h\in\N$ and 
 \begin{equation}\label{e:sigmaur}
 \LL^n\big((\Sigma_u\cap B_{\sfrac12})_{r}\big)\leq C_2\,r^{1+\beta}\qquad\forall r\in(0,\sfrac 12].
 \end{equation}
In particular, $\dim_{\MM}(\Sigma_u\cap B_{\sfrac12})\leq n-1-\beta$.
 \end{itemize}
\end{proposition}
\begin{proof}
For notational convenience we set $K=\SSu$. In what follows we shall repeatedly use Theorem~\ref{t:RMS} 
with $\eps\in(0,1)$ fixed and sufficiently small. 

We split the proof in several steps.
\smallskip

\noindent{\bf Step 1.} \emph{Inductive definition of the families $\FF_j$.}

For $h=1$ we define 
 \[
\FF_1:=\emptyset,\,\, K_1=K\cap\Bb_{S_1},\,\,\widetilde{K_1}=K\cap\Bb_{T_1}, 
\] 
and choose $\CC_1$ to be a maximal family of points at distance $3M^{-2}$ from each other. 
Of course, properties (i) and (ii) and \eqref{e:centres} are satisfied. To check the others, 
one can argue as in the verification below.
 
 Suppose that we have built the families $\{\FF_j\}_{j=1}^h$ as in the statement, to construct 
 $\FF_{h+1}$ we argue as follows. 
 Let $\CC_h=\{x_i\}_{i\in I_h}\subset\widetilde{K_h}$ be a family of points satisfying \eqref{e:centres}, 
 i.e.~$|x_i-x_k|\geq 3M^{-(h+1)}$ for all $j,\,k\in I_h$ with $j\neq k$, and consider 
 \[
\GG_{h+1}:=\{B_{M^{-(h+1)}}(x_i)\}_{i\in I_h}.
 \]
 By the porosity assumption on $K$ for every ball $B_{M^{-(h+1)}}(x_i)\in\GG_{h+1}$ 
 we can find a sub-ball $B_{2\alpha\,M^{-(h+1)}}(y_i)\subset B_{M^{-(h+1)}}(x_i)\setminus K$,
 $\alpha\in(0,\sfrac 14)$ for which the theses of Theorem~\ref{t:RMS} are satisfied. Then, define
 \[
 \FF_{h+1}:=\{B_{\alpha\,M^{-(h+1)}}(y_i)\}_{i\in I_h}.
 \]
 By condition \eqref{e:centres}, the balls $B_{\frac32 M^{-(h+1)}}(x_i)$ are disjoint and
 do not intersect
 \[
  \cup_{j=1}^h\cup_{\FF_j}(B)_{\frac12M^{-(h+1)}}
 \]
 by the very definition of $\widetilde{K_h}$. Thus, item (i) and (ii) are satisfied. 
 
 Hence, we can define $K_{h+1}$, $\widetilde{K}_{h+1}$ and $\CC_{h+1}$ as in the statement.
\smallskip

 \noindent{{\bf Step 2.}} \emph{Proof of \eqref{e:ingro}.}

Let $x\in (K_h\cap \Bb_{R_{h+1}})_{M^{-(h+2)}}$ and let $z$ be a point of minimal distance from
$K_h\cap \Bb_{R_{h+1}}$. 
In case $z\in \widetilde{K}_{h+1}$, by maximality there is $x_i\in \CC_{h+1}$ such that 
$|z-x_i|\leq 3M^{-(h+2)}$ and thus we conclude $x\in B_{5M^{-(h+2)}}(x_i)$.
Instead, if $z\in(K_h\cap \Bb_{R_{h+1}})\setminus\widetilde{K}_{h+1}$, the definitions of 
${K_{h+1}}$ and $\widetilde{K}_{h+1}$ yield the existence of a ball $\widetilde{B}\in\cup_{j=1}^{h+1}\FF_j$ 
for which $z\in (K\cap(\widetilde{B})_{2M^{-(h+2)}})\setminus \widetilde{B}$. In view of property (ii), 
a rescaled version of Lemma \ref{l:grafico} gives a point $y$ satisfying
\[
y\in \big(K\cap(\widetilde{B})_{4M^{-(h+2)}}\big)\setminus (\widetilde{B})_{2M^{-(h+2)}},\quad
\text{and}\quad |z-y|\leq 3M^{-(h+2)}.
\]
Therefore, as $z\in \overline{B}_{R_{h+2}}$ and $T_{{R_{h+1}}}=R_{h+2}+4M^{-(h+2)}$ we get
by property (i) and the definition of $\widetilde{K}_{h+1}$ 
\[
y\in \big(K\cap(\widetilde{B})_{4M^{-(h+2)}}\cap\widetilde{B})_{4M^{-(h+2)}}\big)
\setminus (\widetilde{B})_{2M^{-(h+2)}}.
\]
Finally, by maximality we may find $x_i\in \CC_{h+1}$ such that $|y-x_i|\leq 3M^{-(h+2)}$. In conclusion,
we have
\[
|x-x_i|\leq|x-z|+|z-y|+|y-x_i|\leq 7M^{-(h+2)},
\]
so that  \eqref{e:ingro} follows at once.
\smallskip

\noindent{{\bf Step 3.}} \emph{the $K_h$'s satisfy a suitable d.l.b. as that of $K$ in \eqref{e:SuAR}.}

We claim that for every $h\in\N$
\begin{equation}\label{e:dlbKh}
 K_{h}\cap B_{M^{-(h+1)}}(x_i)=K\cap B_{M^{-(h+1)}}(x_i)\qquad\text{for all $x_i\in \CC_{h}$}.
\end{equation}
In particular, from the latter we infer the conclusion of this step. 
 
The equality above is proven by contradiction: assume we can find $x_i\in \CC_{h}$ and 
 \[
x\in (K\setminus K_{h})\cap B_{M^{-(h+1)}}(x_i). 
 \]
As $x_i\in\widetilde{K}_{h}$ then $x_i\in B_{T_{h}}$, in turn implying 
$x\in B_{S_{h}}$ since $S_h-T_h=4M^{-(h+1)}$.
Therefore $x\in\big(K\setminus K_{h}\big)\cap\Bb_{S_{h}}$, and by definition of 
$K_{h}$ we can find a ball $B\in \FF_{j}$, $j\leq h$, such that $x\in B$. 
We conclude that
\[
\textrm{dist}(x_i,B)\leq|x-x_i|\leq M^{-(h+1)},
\] 
contradicting that $x_i\in\widetilde{K}_{h}$.
\smallskip

\noindent{{\bf Step 4.}} \emph{Proof of \eqref{e:stimaKh}.}

We get first a lower bound for $\#(I_h)$: use \eqref{e:ingro} and the d.u.b in \eqref{e:SuAR} to get
 \[
  \HH^{n-1}\big(K_h\cap \Bb_{R_{h+1}}\big)=
  \HH^{n-1}\big(K_h\cap \Bb_{R_{h+1}}\cap\cup_{i\in I_h}B_{8M^{-(h+1)}}(x_i)\big)\leq
 C_0\,\#(I_h)\big(8M^{-(h+1)}\big)^{n-1}
 \]
that is
 \begin{equation}\label{e:Ihbound}
 \#(I_h)M^{-(h+1)(n-1)}\geq 8^{1-n}C_0^{-1}  \HH^{n-1}\big(K_h\cap \Bb_{R_{h+1}}\big).
 \end{equation}
 Thus, we estimate as follows
 \begin{multline}\label{e:measure}
  \HH^{n-1}(K_{h+1})\leq\HH^{n-1}\left(K_h\setminus\cup_{\FF_{h+1}}B\right)
  \stackrel{\text{disjoint balls}}{=}\HH^{n-1}(K_h)-\sum_{\FF_{h+1}}\HH^{n-1}(K_h\cap B)\\
  \stackrel{\text{d.l.b. in \eqref{e:SuAR}, \eqref{e:dlbKh}}}{\leq}
	\HH^{n-1}(K_h)-\frac{\alpha^{n-1}}{C_0}\#(I_h) M^{-(h+1)(n-1)}
  \stackrel{\eqref{e:Ihbound}}{\leq}\HH^{n-1}(K_h)-\frac{8^{1-n}\alpha^{n-1}}{C_0^2}\HH^{n-1}\big(K_h\cap \Bb_{R_{h+1}}\big)\\
  =(1-\eta)\HH^{n-1}(K_h)+\eta\left(\HH^{n-1}(K_h)-\HH^{n-1}\big(K_h\cap \Bb_{R_{h+1}}\big)\right)\\
  \stackrel{\text{def. of $K_h$}}{\leq}(1-\eta)\HH^{n-1}(K_h)+\eta\,\HH^{n-1}\big(K\cap(\Bb_{S_h} \setminus \Bb_{R_{h+1}})\big)
  \stackrel{\text{(iii) Lemma~\ref{l:technical}}}{\leq}(1-\eta)\HH^{n-1}(K_h)+C_1M^{-\frac{h+1}2},
 \end{multline}
where we have set $\eta:=\sfrac{8^{1-n}\alpha^{n-1}}{C_0^2}$.

By iteration of \eqref{e:measure}, we find by Young inequality
\[
 \HH^{n-1}(K_h)\leq C_1\sum_{i=0}^h(1-\eta)^{h-i}M^{-\sfrac i2}
 \leq C_1\,h\,\max\{(1-\eta)^h,M^{-\sfrac h2}\}.
\]
Choose $\beta\in(0,\sfrac 14)$ such that $(1-\eta)\leq M^{-2\beta}$, 
the previous estimate then yields \eqref{e:stimaKh}, 
\begin{equation*}
 \HH^{n-1}(K_h)\leq C_1\,h\,\max\{M^{-2h\beta},M^{-\sfrac h2}\}=C_1\,h\,M^{-2h\beta}.
\end{equation*}

\smallskip

 \noindent{{\bf Step 5.}} \emph{Proof of \eqref{e:measureingro}.}

Then, we exploit \eqref{e:ingro} to get
 \begin{multline}\label{e:rozza}
  \LL^n\left((K_{h+1}\cap \Bb_{R_{h+2}})_{M^{-(h+2)}}\right)\leq 
  \LL^n\left(\cup_{i\in I_{h+1}}B_{8M^{-(h+2)}}(x_i)\right)\leq 
	\#(I_{h+1})\big(8M^{-(h+2)}\big)^n\\
  \stackrel{\text{d.l.b. in \eqref{e:SuAR}, \eqref{e:dlbKh}}}{\leq} \frac{8^n}{C_0}M^{-(h+2)}\sum_{i\in 
	I_{h+1}}\HH^{n-1}\big(K_{h+1}\cap B_{M^{-(h+2)}}(x_i)\big)\\
  \stackrel{\text{disjoint balls}}{\leq}\frac{8^n}{C_0}M^{-(h+2)}\HH^{n-1}\big(K_{h+1}\big)
  \stackrel{\eqref{e:stimaKh}}{\leq} \frac{8^n\,C_1}{C_0}\,(h+1)\,M^{-2(h+1)\beta-(h+2)}.
  \end{multline}
\smallskip

\noindent{{\bf Step 6.}} \emph{Proof of \eqref{e:sigmaur}}
  
By construction we have that $\Sigma_u\cap B_{\sfrac 12}\subseteq K_h$. 
Therefore, 
\eqref{e:measureingro} gives as $R_h\geq R_\infty\geq\sfrac 12$
\begin{equation*}
 \LL^n\big((\Sigma_u\cap B_{\sfrac12})_{M^{-(h+1)}}\big)\leq
 \LL^n\left((K_{h}\cap \Bb_{R_{h+1}})_{M^{-(h+1)}}\right)
 \leq C_2\,h\,M^{-h(1+2\beta)-1}. 
\end{equation*}
Hence, if $r\in (M^{-(h+2)},M^{-(h+1)}]$ we get
\[
 \LL^n\big((\Sigma_u\cap B_{\sfrac12})_{r}\big)\leq 
 C_2\,h\,M^{-h(1+2\beta)-1}\leq  C_2\,M^{-h(1+\beta)-1}\leq C_2\,r^{1+\beta}.
\]
\end{proof}
\begin{remark}
Apart from Step~2, all the arguments in the other steps of Proposition~\ref{p:alfpor} employ only 
the Alfhors regularity of $\Om\cap\SSu$ and its consequence Lemma \ref{l:technical}.

In addition, note that one can easily infer the (more) intrinsic estimates 
\begin{equation*}
  \HH^{n-1}\left( \SSu\cap(K_h\cap \Bb_{R_{h+1}})_{M^{-(h+1)}}\right)\leq C_2\,h\,M^{-2h\beta},
 \end{equation*} 
 and
 \begin{equation*}
 \HH^{n-1}\big(\SSu\cap(\Sigma_u\cap B_{\sfrac12})_{r}\big)\leq C_2\,r^\beta\qquad\forall r\in(0,\sfrac 12].
 \end{equation*}
 Indeed, by arguing as in \eqref{e:rozza} we get
 \begin{multline*}
  \HH^{n-1}\left(\SSu\cap(K_{h+1}\cap \Bb_{R_{h+2}})_{M^{-(h+2)}}\right)\leq 
  \HH^{n-1}\left(\SSu\cap\cup_{i\in I_{h+1}}B_{8M^{-(h+2)}}(x_i)\right)\\
  \leq \#(I_{h+1})\big(8M^{-(h+2)}\big)^{n-1}
  \stackrel{\text{d.l.b. in \eqref{e:SuAR}, \eqref{e:dlbKh}}}{\leq} \frac{8^{n-1}}{C_0}\sum_{i\in 
	I_{h+1}}\HH^{n-1}\big(K_{h+1}\cap B_{M^{-(h+2)}}(x_i)\big)\\
  \stackrel{\text{disjoint balls}}{\leq}\frac{8^{n-1}}{C_0}\HH^{n-1}\big(K_{h+1}\big)
  \stackrel{\eqref{e:stimaKh}}{\leq} \frac{8^{n-1}\,C_1}{C_0}\,(h+1)\,M^{-2(h+1)\beta}.
 \end{multline*}
\end{remark}
As outlined in Section~\ref{s:porosity} the former result leads to the higher integrability of 
the gradient for MS-minimizers in any dimension.
\begin{theorem}[De Philippis and Figalli \cite{DePFig14}]\label{t:hiDePF}
There is $p>2$ such that $\nabla u\in L^p_{\mathrm{loc}}(\Om)$ for all $u\in\MM(\Om)$ and for 
all open sets $\Om\subseteq\R^n$.
 \end{theorem}
\begin{proof} 
Clearly, it is sufficient for our purposes to prove a localized estimate. Hence,
for the sake of simplicity we suppose that $\Om=B_2$. 

We keep the notation of Proposition~\ref{p:alfpor} and furthermore denote for all $h\in\N$ 
 \begin{equation}\label{e:Ah}
  A_h:=\left\{x\in B_2\setminus K:\,|\nabla u(x)|^2>M^{h+1}\right\}.
 \end{equation}
We claim that 
\begin{equation}\label{e:claimAh}
A_{h+2}\cap B_{R_{h+2}}\subset (K_h\cap B_{R_{h+1}})_{M^{-(h+1)}}.
\end{equation}
Given this for granted we conclude as follows: we use \eqref{e:measureingro} to deduce that 
\begin{equation}\label{e:estmeas}
\LL^n(A_{h+2}\cap B_{R_{h+2}})\leq\LL^n\big((K_h\cap B_{R_{h+1}})_{M^{-(h+1)}}\big)\leq
C_2\,h\,M^{-h(1+2\beta)-1}.
\end{equation}
Therefore, recalling that $\sfrac 12\leq R_\infty\leq R_h$, in view of \eqref{e:estmeas} 
and Cavalieri's formula for $q>1$ we get that 
\begin{multline*}
 \int_{B_{\sfrac12}}|\nabla u|^{2q}dx=q\int_0^{\infty}t^{q-1}
 \LL^n\big(\left\{x\in B_{\sfrac12}\setminus K:\,|\nabla u(x)|^2>t\right\}\big)dt\\
\leq q\sum_{h\geq 3}\int_{M^h}^{M^{h+1}}t^{q-1}
 \LL^n\big(\left\{x\in B_{\sfrac12}\setminus K:\,|\nabla u(x)|^2>t\right\}\big)dt+M^{3q}\LL^n(B_{\sfrac12})\\
\leq \sum_{h\geq 0}M^{(h+4)q}\LL^n\big(A_{h+2}\cap B_{\sfrac 12}\big)+M^{3q}\LL^n(B_{\sfrac12})
\leq C_2\,\sum_{h\geq 0}h\,M^{(h+4)q-h(1+2\beta)-1}+M^{3q}\LL^n(B_{\sfrac12}).
 \end{multline*}
 The conclusion follows at once by taking $q\in(1,1+2\beta)$ and $p=2q$.
\smallskip

 Let us now prove formula \eqref{e:claimAh} in two steps. 
 
 \noindent{{\bf Step 1.}} \emph{For all $M> n$ and $R\in(0,1]$
 we have that}
 \begin{equation}\label{e:prelimAh}
  A_h\cap B_{R-2M^{-h}}\subset\big(K\cap {B_R}\big)_{M^{-h}}
	\qquad\text{for all $h\in\N$}. 
 \end{equation}
Indeed, for $x\in A_h\cap B_{R-2M^{-h}}$ let  $z\in K$ be such that $\mathrm{dist}(x,K)=|x-z|$.
If $|x-z|>M^{-h}$ then $B_{M^{-h}}(x)\cap K=\emptyset$ so that $u$ is harmonic on $B_{M^{-h}}(x)$.
Therefore, by subharmonicity of $|\nabla u|^2$ on the same set and the d.u.b. in 
\eqref{e:SuAR} we infer that
\[
 M^{h+1}\stackrel{x\in A_h}{\leq}|\nabla u(x)|^2\leq\fint_{B_{M^{-h}}(x)}|\nabla u|^2\leq n\,M^h,
\]
that is clearly impossible for $M>n$. 

Finally, as $x\in B_{R-2M^{-h}}$ and $|x-z|\leq M^{-h}$ we conclude that $z\in B_R$.
\smallskip

\noindent{{\bf Step 2.}} \emph{Proof of \eqref{e:claimAh}.}

Since $R_{h+1}-R_{h+2}\geq 8M^{-(h+2)}$ (cf. (i) Lemma~\ref{l:technical}), we apply Step~1 to $A_{h+2}$ 
and $R=R_{h+1}$ and then \eqref{e:prelimAh} implies that 
\[
A_{h+2}\cap B_{R_{h+2}}\subset\big(K\cap B_{R_{h+1}}\big)_{M^{-(h+1)}}.
\]
Let $x\in A_{h+2}\cap B_{R_{h+2}}$, $z\in K\cap B_{R_{h+1}}$ be a point of 
minimal distance, and suppose that $z\in K\setminus K_h$. 

Since $R_{h+1}\leq S_h$, by the very definition of $K_h$ we find a ball 
$B\in\cup_{j=1}^h\FF_j$ such that $z\in B$. Since $B=B_t(y)$ for some $y$ and with the radius 
$t\geq\alpha\,M^{-h}$, then $x\in B_{2t}(y)$ as $|x-z|\leq M^{-(h+1)}$ for $M$ sufficiently large. 
Thus, estimate $|\nabla u(x)|^2<M^{h+1}$ follows from \eqref{e:gradbdd} in item (ii) of 
Proposition~\ref{p:alfpor}. This is in contradiction with $x\in A_{h+2}$.
\end{proof}

\bibliographystyle{alpha}

\end{document}